\documentclass[11pt]{article}

\usepackage[draft]{say}
\definecolor{cadmiumgreen}{rgb}{0.0, 0.42, 0.24}

\definecolor{electricultramarine}{rgb}{0.25, 0.0, 1.0}

\usepackage[margin=1in,marginparwidth=0.8in, marginparsep=0.1in]{geometry}
\usepackage{amsmath}
\usepackage{amssymb}
\usepackage{amsthm}
\usepackage{amsfonts}
\usepackage{mathtools}
\usepackage[makeroom]{cancel}
\usepackage{mathrsfs}
\usepackage{tikz}
\usepackage{tikz-cd}
\usetikzlibrary{arrows}
\usetikzlibrary{positioning}
\usepackage{stmaryrd}
\usepackage[bottom]{footmisc}
\usepackage{enumitem}
\usepackage{todonotes}
\usepackage[hidelinks]{hyperref}

\usepackage[mathscr]{eucal} 
\usepackage[title]{appendix}
\usepackage[nottoc]{tocbibind}
\usepackage[all,cmtip]{xy}
\usepackage{soul}

\theoremstyle{definition}
\newtheorem{definition}{Definition}[section]
\newtheorem{example}[definition]{Example}

\theoremstyle{remark}

\newtheorem{remark}[definition]{Remark}

\theoremstyle{plain}
\newtheorem{theorem}[definition]{Theorem}
\newtheorem{lemma}[definition]{Lemma}
\newtheorem{proposition}[definition]{Proposition}
\newtheorem{corollary}[definition]{Corollary}

\makeatletter
\newcommand*\bigcdot{ {\mathpalette\bigcdot@{.5}} }
\newcommand*\bigcdot@[2]{\mathbin{\vcenter{\hbox{\scalebox{#2}{$\m@th#1\bullet$}}}}}
\makeatother

\usepackage{stackengine}
\usepackage{scalerel}
\usepackage{xcolor}
\usepackage{graphicx}

\newcommand\FF{ \mathbb{F} }

\newcommand\RR{ \mathbb{R} }

\newcommand\ZZ{ \mathbb{Z} }

\newcommand\cA{ \mathcal{A} }

\newcommand\cC{ \mathcal{C} }

\newcommand\cE{ \mathcal{E} }
\newcommand\cF{ \mathcal{F} }

\newcommand\cH{ \mathcal{H} }

\newcommand\cK{ \mathcal{K} }

\newcommand\cM{ \mathcal{M} }

\newcommand\cS{ \mathcal{S} }

\newcommand\sC{ \mathscr{C} }
\newcommand\sD{ \mathscr{D} }

\newcommand\sT{ \mathscr{T} }

\newcommand\pt{\operatorname{\{\ast\} } }

\newcommand\Rp{ { \mathbb{R}_{>0} } }

\newcommand\Tr{\operatorname{Tr}}
\newcommand{\inv}{{\operatorname{inv}}}

\newcommand\CAlg{ {\operatorname{CAlg}} }
\newcommand\End{\operatorname{End}}
\newcommand\Fun{ {\operatorname{Fun}} }
\newcommand\id{\operatorname{id}}
\newcommand\Id{\operatorname{Id}}
\newcommand\Ind{ {\operatorname{Ind} } }
\newcommand\Hom{\operatorname{Hom}}
\newcommand\inHom{\operatorname{\underline{\Hom}} }

\newcommand\PrLst{  {  \operatorname{Pr}^{\operatorname{L} }_{st} } }

\DeclareMathOperator{\colim}{colim}

\newcommand\Ext{\operatorname{Ext}}

\newcommand\cof{\operatorname{cof}}

\newcommand\dMod{\operatorname{-Mod}}

\newcommand\Mod{\operatorname{Mod}}

\newcommand\Loc{ {\operatorname{Loc}} }
\newcommand\mhom{\operatorname{\mu hom}}

\newcommand\Sh{\operatorname{Sh}}

\newcommand\sHom{\operatorname{\mathscr{H}om}}

\newcommand\supp{ {\operatorname{supp}} }

\newcommand\CZ{ {\operatorname{CZ}} }

\newcommand\HH{ {\operatorname{HH} } }

\newcommand\Tam{{\sT}}

\title{On the Hochschild cohomology of Tamarkin categories}
\author{Christopher Kuo, Vivek Shende, and Bingyu Zhang}
\date{}

\newcommand\PrLV[1][\sT]{  {  \operatorname{Pr}^{\operatorname{L} }_{\sT,st} } }
\newcommand\Perv{\operatorname{\Perv}}

\newcommand\nT{ {\textnormal{T}} }
\newcommand\hatLam{ {\widehat{\Lambda}}}

\begin{document}
 
\maketitle

\begin{abstract}
To any open subset of a cotangent bundle, Tamarkin has associated a certain quotient
of a category of sheaves.   
Here we show that the Hochschild cohomology of this category 
agrees with the
filtered symplectic cohomology of the subset.  
\end{abstract}

\renewcommand\contentsname{\vspace*{-20pt}}
\setcounter{tocdepth}{2}
\small
\renewcommand{\baselinestretch}{1}
\tableofcontents
\renewcommand{\baselinestretch}{1.0}\normalsize

 \section{Introduction}

The symplectic cohomology of a Liouville manifold $W$ is, by definition \cite{Viterbo-functorsI}:
\begin{equation}\label{unfiltered symplectic cohomology intro}
    SH^*(W) := \varinjlim_{f
 \, linear} HF^*(W, f),
\end{equation} 

 Here, each $f: W \to \mathbb{R}$ is a smooth function linear at infinity; 
$HF(W, f)$ is the Hamiltonian Floer cohomology for the Hamiltonian function $f$; the limit is taken over  `continuation' maps
 $HF^*(W, f) \to HF^*(W, g)$, which are defined when $f \leq g$
everywhere.  
There are celebrated comparison theorems with the Hochschild homology  \cite{Abouzaid-generation} and cohomology
\cite{Seidel-deformation} of the corresponding wrapped Fukaya category: 
\begin{equation} \label{symplectic homology as hochschild homology}
    HH_{*-n}(Fuk(W)) \to SH^*(W) \to HH^*(Fuk(W)).
\end{equation} 

These maps have been shown to be isomorphisms  \cite{Ganatra-thesis, Ganatra-cyclic}, under hypotheses that are known to be satisfied when $W$ is Weinstein \cite{CDGG, Ganatra-Pardon-Shende2}.  These isomorphisms are important in the study of homological mirror symmetry, e.g. to see that all deformations of the wrapped Fukaya category can be geometrically interpreted via holomorphic curves with punctures asymptotic to Reeb orbits. 

On the other hand, the symplectic cohomology of Liouville manifolds as defined by \eqref{unfiltered symplectic cohomology intro} carries none of the quantitative information that symplectic cohomology has been historically used to capture \cite{FLoer-Hofer-symplectic, Floer-Hofer-Cieliebak,Biran-Poltervich-Salamon}.  This is restored by restricting the class of allowable Hamiltonians and remembering the action filtration.  Fix now a Liouville domain $W_0$ completing to $W$.  The symplectic cohomology of $W_0$ is, by definition, \cite{Viterbo-functorsI}: 
 \[ SH^*(W_0) := \varinjlim_{f|_{W_0} \le 0} HF^*(W, f).\]

Here, we restrict the previous collection of functions to those 
satisfying $f|_{W_0} \le 0$.  The filtration comes from the symplectic action functional $\int (\lambda-Hdz)$ on Hamiltonian trajectories, 
 which are, by definition, the generators of the Hamiltonian Floer complexes. (We may try to remember the action
filtration without restricting the class of functions, but in the limit the filtration would collapse.)  The resulting $\RR$-filtered symplectic cohomology does carry
quantitative information; in particular furnishing a symplectic capacity 
sufficient to establish the non-squeezing principle \cite{Viterbo-functorsI}.  
An $S^1$-equivariant version has been an even richer source of 
embedding obstructions \cite{gutthutchings2018capacities}. 
A recent survey of further applications of the filtration may be found in \cite{PRSZ-persistencebook}.

\vspace{2mm}
Comparing the discussions above, one might want a filtered version of \eqref{symplectic homology as hochschild homology}.  Unfortunately, there is not yet a filtered version of the wrapped Fukaya category. 
However, after \cite{Nadler-Zaslow, Ganatra-Pardon-Shende3, Viterbo, Guillermou-Viterbo}, it is natural to expect that any such filtered Fukaya category would be equivalent to some category of sheaves. 
The purpose of the present article is to show that, for open subsets of cotangent bundles, a filtered analogue of \eqref{symplectic homology as hochschild homology} holds, with the Fukaya category replaced by a certain category of sheaves. 
\vspace{2mm}

Let us review some basic notions of microlocal sheaf theory and introduce the relevant category.  
Fix some background choice of
symmetric monoidal stable presentable 
 category $\mathbf{k}$,
e.g. the category of modules over a commutative ring or ring spectrum.
For a topological space $X$ we write $\Sh(X)$ for the
(symmetric monoidal stable presentable) category of sheaves valued in $\mathbf{k}$.  
For a manifold $M$ and sheaf $F \in \Sh(M)$, Kashiwara and Schapira introduced a closed conic coisotropic subset  $ss(F) \subseteq T^*M$ measuring the failure of propagation of sections of $F$ \cite{Kashiwara-Schapira-SoM} when $\mathbf{k}$ is additionally compactly generated. Tamarkin explained  \cite{Tamarkin-non-displaceability} that to study non-conic subsets, one 
should consider the symplectic reduction diagram
\[ T^*(M \times \RR) \xleftarrow{i_+} J^1 M \xrightarrow{\pi} T^*M \]
and consider, for $U \subseteq T^*M$, the category
\[\sT(U) := \Sh(M \times \RR) / \{F\, |\, 
\pi(i_+^{-1} ss(F)) \cap U  = \varnothing\}.\]

We are ultimately interested in the Hochschild homology of $\sT(U)$.  One basic problem is how to define this at all, since $\sT(U)$, similarly to categories of sheaves $\Sh(M)$, is not compactly generated.  As such, 
at least one usual definition of Hochschild homology (bar complex on compact objects) is not appropriate.  In the case of $\Sh(M)$, one proceeds as follows. 
The category $\Sh(M)$ is dualizable
as a presentable stable category.  For dualiziable categories, 
the trace of the identity functor is 
a natural notion of Hochschild invariant, in particular, specializing to the usual notion for compactly generated categories  \cite[Prop. 4.24]{Hoyois-Scherotzke-Sibilla}. 

Functors on sheaf categories can be presented as integral kernels, and one
can write a formula for the trace.  
Indeed, if $H$ is a locally compact Hausdorff space, and  $K \in \Sh(H \times H)$ a sheaf with associated 
    integral transform  $\Phi_K: \Sh(H) \to \Sh(H)$, then 
\begin{equation} \label{intro trace of sheaves} 
    \Tr(\Phi_K) = \Gamma_c(H, \Delta^* K).
\end{equation}

Results along the lines of \eqref{intro trace of sheaves}
were long previously announced by Efimov, see e.g. \cite{Hoyois-Efimov-K-theory}, although details \cite{Efimov-K-theory} did not
appear until after the original preprint of the present article.  In any case, we include a proof of \eqref{intro trace of sheaves} here as Corollary \ref{trace of sheaves}. Also due to Efimov is the amusing corollary that  the category of sheaves is a categorification of compactly supported cohomology: 
$\Tr(1_{\Sh(H)}) = \Gamma_c(H, 1_H).$

\vspace{2mm}

We wish to apply similar ideas to $\sT(U)$.  We will want to work over the category  $\sT := \sT(point)$.  This category has been well studied in the literature (see e.g. \cite{Guillermou-Schapira}); as we recall in 
Section \ref{Tamarkin point}, $\sT$
carries a symmetric monoidal structure (\S \ref{t is monoidal}) and a natural $(\vec{\RR}, +)$ action induced by $1_{\RR_{\geq a}}\rightarrow 1_{\RR_{\geq b}}$ ($b\geq a$), and is equivalent to a category of filtered complexes. Finally in Section \ref{tamarkin category properties}, we arrive at the first main new contribution of the present article: 

\begin{theorem}[Prop. \ref{t linear} and \ref{t dualizable}]
    The category $\sT(U)$ is linear over 
    $\sT$, and $\sT$-linearly dualizable. 
Thus we may take the $\sT$-linear trace
\[\Tr : \End(\sT(U)) \to \sT\]
and in particular have $\Tr(1_{\sT(U)}) \in \sT$.
\end{theorem}

The $\sT$-linearity amounts to an `action filtration' on $\sT(U)$.  
Thus we have defined a Hochschild invariant of $\sT(U)$ valued in
(a category equivalent to) a category of filtered complexes,  hence obtained, as desired, a filtration on the Hochschild homology.

\vspace{2mm}
To  study the trace, we use the following fact: when $\sC \subseteq \sD$ is the image of a projector $P_{\sC}$, then $\Tr_{\sC}(1_{\sC}) = \Tr_{\sD}(P_{\sC})$ (see Lem. \ref{duality-of-retraction}).  
We will write \[P_U: \sT(T^*M) \to \sT(T^*M)\] for the projector onto $\sT(U)$. 

There is a natural isomorphism $\sT(T^*M) = \Sh(M, \sT)$ (Prop. \ref{tamarkinsheaf}).  After expressing the projector $P_U$ via an integral kernel $P_U \in \Sh(M \times M, \sT)$,  we have, by \eqref{intro trace of sheaves} and the above discussion,
\begin{equation} \label{formula for trace}
    \Tr(1_{\sT(U)}) = \Gamma_c(M, \Delta^* P_U).
\end{equation}

To make use of Equation \eqref{formula for trace}, one needs a formula for the kernel $P_U$.  
When $U$ is a ball, one such formula was given in 
\cite{Chiu-nonsqueezing} using  microlocal cutoff via Fourier transform; 
as explained in 
\cite{Zhang-capacities-Chiu-Tamarkin}, the formula and proof work in much more generality.  
We review this in \S \ref{fourier projector}.  Meanwhile in \cite{Kuo-wrapped-sheaves}, a different formula was given which computes the adjoint of inclusions $\Sh_X(M) \to \Sh(M)$ `by wrapping sheaves'; the ideas can be straightforwardly extended to the setting $\sT(U) \to \sT(T^*M)$, which we do in \S \ref{wrapping projector}.  The right hand side of \eqref{formula for trace} was the main object of study in \cite{Zhang-capacities-Chiu-Tamarkin,Zhang-S1-Chiu-Tamarkin}; we recall some results of those articles in \S \ref{bingyu results}.  

We will ultimately express our results in terms of Hochschild cohomology rather than Hochschild homology.  However, there is not much difference: in Theorem \ref{right calabi yau}, we show that the categories $\sT(U)$ are right Calabi-Yau when $U \subset T^*M$ with $M$ orientable, and hence obtain in Corollary \ref{proposition: tw HH co using sheaf} a formula relating Hochschild chains and cochains:  
\begin{equation} \HH_{\sT}^\bullet(\sT(U),(-\infty,L]) \simeq  \Hom(\Tr(1_{\sT(U)}), 1_{\RR_{\ge L} } [-n]). 
\end{equation}

Here, $\HH_{\sT}^\bullet(\sT(U),(-\infty,L]$ is the `action smaller than $L$' part of $\HH_{\sT}^\bullet(\sT(U))$.

In \S \ref{sec: guillermou-viterbo},
we turn to a comparison with filtered symplectic cohomology. Due to the present status of the Floer-theoretic literature, we now fix our coefficients $\mathbf{k}$ to be the dg category of modules over a field $\mathbb{F}$.  We
follow the symplectic cohomology conventions of \cite{Cieliebak-Frauenfelder-Oancea2010};  for a Liouville domain $W_0$ and for any interval $(a, b) \subseteq \RR$ whose ends $a,b$ are not in the action spectrum of $\partial_\infty W_0$,  there is a graded abelian group $SH_{(a, b)}^{*}(W_0)$. By applying a comparison result of Guillermou and Viterbo \cite[App. E]{Guillermou-Viterbo} to the aforementioned `wrapping' formula for the projector, we deduce:

\begin{theorem} \label{main theorem} Let $M$ be a closed manifold. Let $U \subseteq T^*M$ be a relatively compact set such that the Liouville form on $T^* M$ restricts to a contact form on $\partial U$. For any $L > 0$ which is not in the action spectrum of $\partial U$, we have an isomorphism 
\[\HH_{\sT}^*(\sT(U),(-\infty,L])\simeq SH_{(-\infty,L)}^{*}(\overline{U}) .\]    \end{theorem}

Theorem \ref{main theorem} is a filtered analogue of the celebrated isomorphism $SH^*(W) \xrightarrow{\sim} \HH^*(Fuk(W))$ for Weinstein manifolds $W$ \cite{Seidel-deformation, Ganatra-thesis, Ganatra-cyclic}, with the Tamarkin category $\sT(U)$ standing in for a filtered wrapped Fukaya category.  

We deduce the following comparison of previously defined symplectic capacities: 

\begin{corollary} \label{main corollary}
    The symplectic capacities defined by Viterbo in \cite{Viterbo-functorsI} using symplectic cohomology, and by the third-named author in  \cite{Zhang-S1-Chiu-Tamarkin} using microsheaf theory agree.
\end{corollary}

Theorem \ref{main theorem} and Corollary \ref{main corollary} establish the non-equivariant version of Conjectures 5.1, 5.2 and 5.3 from \cite{zhang-thesis}, which in turn are inspired the discussion in \cite[Sec. 4.8]{JunZhang}.

\vspace{2mm}

Finally, a similar (but easier) argument gives a comparison result between the Hochschild cohomology and filtered generating function homology, in Theorem \ref{generating function comparison}.

\vspace{5mm}
{\bf Acknowledgements.}  We thank Tomohiro Asano, Shaoyun Bai, Sheng-Fu Chiu, Kai Cieliebak, Sheel Ganatra, Zhen Gao, St\'ephane Guillermou,  Peter Haine, Joseph Helfer, Yuichi Ike, Wenyuan Li, Shuaipeng Liu, Alexandru Oancea, Nick Rozenblyum, Germ\'an Stefanich, Kyler Siegel, and Claude Viterbo for helpful conversations. 
The work presented in this article is supported by Novo Nordisk Foundation grant NNF20OC0066298, Villum Fonden Villum Investigator grant 37814, and Danish National Research Foundation grant DNRF157.

\section{Traces and dualizable categories} \label{traces}

Let us recall the formalism of traces, 
which provides an appropriate notion of Hochschild homology for presentable categories which are dualizable but not compactly generated. 

Let $\cM$ be a symmetric monoidal 1-category and $1_\cM$ its unit object. 
An object $X \in \cM$ is said to be dualizable if there exists $Y \in \cM$ and maps
\begin{equation} \label{unitcounit} \eta: 1_\cM \rightarrow Y \otimes X, \qquad \qquad \ \epsilon: X \otimes Y \rightarrow 1_\cM
\end{equation}
such that the triangle equalities
\begin{equation} \label{for: triangle-identity} 
(\epsilon \otimes \id_X)\circ (\id_X \otimes \eta) = \id_X \, \text{and} \ (\id_Y\otimes \epsilon) \circ (\eta\otimes\id_Y ) = \id_Y,
\end{equation}
are satisfied. In this case we write $Y = X^\vee$. Note that when $X^\vee$ exists, the inner Hom, $\inHom(X,Z)$, for $Z \in \cM$ is given by $X^\vee \otimes Z$ since $\epsilon$ and $\eta$ exhibit $X^\vee \otimes (-)$ as the right adjoint of $X \otimes (-)$.

We recall the classical: 

\begin{definition}
For a dualizable object $X$, and an endomorphism $f: X \rightarrow X$, the trace of $f$, denoted by $\Tr(f,X)=\Tr(f)$, is defined to be the object in $\End(1_\cM)$ defined by the composition
\[1_\cM \xrightarrow{\eta} X^\vee \otimes X \xrightarrow{\id \otimes f} X^\vee \otimes X  = X \otimes X^\vee \xrightarrow{\epsilon} 1_\cM .\]
\end{definition}
The term `trace' comes from the fact that
when $\cM$ is the category of vector spaces, dualizablility is equivalent to finite dimensionality, and the trace of an endomorphism of a finite dimensional vector space is the trace in the sense of linear algebra.  

The notion of trace generalizes naturally to  higher categorical contexts, where the trace recovers and generalizes Hochschild homology.  Here we recall some of the relevant notions from \cite[\S 4]{Hoyois-Scherotzke-Sibilla}; see also \cite[\S 5.1.1]{Ben-Zvi-Francis-Nadler}. 

Fix a rigid
symmetric monoidal idempotent-complete small stable category $\cE$ and $\mathbf{k}\coloneqq \Ind(\cE)$; we consider $\PrLst(\mathbf{k})$, the symmetric monoidal category of stable presentable categories linear over  $\mathbf{k}$.

We recall some relevant facts about $\PrLst(\mathbf{k})$, all of which can
be found in \cite[\S 4]{Hoyois-Scherotzke-Sibilla}. 

\begin{itemize}
    \item The endomorphims 
of $\mathbf{k}$ in $\PrLst(\mathbf{k})$ is 
$\mathbf{k}$. 
    \item There is a full subcategory 
$Cat^{Mor}(\cE) \subseteq \PrLst(\mathbf{k})$, whose objects comprise the essential image of $\Ind$. 
    \item  Passing to compact objects gives an equivalence from $Cat^{Mor}(\cE)$ to a category whose objects are small idempotent complete $\mathbf{k}$-linear categories and whose morphisms can be identified with bimodules.  
    \item The objects $\cC$ of  $Cat^{Mor}(\cE)$ are all dualizable, with
$\Ind(\cC)^\vee = \Ind(\cC^{op})$.  
    \item If a morphism $f: \Ind\, \cC \to \Ind\, \cC$ is given by the $\cC$-bimodule $B_f$, then $\Tr(f) \in \Ind \, \cE$ is naturally identified with the Hochschild homology of $B_f$, e.g. as computed by the bar complex \cite[Prop. 4.24]{Hoyois-Scherotzke-Sibilla}. 
\end{itemize}

Because of the last point above, it is natural to use the trace as a definition of (or substitute for) the Hochschild homology for 
categories in $\PrLst(\mathbf{k})$ which
are dualizable but not compactly generated.  In fact, according to Efimov, the trace is the Kan extension of Hochschild homology to dualizable categories \cite{Efimov-K-theory}. 

Dualizable but not compactly generated categories are plentiful: a classical 
and relevant example is the category of sheaves on a manifold \cite{Neeman-sheavesnotcompact, Neeman-triangulated}.  More generally, Lurie has shown that presentable categories are dualizable iff they are retracts of compactly generated categories \cite[Proposition D.7.0.7]{Lurie-SAG} -- which applies to the above example, as the category of sheaves is a retract of the compactly generated category of presheaves \cite[Proposition 21.1.7.1]{Lurie-SAG}. 

\subsection{Dualizing datum for retraction}

In linear algebra, the trace of a projector is equal to the trace of the identity on its image.  We have the following generalization: 

\begin{lemma}\label{duality-of-retraction}
For a dualizing datum $(X,X^\vee,\epsilon,\eta)$ in $\cM$,
 let $e: X \rightarrow X$ be an idempotent which can be written as $X \xrightarrow{r} Y \xrightarrow{i} X$
for some inclusion $i$ and some retraction $r$. 

Assume that the dual idempotent $e^\vee: X^\vee \rightarrow X^\vee$ also splits to $X^\vee \xrightarrow{s} Z \xrightarrow{j} X^\vee$. 
Then the pair
\[\epsilon_0 \coloneqq \epsilon \circ (i \otimes j) : Y \otimes Z \rightarrow 1_{\cM} \qquad \qquad \qquad \eta_0 \coloneqq (s \otimes r) \circ \eta : 1_{\cM} \rightarrow Z \otimes Y\]
exhibits $Z$ as the dual of $Y$. 
Moreover, we have
\[\Tr(e,X)= \Tr(\id_Y,Y) =\Tr(\id_Z,Z)=\Tr(e^\vee,X^\vee).\]

\end{lemma}

\begin{proof}
To verify $(Y,Z,\epsilon_0,\eta_0)$ is a dualizing datum, we shall verify the triangle identities \eqref{for: triangle-identity}, i.e., we claim that
$(\epsilon_0 \otimes \id_Y) \circ (\id_Y \otimes \eta_0)=\id_Y$ and $(\id_Z \otimes \epsilon_0 ) \circ (\eta_0 \otimes \id_Z)=\id_Z$. Since the proofs of these two equivalence are similar, we check the first one, which can be exhibited by the following commuting diagram that we will explain:
$$
\begin{tikzpicture}
\node at (0,4) {$Y$};
\node at (3,4) {$Y \otimes Z \otimes Y$};
\node at (6,4) {$Y$};
\node at (0,2) {$Y \otimes X^\vee \otimes X$};
\node at (6,2) {$X \otimes X^\vee \otimes Y$};
\node at (-3,0) {$X$};
\node at (0,0) {$X \otimes X^\vee \otimes X$};
\node at (6,0) {$X \otimes X^\vee \otimes X$};
\node at (9,0) {$X$};

\node[scale=1.5] at (3,3) {$\circlearrowleft$};

\draw [->, thick] (0.2,4) -- (2,4) node [midway, above] {$ \id \otimes \eta_0 $};
\draw [->, thick] (4,4) -- (5.8,4) node [midway, above] {$ \epsilon_0 \otimes \id$};
\draw [->, thick] (1.2,2) -- (4.8,2) node [midway, above] {$i \otimes e^\vee \otimes r$};

\draw [->, thick] (-2.8,0) -- (-1.2,0) node [midway, above] {$\id \otimes \eta$};
\draw [->, thick] (1.2,0) -- (4.8,0) node [midway, above] {$\id \otimes e^\vee \otimes \id$};
\draw [->, thick] (7.2,0) -- (8.8,0) node [midway, above] {$\epsilon \otimes \id$};

\draw [->, thick] (0,3.7) -- (0,2.3) node [midway, left] {$\id \otimes \eta$}; 
\draw [<-, thick] (6,3.7) -- (6,2.3) node [midway, right] {$\epsilon \otimes \id$};
\draw [->, thick] (0,1.7) -- (0,0.3) node [midway, left] {$i \otimes \id$}; 
\draw [<-, thick] (6,1.7) -- (6,0.3) node [midway, right] {$\id \otimes r$};

\draw [->, thick] (0.5,2.3) -- (2.6,3.7) node [midway, right] {$ $}; 
\draw [->, thick] (3.4,3.7) -- (5.5,2.3) node [midway, left] {$ $};

\draw [->, thick] (-0.3, 4){[rounded corners]--(-3,4)}--(-3,0.3) node [midway, above] {$ $}; 
\node at (-3.3,2) {$i$};
\draw [->, thick] (9, 0.3){[rounded corners]--(9,4)} --(6.3,4) node [midway, above] {$ $}; 
\node at (9.3,2) {$r$};

\node at (1,3.4) {$(1)$};
\node at (5,3.4) {$(2)$};
\node at (3,1.1) {$(3)$};
\node at (-2,2) {$(4)$};
\node at (8,2) {$(5)$};

\end{tikzpicture}
$$

First, the maps $\eta_0$ and $\epsilon_0$ are defined by the composition indicated by (1) and (2), and the two slanted arrows are given by the composition
\[ ( i \otimes j \otimes \id_Y) \circ (\id_Y \otimes s \otimes r)  = i \otimes e^\vee \otimes r,\]
which is the horizontal arrow in the middle level. Being a tensor of three maps, we can further decompose it to three arrows as indicated by (3).
But if $a$ and $b$ are morphisms, the tensor structure allows us to interchange the order of composition, i.e., 
$$a \otimes b = (\id \otimes b) \circ (a \otimes \id) = (a \otimes \id) \circ (\id \otimes b),$$
and we thus obtain (4) and (5). But the composition of the three arrows at the bottom
\[(\epsilon \otimes \id) \circ (\id \otimes e^\vee \otimes \id) \circ (\id \otimes \eta)\]
is by the definition $(e^\vee)^\vee = e$, since taking dual is an (anti)-involution. In summary, the composition $(\epsilon_0 \otimes \id_Y) \circ (\id_Y \otimes \eta_0)=\id_Y$ can be identified, by going through the outer arrows, as 
$$ r \circ e \circ i = r \circ i \circ r \circ i = \id_Y.$$

The equality $(\id_Z \otimes \epsilon_0 ) \circ (\eta_0 \otimes \id_Z)=\id_Z$ can be checked similarly by reversing the roles of $Y$ and $Z$.

For the trace, we have
\[\Tr(e,X)=\Tr(i\circ r,X)=\Tr(r\circ i,Y)=\Tr(\id_Y,Y),\]
where the second equality is the commutativity of trace \cite[Proposition 2.4]{Ponto-Shulman}.\end{proof}

In all cases of interest to us, splitting of the dual idempotent is verified by the following: 

\begin{lemma} \label{closed idempotent splits}
If $\cM$ is closed, then the dual idempotent splits. 
\end{lemma}
\begin{proof}
    We can naturally identify $X^\vee=\inHom(X,1_\cM)$ and $e^\vee=\inHom(e,1_\cM)$. So, we can take $Z=\inHom(Y,1_\cM)$, $j=\inHom(r,1_\cM)$ and $s=\inHom(i,1_\cM)$.
\end{proof}

\subsection{Extension of scalars}\label{subsection: Extending coefficient}

 Let $\cM$ be a closed symmetric monoidal category and $A\in \CAlg(\cM)$ be a commutative ring object, i.e., there exists a multiplication and unit map,
$$m_A: A \otimes A \rightarrow A \ \text{and} \ u: 1_\cM \rightarrow A$$
satisfying the usual ring axioms up to higher coherence.

Denote by $A \dMod(\cM)$ the category of $A$-module objects in $\cM$. Regard $A$ as a module over itself, we have the (left) regular representation of $A$
\[\rho_A: A\to \inHom(A,A),\quad \rho_A(a)= [b\mapsto m_A(a\otimes b)],\] 
or pass to tensor-hom adjunction defined by
\[m_A \in  \Hom(A\otimes A,A)=\Hom(A,\inHom(A,A)).\]

If we furthermore assume that $\cM$ admits geometric realization, then $A \dMod(\cM)$ admits a symmetric monoidal structure given by the relative tensor $\otimes_A$ with $A$ being its unit. Thus, we have a symmetric monoidal functor 
\[(-) \otimes A: \cM  \rightarrow A \dMod(\cM), \quad
X  \mapsto X \otimes A .\]

\begin{lemma} \label{lem: coefficient-extension}
If $(X, X^\vee, \epsilon, \eta)$ is a dualizing datum in $\cM$, then $ (-)\otimes A$ induces a dualizing datum
\[(X \otimes A, X^\vee \otimes  A,  \epsilon\otimes\id_A ,  \eta\otimes  \id_A)\]
in $A \dMod(\cM)$.
\end{lemma}
\begin{proof}
We notice that there is an identification $(X \otimes A) \otimes_A (X^\vee \otimes A) =  (X \otimes X^\vee) \otimes A$. Then it follows from that $(-)\otimes A$ is a symmetric monoidal functor.
\end{proof}

\begin{lemma}\label{corollary: dualizing data of free modules over locally rigid algebra}For if $A$ and $X$ are dualizable objects with the dualizing data $(A,A^\vee,\epsilon_A,\eta_A )$ and $(X,X^\vee,\epsilon,\eta)$ respectively, we have that $X \otimes \cA$ is dualizable in $A \dMod(\cM)$ with dualizing datum
\[(  X\otimes A, X^\vee \otimes  A^\vee, \epsilon\otimes \epsilon_A, \eta \otimes  \eta_A ).\]
\end{lemma}
\begin{proof}
    A tensor product of two dualizable objects is dualizable.
\end{proof}

We can define $A$-linear relative internal hom $\inHom_A$ on $A \dMod(\cM)$ using $\inHom$ of $\cM$: for $E,M,N$ in $A \dMod(\cM)$, we define $\inHom_A(M,N)\in A \dMod(\cM)$ as the unique object such that
\[\inHom(E,\inHom_A(M,N))= \inHom(E\otimes_A M,N).\]

There is a forgetful map $\inHom_A(M,N) \rightarrow \inHom(M,N)$ defined in the following way: For all $E\in A \dMod(\cM)$ we have
\[\inHom(E,\inHom_A(M,N))=\inHom(E\otimes_A M,N)\rightarrow \inHom(E\otimes M,N)=\inHom(E,\inHom(M,N)). \]
Then we pick $E=\inHom_A(M,N)$, then $\id_E$ defines a map
\[\inHom_A(M,N) \rightarrow \inHom(M,N).\]

\begin{lemma}\label{lemma: restriction fully-faithful}For $X,Y,A \in \cM$  such that $A,X$ are dualizable, we have the natural forgetful map 
\[\inHom_A(X\otimes A, Y\otimes A) \rightarrow \inHom(X\otimes A, Y\otimes A)\]
can be identified with 
\[\inHom(X,Y\otimes A ) \xrightarrow{\id\otimes \rho_A} \inHom(X,Y\otimes \inHom(A,A))\]
for the regular representation $\rho_A$.
\end{lemma}
\begin{proof}By Yoneda and dualizability of $X$, we have
\[\inHom_A(X\otimes A, Y\otimes A)=\inHom(X, Y\otimes A).\]

Then the dualizability of $A$ and $X$ gives similarly that
\[\inHom(X\otimes A, Y\otimes A)=\inHom(X , Y ) \otimes \inHom(A,A)=\inHom(X , Y   \otimes \inHom(A,A)).\]

To verify the natural forgetful functor is indeed $\id\otimes \rho_A$, we play the same game with more careful on morphisms when applying Yoneda.
\end{proof}

Now we specializes to the case when $\cM = \PrLst(\mathbf{k})$. As several coefficient categories show up, we will carefully indicate corresponding coefficient categories we are working with.

Equipping the $\mathbf{k}$-linear Lurie tensor product, $ \PrLst(\mathbf{k})$ is a closed symmetric monoidal category admits geometric realization. The internal hom is the category of $\mathbf{k}$-linear left adjoint functors $\Fun^L$; and for a commutative ring object $\cA$, i.e. $\cA$ is a symmetric monoidal presentable category and there is a symmetric monoidal functor $u: \mathbf{k} \rightarrow \cA$ which is colimit-preserving. The $\cA$-linear relative internal hom in $\cA \dMod(\PrLst(\mathbf{k}))$, which consists of $\cA$-linear functors, is denoted by $\Fun^L_\cA$. 

The regular representation functor is denoted by 
\begin{equation} \label{regular representation} \rho_\cA: \cA\rightarrow \Fun^L(\cA,\cA). \end{equation}

What is relevant to us is the following lemma.
\begin{lemma}\label{lemma: restriction fully-faithful-PrL}Let $\cA$ be a dualizable symmetric monoidal $\mathbf{k}$-linear category. For two $\mathbf{k}$-linear dualizable categories $X$ and $Y$, the $\mathbf{k}$-linear forgetful functor
\begin{equation}\label{equation: forgetful functor between functor categories}
  \Fun^L_\cA( X \otimes\cA,  Y \otimes\cA) \rightarrow \Fun^L(  X\otimes\cA,  Y\otimes\cA)  
\end{equation}
is fully-faithful if the regular representation functor $\rho_\cA: \cA\to \Fun^L(\cA,\cA)$ is fully-faithful.
\end{lemma}
\begin{proof}By Lemma \ref{lemma: restriction fully-faithful}, under the circumstance,we can identify the functor \eqref{equation: forgetful functor between functor categories} as the functor
\[ (X^\vee \otimes Y)\otimes\cA \xrightarrow{\id_{X^\vee\otimes Y} \otimes \rho_\cA}  (X^\vee \otimes Y)\otimes\cA \otimes \cA^\vee .\]

Since that $X^\vee \otimes Y$ is dualizable over $\mathbf{k}$, then it follows from \cite[Theorem 2.2]{Efimov-K-theory} that the above functor is fully-faithful if $\rho_\cA$ is fully-faithful.    
\end{proof}

\begin{corollary}\label{lemma: restriction fully-faithful for quotient}For $\cA,X,Y$ in Lemma \ref{lemma: restriction fully-faithful-PrL}, and we assume that the regular representation $\rho_\cA$ is fully-faithful, then for an $\cA$-linear quotient $\sC$ of $X\otimes \cA $ and an $\cA$-linear full-subcategory $\sD$ of $Y\otimes \cA $, we have the $\mathbf{k}$-linear forgetful functor
\[\Fun^L_\cA( \sC,  \sD) \rightarrow \Fun^L(  \sC,  \sD)\]
is fully-faithful.    
\end{corollary}
 \begin{proof}We have a commutative diagram of $\mathbf{k}$-linear functors 
 \[\begin{tikzcd}
{\Fun^L_\cA( \sC,  \sD)}  \arrow[d,"f_1"] \arrow[r,"A"] & {\Fun^L_\cA( X \otimes\cA,  \sD)} \arrow[r,"B"] \arrow[d,"f_2"] & {\Fun^L_\cA( X \otimes\cA,  Y \otimes\cA)} \arrow[d,"f_3"] \\
{{\Fun^L(  \sC,  \sD)}} \arrow[r,"a"]              & {\Fun^L( X \otimes\cA,  \sD)} \arrow[r,"b"]               & {\Fun^L( X \otimes\cA,  Y \otimes\cA)} .             
\end{tikzcd}\]

Functor $f_3$ is fully-faithful by Lemma \ref{lemma: restriction fully-faithful-PrL} and $b,B$ are fully-faithful since $\sD$ is a full-subcategory of $Y \otimes\cA$. Then we known that $f_2$ is $\mathbf{k}$-linearly fully-faithful. 

Lastly, since $\sC$ is a quotient, we have that $a,A$ are fully-faithful. Then we have $f_1$ is fully-faithful.     
 \end{proof}

 \begin{remark}In general, $\rho_\cA$ need not be fully-faithful (see Remark \ref{remark: non-fully faithful regular rep example}). However we show in Corollary \ref{corollary: m^r_T fully-faithful}  that the regular representation for the Tamarkin category $\sT$ is indeed fully-faithful.  Using this, we will later deduce from Corollary \ref{lemma: restriction fully-faithful for quotient}
 a comparison between the Hochschild cohomology over $\cA$ and over $\mathbf{k}$; see Lemma \ref{lemma: Hochschild comparison}.
 \end{remark}

Lastly, let us remark the enriched hom objects.

Consider $\mathbf{k}$-linear category $\sC \in \PrLst(\mathbf{k})$. For $X,Y \in \sC$, we have a $\mathbf{k}$-linear valued hom object, $\Hom_{\sC/\mathbf{k}}(X,Y)\in \mathbf{k}$. Its definition is similar to the $\cA$-linear valued hom object we explained below when consider $\sC$ just as a presentable stable category in $\PrLst$. Without further clarify, we will just call $\Hom_{\sC/\mathbf{k}}(X,Y)$ as the $\mathbf{k}$-linear Hom of $\sC$, and we will write $\Hom_{\sC}(X,Y)$ if no need to distinguish coefficient there, and $\Hom(X,Y)$ is the category $\sC$ is clear.

For $\cA$ the symmetric monoidal presentable category, as it is presentable, $\cA$ also admits an internal hom object valued in $\cA$, which is denoted by $\sHom_{\cA}$. 

If the $\mathbf{k}$-linear category $\sC$ is a further more $\cA$-module, i.e., there is an action functor $\cA\otimes \sC\rightarrow \sC$ such that the action functor is compatible with $u: \mathbf{k} \rightarrow \cA$ and $m_\cA: \cA\otimes \cA\to \cA$, we can construct the $\cA$-linear relative inner hom in the following way: The $\cA$-module structure produces a functor
\[\cA \rightarrow  \cA \otimes \sC \rightarrow  \sC ,\quad
a \mapsto     a \otimes X \mapsto a\cdot X,\]
where the last functor is the $\cA$-module action on $\sC$. Then we define the $\cA$-valued hom object $\Hom_{\sC/\cA}(X,Y) \in \cA$ as the right adjoint functor of the $\cA$-action functor above, i.e., we have
\begin{equation}\label{lem: for-coefficient}
  \Hom_{\cA/\mathbf{k}}(a, \Hom_{\sC/\cA}(X,Y) ) = \Hom_{\sC/\mathbf{k}}(a\cdot X,Y )\,  \in\mathbf{k}.  
\end{equation}

In particular, if $f: \sC\rightarrow \sD$ is a fully-faithful $\cA$-linear functor, then $f$ induces an equivalence on $\Hom_{\sC/\cA}(X,Y)$ and vice versa. 

By definition and Yoneda, one has that $\sHom_{\cA}=\Hom_{\cA/\cA} $ and 
\[\sHom_{\cA}(a, \Hom_{\sC/\cA}(X,Y) ) = \Hom_{\sC/\cA}(a\cdot X,Y ) \, \in  \cA.\]

\section{Trace of \texorpdfstring{$\Sh_X(M)$}{}} \label{sheaftraces}

For $H$ a locally compact Hausdorff space,  $\Sh(H)$ is a retract of the category of presheaves \cite[Def. 21.1.2.1, Thm. 21.1.6.12, Prop. 21.1.7.1]{Lurie-SAG} and hence
is dualizable.  
It is also easy to explicitly exhibit the unit and counit, using the (closely related) fact
that the natural functor 
$\Sh(H) \otimes \Sh(H) \to \Sh(H \times H)$ is an equivalence, for any locally compact Hausdorff $H$  \cite[Proposition 2.30]{volpe_6_functor}. 

\begin{proposition} \label{dualizability of sheaves}
For $H$ locally compact Hausdorff, 
the category $\Sh(H)$ is its own dual, with unit and co-unit given by
\begin{align*}
     \eta&: {\mathbf{k}} \rightarrow \Sh(H\times H),\quad V \mapsto \Delta_! a^*V,\\
    \epsilon &: \Sh(H\times H) \rightarrow {\mathbf{k}},\quad F \mapsto a_!\Delta^*F,
\end{align*}
where $a: H\rightarrow \pt$ is the constant map and $\Delta: H\rightarrow H\times H$ is the diagonal map.
\end{proposition}
\begin{proof}
Computing $(1 \otimes \eta) \circ (\epsilon \otimes 1)$ is an elementary exercise in the use of base change. 
\end{proof} 

\begin{proposition}\label{Kernel-functor corresponding}
Via the above identification  $\Sh(H) = \Sh(H)^\vee$, 
the composition of equivalences
\[\Sh(H_1 \times H_2) \xleftarrow{\sim} \Sh(H_1) \otimes \Sh(H_2) = \Sh(H_1)^\vee \otimes \Sh(H_2) \xrightarrow{\sim} \Fun^L(\Sh(H_1), \Sh(H_2))\]
sends $K \in \Sh(H_1 \times H_2)$ to the integral transform $\Phi_K=[F\mapsto p_{2!}(K\otimes p_1^*F)]$.
\end{proposition}
\begin{proof}For $K\in \Sh(H_1 \times H_2) \xleftarrow{\sim} \Sh(H_1) \otimes \Sh(H_2) = \Sh(H_1)^\vee \otimes \Sh(H_2)$, the corresponding functor is computed by the composition
\[
    \Sh(H_1) \xrightarrow{\Id_1 \otimes K} \Sh(H_1)\otimes \Sh(H_1 \times H_2) \simeq \Sh(H_1 \times H_1)\otimes \Sh( H_2) \xrightarrow{\epsilon_1\otimes \Id_2} \Sh( H_2).
\]

On objects this sends
\[    F \mapsto F\boxtimes K \mapsto p_{2!}(\Delta_1\times \Id_{H_2})^* (F\boxtimes K) \simeq p_{2!}(K\otimes p_1^*F).\qedhere\] 
\end{proof}

\begin{corollary}\label{trace of sheaves}
There is a natural isomorphism 
\[\Tr(\Phi_K)=\Gamma_c(H, \Delta^*K).\]

In particular, as $\Phi_{1_{\Delta}}=1_{\Sh(H)}$, we have
\[\Tr(1_{\Sh(H)}) = \Gamma_c(H, 1_H).\]
\end{corollary}
\begin{proof}
We compute: 
\[\Tr(\Phi_K) \cong a_!\Delta^* v^* K  \cong \Gamma_c(H,\Delta^*K). \]
where $v: H \times H \to H \times H$ is the interchange of factors (and is acting trivially in this formula since $\Delta^* v^* = \Delta^*$). 
 \end{proof}

\begin{remark}\label{remark: compactly generated}For validity of Proposition \ref{dualizability of sheaves}, Proposition \ref{Kernel-functor corresponding} and Corollary \ref{trace of sheaves}, we only need that $\mathbf{k}$ is presentable. But for discussion involve microsupport, we do require that $\mathbf{k}$ is compactly generated as explained in \cite[Remark 4.24.]{Efimov-K-theory}.
\end{remark} 

Let us turn to the category of sheaves with some prescribed microsupport. 
For a closed subset $X \subseteq S^* M$, the inclusion $\Sh_X(M) \hookrightarrow \Sh(M)$ is continuous and co-continuous, so in particular has for formal reasons a left adjoint, which we denote $\iota_X^*: \Sh(M) \to \Sh_X(M)$.

In fact, this functor is realized by the integral transform $\Phi_K$ when $K = \iota_{-X \times^c X}^*(1_{\Delta_M})$ where, for another closed subset $Y \coloneqq S^* N$, where the product $X \times^c Y$ is the `conic' product in $S^*(M \times N)$ which is given by
\[X \times^c Y \coloneqq \left( \left( (\Rp X \cup 0_M) \times (\Rp Y \cup 0_N) \right) \setminus 0_{M \times N}  \right) / \Rp.\]

This assertion is a special case of the following proposition:

\begin{proposition}[{\cite[Lemma 7.10]{Kuo-Li-duality-spherical}}]\label{projector kunneth}
For any sheaf kernel $K \in \Sh( M \times M)$, the three functors from $\Sh(M)$ to $\Sh_X(M)$, $\iota_X^* \circ \Phi_K \circ \iota_X^*$,  $\Phi_{\iota_{-X \times X}^*(K)}$, and $ \Phi_{\iota_{-X \times X}^*(K)} \circ \iota_X^*$ are  all equivalent to each other.

In particular, $\iota_X^*  = \Phi_{\iota_{-X \times X}^*(1_\Delta)}$.
\end{proposition}

\begin{proof}
All three functors having target in $\Sh_X(M)$ is a standard exercise of microsupport estimation.
The functor $\Phi_K$ has a right adjoint which is given by $\Psi_K \coloneqq [G \mapsto {p_1}_* \sHom(K, p_2^! G)]$.
Thus, for any $G \in \Sh(M)$ and $F \in \Sh_X(M)$, we have
\begin{align*}
&\Hom(\Phi_{\iota_{-X \times X}^*(K)}(G),F)
= \Hom \left( G, \Psi_{\iota_{-X \times X}^*(K)} (F)  \right) \\
=& \Hom \left(\iota_X^*(G), \Psi_{\iota_{-X \times X}^*(K)} (F)  \right)
= \Hom(\Phi_{\iota_{-X \times X}^*(K)} \circ \iota_X^*(G),F),
\end{align*}
and we conclude the latter two are the same. Note we use the fact that $\Psi_{\iota_{-X \times X}^*(K)} (F)$ is in $\Sh_X(M)$ for the third equality.

To show that the first two are the same, by the above, we may restrict them to $\Sh_X(M)$ and assume $G \in \Sh_X(M)$. As a consequence, for $F \in \Sh_X(M)$, we have $\sHom(p_1^* G, p_2^! F) \in \Sh_{-X \times X}(M \times M)$. This implies that,

\vspace{-0.1in}
\begin{align*}
&\Hom(\Phi_{\iota_{-X \times X}^*(K)}(G),F) = \Hom({p_2}_!(\iota_{-X \times X}^* (K) \otimes G), F)
= \Hom \left(\iota_{-X \times X}^* (K), \sHom(p_1^* G, p_2^! F) \right) \\
=& \Hom \left( K, \sHom(p_1^* G, p_2^! F) \right) 
= \Hom( \Phi_K(G),F) 
= \Hom(\iota_X^* \circ \Phi_K(G), F).\qedhere
\end{align*}

\end{proof}

\begin{corollary}
$\Tr(1_{\Sh_X(M)},\Sh_X(M)) = \Tr(\Phi_{\iota_{-X \times X}^* (1_\Delta)},\Sh(M))=\Gamma_c(M,\Delta^*\iota_{-X \times X}^* (1_\Delta)).$  
\end{corollary}
\begin{proof}
    Immediate from Cor. \ref{trace of sheaves} and Prop. \ref{projector kunneth}. 
\end{proof}

\begin{remark} \label{dual-microsupport-condition}
One can show along similar lines that the dual idempotent in the sense of Lemma \ref{duality-of-retraction} is given by $\iota_{X \times -X}^* (1_\Delta) \circ (-): \Sh(M) \rightarrow \Sh(M)$, and that correspondingly $\Sh_X(M)^\vee = \Sh_{-X}(M)$ with unit and counit restrict from 
\[\epsilon_X \coloneqq a_!  \Delta^*: \Sh(M \times M) \rightarrow \mathbf{k}, \ \eta_X \coloneqq \iota_{-X \times X}^* (1_\Delta): \mathbf{k} \rightarrow \Sh(M \times M).\]

We also have that, for $U=T^*M\setminus X$, $\Sh(M;U)=\Sh(M)/\Sh_X(M)$ is dualizable that correspondingly $\Sh(M;U)^\vee = \Sh(M;-U)$ with unit and counit descent from
\[\epsilon_U \coloneqq a_!  \Delta^*: \Sh(M \times M) \rightarrow \mathbf{k}, \ \eta_U \coloneqq \cof{(1_\Delta\rightarrow \iota_{-X \times X}^* (1_\Delta))}: \mathbf{k} \rightarrow \Sh(M \times M).\]

This is a consequence of the following proposition which states that the dual of an integral transform is simply given by the same integral transform with its components swapped.
\end{remark}

\begin{proposition}
Let $K \in \Sh(H_1 \times H_2)$ be a sheaf kernel.
The dual of the integral transform $\Phi_K: \Sh(H_1) \rightarrow \Sh(H_2)$, under Proposition \ref{dualizability of sheaves},
is given by $\Phi_{v^* K}   :\Sh(H_2) \rightarrow \Sh(H_1)$ where $v: H_1 \times H_2 = H_2 \times H_1$ is the swapping map $v(x,y) \coloneqq (y,x)$.
\end{proposition}

\begin{proof}
Recall that if $X$ and $Y$ are dualizable, then for a morphism $f: X \rightarrow Y$, the dual $f^\vee$ is given by 
the composition
\[  Y^\vee \xrightarrow{\eta_X \otimes \id_{Y^\vee} } X^\vee \otimes X \otimes Y^\vee
\xrightarrow{ \id_{X^\vee} \otimes f \otimes \id_{Y^\vee}} X^\vee \otimes Y \otimes Y^\vee 
\xrightarrow{\id_{X^\vee} \otimes \epsilon_Y} X^\vee.\]

Following this definition, the proof for the statement becomes a straightforward yet lengthy six-functor yoga on the triple product $H_2 \times H_1 \times H_1$, which is very close, in spirit, to the proof of \cite[Proposition 3.6.4]{Kashiwara-Schapira-SoM}.
\end{proof}

\section{The Tamarkin category of a point} \label{Tamarkin point}

Let us write $\Sh_-(\RR)$ for the category of sheaves with nonpositive microsupport, and correspondingly $\Sh_+(\RR)$ for the sheaves with nonnegative microsupport.  In this section we study the ($\mathbf{k}$-linearly) dualizable category

\[\sT := \Sh(\RR) / \Sh_-(\RR).\]

In Remark \ref{dual-microsupport-condition}, we explain that the dual $\sT^\vee$ is naturally identified with $\Sh(\RR)/\Sh_{+}(\RR)$, which is not identical with $\sT$. However, we can identity $\sT^\vee=\Sh(\RR)/\Sh_{+}(\RR)$ with $\sT$ via the map $\inv:\RR\rightarrow \RR, \quad t\mapsto -t$, which leads to
\begin{proposition}The Tamarkin category $\sT$ are self-dual.
\end{proposition}

\begin{proof}Here, we use the following identification
\[\Sh(\RR^2)=\Sh(\RR)\otimes \Sh(\RR) \xrightarrow{\inv\otimes \id} \Sh(\RR)\otimes \Sh(\RR),\]
which can also be identified with $A_*$ for $A(t_1,t_2)=(-t_1,t_2)$. 

Then for $\delta=A\Delta $ where $\delta:\RR\rightarrow \RR^2, \,\delta(t)=(-t,t)$, we have $\Sh(\RR)$ is self-dual with the dualizing datum $\epsilon'=a_!\delta^*$ and $\eta'=\delta_!a^*$ via the same computation as Proposition \ref{dualizability of sheaves}.

For $\sT$, we notice that $\inv_*: \Sh_{\pm}(\RR)\simeq\Sh_{\mp}(\RR) $, $A_*:\Sh_{+}(\RR)\otimes \Sh_{-}(\RR) \simeq \Sh_{-}(\RR)\otimes \Sh_{-}(\RR)$, and then descent to $A_*: \sT^\vee\otimes \sT \simeq \sT\otimes \sT$. Then the dualizing datum $\epsilon'$ and $\eta'$ of $\Sh(\RR)$ descent to a self-dualizing datum of $\sT$ via Lemma \ref{duality-of-retraction} and Remark \ref{dual-microsupport-condition}.
\end{proof}

We can also think of $\sT$ in terms of 
the left adjoint to the quotient $\Sh(\RR) \to \sT$, which  embeds $\sT \hookrightarrow \Sh_+(\RR)$ as the full subcategory on objects with no compactly supported global sections. 

\begin{proposition}
The following three subcategories of $\Sh(\RR)$ are the same: (i) $\sT$, (ii)$\langle 1_{\RR_{\geq a}}, a \in \RR \rangle$, (iii)$\{F \in \Sh_+(\RR) | \Gamma_c(\RR;F) = 0\}$.

Here, for a set of objects $S \subseteq Obj(\Sh(\RR))$, $\langle S \rangle$ means the smallest subcategory closed under taking colimits containing $S$. 
\end{proposition}

\begin{proof}
Recall that the recollement $\Sh_-(\RR) \hookrightarrow \Sh(\RR) \twoheadrightarrow \sT$ 
decomposes objects in $\Sh(\RR)$, using the left adjoints, by the fiber sequence 
\[1_{\{t_2 \geq t_1\}} \rightarrow 1_{\Delta_\RR} \rightarrow 1_{\{t_2 > t_1\}}[1],\]
where $(t_1,t_2)$ is the coordinate for $\RR^2$. 

This implies that $\langle 1_{\RR_{\geq a}}, a \in \RR \rangle$ is contained in both $\{F \in \Sh_+(\RR) | \Gamma_c(\RR;F) = 0\}$ and $\sT$, since $1_{\RR_{\geq a}} = 1_{\{t_2 \geq t_1\}} \circ 1_{\{a\}}$. In addition, standard microsupport estimation, e.g., \cite[Proposition 3.2]{Kuo-wrapped-sheaves}, implies that $F \in \Sh_+(\RR)$ so $\sT \subseteq \Sh_+(\RR)$ as well.

To show that both inclusions are in fact equalities, we have to check that, for both cases, $F = 0$ if and only if $\Hom(1_{\RR_{\geq a}}, F) = 0$ by, e.g., \cite[I.1, Proposition 5.4.5]{Gaitsgory-Rozenblyum-DAG-I}. The latter is equivalent to 
\[\Gamma((-\infty,b);F) \xrightarrow{\sim} \Gamma((-\infty,a);F),\]
for all $a \leq b $. 

For any $F \in \Sh_+(\RR)$, the non-characteristic deformation lemma \cite[Proposition 2.7.2]{Kashiwara-Schapira-SoM} implies that $\Gamma((-\infty,a);F) \xrightarrow{\sim} \Gamma((c,a);F)$ for all $c < a$. In short, we conclude that, for $F \in \Sh_+(\RR)$ satisfying $\Hom(1_{\RR_{\geq a}}, F) = 0$ for all $a \in \RR$, we have
\[ \Gamma( (a - \epsilon, a + \epsilon); F) \xrightarrow{\sim} \Gamma( (a - \epsilon, a); F)\]
for all $a\in \RR$ and $\epsilon>0 $. So $F \in \Sh_-(\RR)$. Thus $F \in \Loc(\RR)$ is locally constant, which can only be $0$ if $F$ is also in $\sT$ or $\{F \in \Sh_+(\RR) | \Gamma_c(\RR;F) = 0\}$.
\end{proof}

\subsection{Symmetric monoidal structure  from convolution on \texorpdfstring{$\RR$}{}} \label{t is monoidal}
We recall some results from \cite{Tamarkin-non-displaceability, Guillermou-Schapira}.

Let $s: \RR \times \RR \rightarrow \RR$ be the sum.
Consider the map 
$\star: \Sh(\RR) \times \Sh(\RR) \rightarrow \Sh(\RR)$  which is given by the formula
\[ (G, F) \mapsto s_! (G \boxtimes F) = s_! (p_2^* G \otimes p_1^* F ).\]

This underlies a symmetric monoidal structure, with unit $1_{\{0\}}$.  More generally, if we write the 
translation map  on $\RR$ as $\nT_c(t)=t+c$, then 
one checks readily that $\nT_{c*} H \cong 1_{\{c\}} \star H$. 

For $a, b > 0$,  we note the formulas: 
\begin{equation} \label{R convolution formulas}
1_{(-a,a)}[1] \star 1_{(-b,b)}[1] = 1_{\left(-(a+b),a+b\right)}[1], \qquad \qquad \ 1_{[-a,a]} \star 1_{(-a,a)}[1] = 1_{\{0\}}.
\end{equation}

Using 6-functor formalism over $\mathbf{k}$, we can define a $\sHom^\star$ with the formula
\[\sHom^\star(F,G)=p_{2*}\sHom(p_1^{*}F,s^!G),\quad F,G\in \Sh(\RR),\]
such that, for $F,G,H\in \Sh(\RR)$, 
\[\Hom_{\Sh(\RR)}(F\star G,H)=\Hom_{\Sh(\RR)}(F,\sHom^\star(G,H)).\]
Then we find that $\Sh(\RR)$ is a closed symmetric monoidal category.

\begin{lemma} \label{lem: convolution estimates}
The following estimates hold (see \cite[Corollary 4.14]{Guillermou-Schapira} )
\begin{align*}
    \star : &\Sh(\RR)\times \Sh(\RR)_- \rightarrow \Sh(\RR)_-,\, \text{ or } \Sh(\RR)_-\times \Sh(\RR) \rightarrow \Sh(\RR)_-;\\
   \sHom^\star : &\Sh(\RR)^{op}\times \Sh(\RR)_- \rightarrow \Sh(\RR)_-,\, \text{ or } \Sh(\RR)_-^{op}\times \Sh(\RR) \rightarrow \Sh(\RR)_-.
\end{align*}
\end{lemma}

Lemma \ref{lem: convolution estimates} implies that
the closed symmetric monoidal structure determined by $\star$
descends to $\sT$.  Note that in $\sT$, we have 
$[1_{\RR_{\geq 0}}]\cong [1_{\{0\}}]$, giving two useful expressions for
the unit.  
Let us note also that the inclusion $\sT \hookrightarrow \Sh(\RR), [F] \mapsto F \star 1_{\RR_{\geq 0}}$ respects the monoidal structure $\star$.  Indeed: $ 1_{\RR_{\geq 0}}\star1_{\RR_{\geq 0}}=1_{\RR_{\geq 0}}$, so
\[([F]\star [G])\star 1_{\RR_{\geq 0}}=[F\star G] \star 1_{\RR_{\geq 0}}= F\star G \star 1_{\RR_{\geq 0}}=(F \star 1_{\RR_{\geq 0}}) \star (G\star 1_{\RR_{\geq 0}}).\]

\begin{lemma}\label{lemma: hom computation}For $F,G\in \sT$, we have
\[\Hom_{\sT}(F,\nT_{c*}G)=\Hom_{\Sh(\RR)}(F\star 1_{\RR_{\geq 0}},\nT_{c*}G)=\Gamma_{[-c,\infty)}(\RR,\sHom^{\star}(F,G)).\]
\end{lemma}
\begin{proof}
    The first isomorphism comes from \cite[(61)]{Guillermou-Schapira}, and the second comes from adjunction and $\Hom_{\Sh(\RR)}(F\star 1_{\RR_{\geq 0}},\nT_{c*}G)\simeq \Hom_{\Sh(\RR)}(F\star \nT_{-c*}1_{\RR_{\geq 0}},G)$.
\end{proof}
\begin{remark}\label{remark: k-linear hom}From now on, we will use $\Hom$, the $\mathbf{k}$-linear hom, for equivalently $\Hom_{\sT/\mathbf{k}}$ or $\Hom_{\Sh(\RR)/\mathbf{k}}$ for objects in $\sT \hookrightarrow \Sh(\RR)$. We also recall that  $\sHom^\star$ is by definition $\sHom_\sT$ for $\cA=\sT$.  
\end{remark}

\subsection{Filtered complexes}
\label{filtered complexes}

Consider the symmetric monoidal category $(\vec{\RR}, +)$ associated to the ordered group $\RR$  (there's a map $a \to b$ if $a \le b$).  

\begin{lemma} \label{vecr as sheaves}
    There is a symmetric monoidal functor 
    \[(\vec{\RR},+)  \rightarrow  (\Sh(\RR),\star) ,\qquad
      c \mapsto 1_{\RR_{\ge c}} .
    \]
\end{lemma}
\begin{proof}
    The content of the assertion is that
    there are natural isomorphisms
    $1_{\RR_{\ge a}}  \star 1_{\RR_{\ge b}}  = 1_{\RR_{\ge a + b}}$ and, when $d\geq c$, canonical maps $1_{\RR\geq c}\rightarrow 1_{\RR\geq d}$.  This is straightforward.  
\end{proof}

We may of course further compose with $(\Sh(\RR), \star) \to \sT$.  Pullback of Yoneda modules defines a 
functor 
\begin{equation} \label{equation: Yoneda module}
    \Gamma_{[\cdot, \infty)} :  \sT  \to \Fun(\vec{\RR}^{op}, \mathbf{k}), \quad 
    G  \mapsto  [c \mapsto \Hom(1_{\RR_{\ge c}}, G)].
\end{equation}
where the image carries the natural `Hopf algebra' monoidal structure on functors out of a monoidal category. 
We regard the target as a version of `$\RR$-filtered complexes'.

We choose the notation for this functor because:
\[\Gamma_{[c,\infty)}(\RR,-)\simeq \Hom_{\Sh(\RR)}(1_{\RR_{\geq c}},-):\Sh(\RR)\rightarrow \mathbf{k}.\]

\begin{remark} 
By the microlocal Morse lemma, 
$\Gamma_{[c,\infty)}(\RR,H)=0$ when $H\in \Sh_-(\RR)$.  Thus always 
$\Gamma_{[c, \infty)}(\RR, H)$ depends only on
the isomorphism class $[H] \in \sT$,  and
\[\Gamma_{[c, \infty)}(\RR, H) =  \Gamma_{[\cdot, \infty)} ([H])(c).\]
\end{remark}

Let us introduce notation for the categories
of $\RR$-filtered objects: 
\begin{equation} \label{R filtrations}
    Filt^+ := \Fun(\vec{\RR}, \mathbf{k}), \qquad \qquad 
Filt^- := \Fun(\vec{\RR}^{op}, \mathbf{k}),
\end{equation}
which have increasing and decreasing filtrations, respectively.   

\begin{definition}
    For a filtered complex $F \in Filt^{\pm}$ and an element $x \in F(s)$,\footnote{We allow ourselves the following standard abuse of language: for $M \in \mathbf{k}$, by `element of $M$' we mean `map from $1_{\mathbf{k}} \to M$.  An element `is zero' if the map factors through the zero object, etc.} we define the {\em persistence}:
\[Per(x) := \inf \{\epsilon \, | \, \mbox{the image of $x$ in $F(s \pm \epsilon)$ is zero}\}.\]
    
    We say $F \in Filt^{\pm}$ is a {\em persistence module} if every nonzero element has nonzero persistence.\footnote{In the literature, the term `persistence module' is used for various flavors of $\RR$-filtered complexes; here we reserve it for this kind.} We say $F \in Filt^{\pm}$ is {\em torsion} if every element has bounded persistence. 
\end{definition}

\begin{lemma}
     $\Gamma_{[\cdot, \infty)}$ embeds $\sT$ fully faithfully in the torsion persistence modules in $Filt^-$.  
\end{lemma}
\begin{proof}
This is a special case of \cite[\S 1.4]{Kashiwara-Schapira-persistent}.  
    Let us explain the origin of the persistence and torsion conditions.  

The embedding of Lemma \ref{vecr as sheaves} preserves filtered colimits, so any $F \in \Fun(\vec{\RR}^{op}, \mathbf{k})$ 
in the image of $\Gamma_{[\cdot, \infty)}$ preserves limits. The only nontrivial resulting condition is
\[\varprojlim_{\epsilon \to 0} F(x-\epsilon) = F(x).\]

This is precisely the condition of being a persistence module. 

Additionally, $\varinjlim_{c \to \infty} 1_{\RR_{\ge c}} = 0$, so such functors must satisfy \[\varprojlim_{N \to \infty} F(N) = 0.\]

This is the condition of being torsion. 
\end{proof}

\begin{remark}
The image
is the category of limit-preserving functors  $(\vec \RR \cup \infty)^{op} \to \mathbf{k}$. 
\end{remark}

\section{Some properties of Tamarkin categories} \label{tamarkin category properties}

Recall that for $U \subseteq T^*M$, we define 
\[\sT(U) := \Sh(M \times \RR) / \{F\, |\, 
\pi(i_+^{-1} ss(F)) \cap U  = \varnothing\},\]
where the maps are $ T^*(M \times \RR) \xleftarrow{i_+} J^1 M \xrightarrow{\pi} T^*M $, $(q,p,t,1) \mapsfrom (q,p,t) \mapsto (q,p)$.
Note that for 
$V \subseteq U$, there is an evident natural surjection 
$\sT(U) \to \sT(V)$.  
The notion of microsupport descends to define
a map $\underline{ss}$ from objects of $\sT(U)$ to closed coisotropic
subsets of $U$. The corresponding triangulated version $\mathcal{D}_U(M)$ (over a discrete ring $\FF$) was defined in \cite{Chiu-nonsqueezing,Zhang-capacities-Chiu-Tamarkin}. By results of \cite[Section 5]{Blumberg-Gepner-Tabuada-universal-hihger-K}, we have an equivalence of $\FF$-linear triangulated categories $h\sT(U)\simeq \mathcal{D}_U(M)$.

In this section we give various further properties of $\sT(U)$.  Many are taken from or are reformulations of results in \cite{Tamarkin-non-displaceability, Guillermou-Schapira, Guillermou-Kashiwara-Schapira, Guillermou-book}, or are otherwise known or obvious to experts. 

\subsection{\texorpdfstring{$\sT$-linearity of $\sT(U)$}{}}

For a subset $X \subseteq T^*(M \times \RR)$, let us 
write $\RR \cdot X$ for the union of all translates of $X$ along the base $\RR$ direction. 

For $U \subseteq T^*M$, let us write $\widetilde{U} \subseteq T^*(M \times \RR)$
for the conic saturation of $i_+(\pi^{-1}(U))$.  
Microsupports being conic, we have, for any $F \in \Sh(M \times \RR)$, 
\[\pi(i_+^{-1} ss(F)) \cap U = \varnothing \iff 
ss(F) \cap \widetilde{U} = \varnothing.\]

\begin{lemma} \label{tilde U translation}
$\widetilde{U} = \RR \cdot \widetilde{U}$.
\end{lemma}
\begin{proof}
    If we write $(q,t)$ for coordinates on $M \times \RR$ and $(q, p, t, \tau)$ for coordinates on $T^*(M \times \RR)$, then $\widetilde{U}=\{(q,p,t,\tau):(q,p/\tau)\in U,\tau>0\}$. 
\end{proof}

\begin{lemma} \label{convolution support estimate}   \cite[Proposition 3.13]{Guillermou-Schapira}
For any $A \in \Sh(\RR)$ and $F \in \Sh(M \times \RR)$
\[ss(A \star F) \subseteq \RR \cdot((T^*M\times ss(A) )\cap ss( F)) \subseteq \RR \cdot ss(F).\]
\end{lemma}

\begin{proposition} \label{t linear}
The  action of $(\Sh(\RR), \star)$ on $\Sh(M \times \RR)$
descends to an action of $\sT$ on $\sT(U)$. 
\end{proposition}
\begin{proof}
First let us check that the action of $(\Sh(\RR), \star)$
descends to an action on $\sT(U)$.  
We must show that 
for any $A \in \Sh(\RR)$, the functor $A \star (-): \Sh(M \times \RR) \rightarrow \Sh(M \times \RR)$ fixes the subcategory $\{F\, | \, ss(F) \cap \widetilde{U} = \varnothing\}$.  This is clear from Lemmas \ref{tilde U translation} and \ref{convolution support estimate}. 

Now let us check that the action on $\sT(U)$
factors through $\sT$.  We should show that 
if $A \in \Sh_{-}(\RR)$, then the action of $A$ on $\sT(U)$ is trivial.   Now recall that $A \in \Sh_{-}(\RR)$ if and only if $A = 1_{\{0\}}  \star A \xrightarrow{\sim} 1_{\RR_{> 0}}[1] \star A$ is an isomorphism.
For such $A$, we have \[A \star F = (1_{\RR_{> 0}}[1] \star A) \star F = 1_{\RR_{> 0}}[1] \star (A \star F),\] and we recall that $1_{\RR_{>0}}[1]$ sends $\Sh(M \times \RR)$ into $\Sh_{-}(M \times \RR)$.
\end{proof}

The same considerations show: 

\begin{proposition} \label{T linear retract}
    For $V \subseteq U$, the quotient map $\sT(U) \to \sT(V)$ and both its adjoints  are $\sT$-linear. 
\end{proposition}

\subsection{Tamarkin category as a sheaf category} \label{Tamarkin-coefficient-shaves}

As $\sT$ is a symmetric monoidal presentable stable category, we 
may consider the (symmetric monoidal presentable stable) categories  $\Sh(M; \sT)$ of sheaves with coefficients in $\sT$. 
As always we have a 6-functor formalism \cite{volpe_6_functor,Scholze-Six-Functor}.  We denote the resulting functors by 
 $\otimes_\sT, \sHom_{\sT}, {f}_*^{\sT},{f}^*_{\sT},{f}_!^{\sT}, {f}^!_{\sT}$.

\begin{proposition} \label{tamarkinsheaf}
As $\sT$-linear categories, we have $\sT(T^*M)\cong \Sh(M; \sT)$.   
\end{proposition}
\begin{proof} 
We first prove that two categories are equivalent as $\mathbf{k}$-linear categories.

We use the identification
$\Sh(M; \sT) = \Sh(M) \otimes \sT$  (see \cite[Corollary 1.3.1.8]{Lurie-SAG}, \cite[Theorem 7.3.3.9]{Lurie-HTT}, or \cite[Corollary 2.24]{volpe_6_functor}). 

The functor $- \mapsto \Sh(M)\otimes -$ preserves colimits, so
\begin{eqnarray*}
    \Sh(M) \otimes \sT & = &\Sh(M)\otimes \cof(\Sh_{-}(\RR) \hookrightarrow  \Sh( \RR)) \\ & \cong & \cof(\Sh(M)\otimes\Sh_{-}( \RR) \hookrightarrow \Sh(M)\otimes\Sh( \RR)) 
    \\ & = & \cof(\Sh(M)\otimes\Sh_{-}( \RR) \hookrightarrow \Sh(M \times \RR)) .
\end{eqnarray*}

It remains to verify that the essential image of the functor\[\Sh(M)\otimes\Sh_{-}( \RR)\rightarrow \Sh(M)\otimes\Sh( \RR) = \Sh(M\times \RR), \quad F\otimes G \mapsto F\boxtimes G\] is $\Sh_{-}(M\times \RR)$.  

Any object in $\Sh_{-}(M\times \RR)$ can be written
as a colimit of some $F_\alpha \boxtimes G_\alpha \in \Sh(M) \boxtimes \Sh(\RR)$.  We should show that the $G_\alpha$ can be chosen to be in $\Sh_-(\RR)$.  But since $\Sh_-( M\times \RR)$ is stable under  convolution with $1_{\RR_{>0}}[1]$, we have: 

\[\colim F_\alpha \boxtimes G_\alpha = 
(\colim F_\alpha \boxtimes G_\alpha) \star 1_{\RR_{>0}}[1] =
\colim F_\alpha \boxtimes ( G_\alpha \star 1_{\RR_{>0}}[1]).
\]

Now, we show that the equivalence is actually $\sT$-linear. Any arbitrary object in $\Sh(M)\otimes\sT$ can be written as a colimit of some $F_\alpha \boxtimes G_\alpha \in \Sh(M) \otimes \sT$. So $A\in \sT$, we have 
\[ 
A \star (\colim F_\alpha \boxtimes G_\alpha) =
\colim F_\alpha \boxtimes (   A\star G_\alpha).\qedhere
\]

\end{proof}

\begin{remark} \label{tamarkin-six-functors}
    For a continuous map $f:M\rightarrow N$, we set $\underline{f}=f\times \id_{\RR}$, and one can check $\underline{f}_*,\underline{f}^*,\underline{f}_!,\underline{f}^!$ descent to $\sT(T^*M)$.  One can check that, under the above isomorphim, ${f}_*^{\sT},{f}^*_{\sT},{f}_!^{\sT}, {f}^!_{\sT}$ correspond to these $\underline{f}_*,\underline{f}^*,\underline{f}_!,\underline{f}^!$.

    In \cite{Guillermou-Schapira}, a 
    closed symmetric monoidal structure $(\star, \sHom^\star)$ is defined on $\sT(T^*M)$  by tensor in the $M$ factor and convolution in the $\RR$ factor.  One can check that under
    the above isomorphism, it is carried to $(\otimes_\sT, \sHom_{\sT(T^*M)})$. Indeed,  the argument of Proposition \ref{tamarkinsheaf} also gives us that $\otimes_\sT \simeq \star $ as bifunctors. Then $\sHom_{\sT(T^*M)}\simeq \sHom^\star$ since they are right adjoints of $\otimes_\sT \simeq \star $. We will only use $\sHom_{\sT(T^*M)}$ in this article. We will also abuse $\sHom_{\sT(T^*M)}$ by $\sHom_{\sT}$ when $M$ is clear; but be careful to distinguish it with $\sHom_{\sT}$ for a $M=\pt$, where we also use $\sHom_{\sT}$.
\end{remark}
\begin{remark}\label{remark: different hom}In $\Sh(M;\sT)$, we emphasize that $\sHom_{\sT}$ means the internal hom sheaf of $\Sh(M;\sT)$ with respect to the $\star$-monoidal structure on $\Sh(M;\sT)$. In particular, for $F,G\in \Sh(M;\sT)$, we have $\sHom_{\sT}(F,G)\in \Sh(M;\sT)$. On the other hand, as a $\sT$-linear category, $\Sh(M;\sT)$ suppose to have a $\sT$-valued relative hom $\Hom_{\Sh(M;\sT)/\sT}(F,G)$ in terms of notation of Section \ref{subsection: Extending coefficient}, which is computed by $a_{M*}^\sT\sHom_{\sT}(F,G) \in \sT$ by 6-functor yoga. 

Recall Remark \ref{remark: k-linear hom}, we denote $\Hom=\Hom_{\sT}$ the $\mathbf{k}$-linear hom of the $\mathbf{k}$-linear Tamarkin categories. 

In summary, for $F,G\in \sT(T^*M)$ and $c\in \RR$, by $\sT$-linear 6-functor yoga, the precious relation of various hom objects are given
\begin{equation}\label{equiation: hom computation on T*M}
    \Hom(F,\nT_{c*}G)=\Gamma_{M\times [-c,\infty)}(M\times \RR,\sHom_{\sT}(F,G))=\Gamma_{[-c,\infty)}(\RR,a_{M*}^\sT\sHom_{\sT}(F,G)).
\end{equation}
\end{remark}

\begin{remark}
    Another way to read Proposition \ref{tamarkinsheaf} is 
    that for $F \in \Sh(M; \sT)$, there is a natural {\em nonconic} microsupport $\underline{ss}(F) \subseteq T^*M$. In this language, $\sT(U)$ is the quotient of $\Sh(M;\sT)$ by 
    \[ \Sh_{U^c}(M;\sT) \coloneqq \{ F \in \Sh(M;\sT) | \underline{ss}(F) \subseteq U^c\}.\]
\end{remark}

\begin{corollary} \label{tamarkin-category-product}
Consider the map $s: M \times \RR \times N \times \RR \to M \times N \times \RR$ given by $s(x, t_1, y, t_2) = (x, y, t_1 + t_2)$.  Then
the equivalence $\sT(T^* M) \otimes_\sT \sT(T^* N) \simeq \sT(T^* M \times T^* N)$ is given by 
\[\sT(T^* M) \otimes_\sT \sT(T^* N)  \to \sT(T^* M \times T^* N) ,\quad
(F,G) \mapsto s_! (F \boxtimes G).\]
\end{corollary}

\begin{proof}
By Proposition \ref{tamarkinsheaf}, it is clear that 
\begin{equation*}\label{equation: Kunneth T-linear sheaves}
  \begin{tikzcd}
\sT(T^* M) \otimes_\sT \sT(T^* N) \arrow[d, "="] \arrow[rrr, "\simeq"] &                                                                     &                                                & \sT(T^* M \times T^* N) \arrow[d, "="] \\
\Sh(M;\sT) \otimes_\sT \Sh(N;\sT) \arrow[r, "\simeq"]                                 & \Sh(M)\otimes\Sh(N) \otimes\sT \otimes_\sT  \sT \arrow[r, "\simeq"] & \Sh(M\times N)\otimes \sT \arrow[r, "\simeq"] & \Sh(M\times N;\sT) .                        
\end{tikzcd}
\end{equation*}

Here, we want to show that this equivalence is the one given in the statement of the Corollary.

By definition of $\sT$-linear tensor product \cite[Construction 4.4.2.7]{Lurie-HA}, the equivalence in second row the commutative diagram is induced by the following bar construction
\[\Sh(M;\sT) \otimes_\sT \Sh(N;\sT)\coloneqq\varinjlim\nolimits_{n \geq 0} \Sh(M;\sT) \otimes \sT ^ { \otimes n } \otimes  \Sh(N;\sT)  \xrightarrow{\simeq} \Sh(M\times N;\sT).\]

It remains to identify the augmentation functor $\Sh(M;\sT)  \otimes  \Sh(N ;\sT) \rightarrow \Sh(M\times N;\sT)$.
We naturally identify $\Sh(M;\sT) \simeq \Sh(M) \otimes \sT$. The argument of Proposition \ref{tamarkinsheaf} identifies the augmentation $\Sh(M;\sT)  \otimes  \Sh(N ;\sT) \rightarrow \Sh(M\times N;\sT)$ with 
\[\Sh(M) \otimes \sT    \otimes  \Sh(N)\otimes\sT =\Sh(M\times N)\otimes\sT \otimes \sT\xrightarrow{\Sh(M\times N)\otimes \star} \Sh(M\times N)\otimes\sT\]
via the the monoidal structure on $\sT$, i.e. the $\star$-convolution $\star:\sT\otimes \sT\rightarrow \sT$. It is clear now that the augmentation functor $\Sh(M;\sT)   \otimes  \Sh(N ;\sT) \rightarrow \Sh(M\times N;\sT)$ is the functor $F\boxtimes G \rightarrow s_! (F \boxtimes G)$, which finish the proof.\end{proof}

\subsection{Dualizability}

\begin{proposition}\label{t dualizable}
$\sT(U)$ is dualizable as $\sT$-linear category.
\end{proposition}
\begin{proof}
The case of $U = T^*M$ is a special case of Proposition \ref{dualizability of sheaves} (as explained in Remark \ref{remark: compactly generated}), after the identification $\sT(T^*M) =\Sh(M;\sT)$ of Proposition \ref{tamarkinsheaf}. 

We may then deduce the result for general $U$, via Proposition \ref{T linear retract} and Lemmas \ref{duality-of-retraction} and \ref{closed idempotent splits}. 
\end{proof}

\subsection{Some compatibilities of \texorpdfstring{$\sT$}{}-linear kernels}\label{subsection: k-linear vs T-linear}

\begin{proposition}\label{prop: regular represtation}Denote $m(t_1,t_2)=t_2-t_1$ the subtraction map. We have the following:
\begin{enumerate}[fullwidth]
    \item For $K,F\in \Sh(\RR)$, we have $K \star F = (m^* K) \circ F$;
    \item The functor $m^*: (\Sh( \RR),\star) \rightarrow (\Sh(\RR   \times \RR),\circ)$ is monoidal: for any $G, F \in \Sh( \RR)$, we have
\[m^* (G \star F) = (m^* G) \circ (m^* F);\]

    \item The functor $m^*$ is identified with the regular representation of $\Sh(\RR)$:
\[\rho_{\Sh(\RR)}:\Sh(\RR)\rightarrow \Fun^L(\Sh(\RR),\Sh(\RR)).\]
\end{enumerate}
    
\end{proposition}

\begin{proof} 
    \begin{enumerate}[fullwidth]
        \item By the definition, $(m^* K) \circ F  = {q_{2}}_! ( m^* K \otimes q_{1}^*F)$ for projections $q_1,q_2:\RR^2\rightarrow \RR$. 
However, we have $(q_1, m )^{-1}=(q_1, s)$ and $q_{2} =s \circ (q_1, m ) $. Thus, we have 
\[
(m^* K) \circ  F
= s_! (q_1, m )_! ( m^* K \otimes q_{1}^*F) 
= s_! (q_1,s)^* (m^* K \otimes q_{1}^*F)= s_!(q_{2}^*K \otimes q_{1}^* K)  = K \star F.
\]

\item By the definition, $m^* (G \star F) = m^* s_! (p_2^* G \otimes p_1^* F )$. From the pullback diagram
\[
\begin{tikzpicture}
\node at (0,1.5) {$\RR^3$};
\node at (4,1.5) {$\RR^2$};
\node at (0,0) {$\RR^2$};
\node at (4,0) {$\RR$};
\draw [->, thick] (0.4,1.5) -- (3.7,1.5) node [midway, above] {$m \times id$};
\draw [->, thick] (0.4,0) -- (3.7,0) node [midway, above] {$m$};

\draw [->, thick] (0,1.2) -- (0,0.3) node [midway, right] {$\id \times s$}; 
\draw [->, thick] (4,1.2) -- (4,0.3) node [midway, right] {$s$};
\end{tikzpicture}
\] 
we have that $m^* (G \star F)  = (\id \times s)_! (m \times \id)^* (p_2^* G \otimes  p_1^* F) = (\id \times s)_! (q_3^* G \otimes q_{12}^* m^* F )$,
where we use the equality $p_1 \circ (m \times \id) = m \circ q_{12} = \left[ (t_1, t_2, t_3) \mapsto t_2-t_1  \right]$.
Next, we notice that $q_{13} = (\id \times s) \circ (q_1,q_2, m \circ q_{23}) = \left[ (t_1, t_2, t_3) \mapsto  (t_1, t_2, t_3-t_2 ) \mapsto (t_1,t_3) \right]$, and thus
\begin{align*}
m^* (G \star F) 
&= (\id \times s)_! (q_1,q_2, m \circ q_{23})_! (q_1,q_2, m \circ q_{23})^* (q_3^* G  \otimes q_{12}^* m^* F) \\ 
&= {q_{13} }_! (q_1,q_2, m \circ q_{23})^* (q_3^* G \otimes q_{12}^* m^* F) \\
&= {q_{13} }_!  ({q_{23}}^* m^* G \otimes q_{12}^* m^* F) =  (m^* G) \circ (m^* F),
\end{align*}
where we use the fact that $q_3 \circ (q_1,q_2, m \circ q_{23}) = m \circ q_{23}$ for the second to last equality.

\item This is just a combination of the first two terms and Proposition \ref{Kernel-functor corresponding}.\qedhere
    \end{enumerate}
\end{proof}

\begin{corollary}\label{corollary: m^r_T fully-faithful}The regular representation $\rho_{\Sh(\RR)}$ is fully-faithful.  
\end{corollary}  
\begin{proof}The regular representation $\rho_{\Sh(\RR)}$ is identified with $m^*$. However, the subtraction map $m$ is a fiber bundle with contractible fibers $\RR$, so $m^*$ is fully-faithful (see e.g. \cite[Proposition 2.7.8]{Kashiwara-Schapira-SoM}).    
\end{proof}
\begin{remark}\label{remark: non-fully faithful regular rep example}For a commutative topological group $(G,+)$, the same discussion about the convolution monoidal structure on $\Sh(\RR)$ shows that $\Sh(G)$ is dualizable symmetric monoidal, and the regular representation is given by $\rho_{\Sh(G)}=m^*$. However, it is not true that $\rho_{\Sh(G)}$ is fully-faithful in general. For example $G=S^1$, one can check by hand that $\rho_{\Sh(S^1)}$ is not fully-faithful.  
\end{remark}

We have the following compatibility.

\begin{lemma}\label{lemma: star-kernel vs circ-kernel}
    The $\star$-integral kernel representation
    \[\Sh(M \times N \times \RR)\to  \Fun^L_{\Sh (\RR)}(\Sh(M;\Sh (\RR)),\Sh(N;\Sh (\RR)))\] factors through the usual integral kernel
    representation \[\Sh(M \times \RR \times N \times \RR) \to \Fun^L(\Sh(M\times \RR),\Sh(N\times \RR))\] via pullback
    along the subtraction map $m: \RR \times \RR \to \RR$, and is fully-faithful.

    \begin{equation*}
    \begin{tikzcd}
\Sh(M\times N\times \RR) \arrow[rr, "\simeq"] \arrow[d, "m^*"] &  & {\Fun^L_{\Sh (\RR)}(\Sh(M;\Sh (\RR)),\Sh(N;\Sh (\RR)))} \arrow[d] \\
\Sh(M\times \RR\times N\times \RR)  \arrow[rr, "\simeq"]       &  & {\Fun^L(\Sh(M\times \RR),\Sh(N\times \RR))} ,                      
\end{tikzcd}
\end{equation*}
\[\begin{tikzcd}
\cK \arrow[d, maps to] \arrow[r, maps to] & {[F \mapsto \cK \circ_{\Sh(\RR)} F \coloneq  q_{2!}(\cK\otimes_{\Sh(\RR)} q_1^*F)]} \arrow[d, maps to] \\
K={m^*\cK} \arrow[r, maps to]               & {[F \mapsto {m^*\cK}  \circ F{\coloneqq}  p_{2!}(m^*\cK  \otimes p_1^*F)].}                                             
\end{tikzcd}\]

\end{lemma}
\begin{proof}This is a example of Lemma \ref{lemma: restriction fully-faithful-PrL}. Here, we take $X=\Sh(M)$ and $Y=\Sh(N)$, and $\cA=\Sh(\RR)$. In particular, we have known that the regular representation $\rho_{\Sh(\RR)}$ is fully-faithful, then the $\star$-integral kernel/usual integral kernel representation comparison functor is fully-faithful.  
\end{proof}

\begin{proposition}
The $\star$ action
\[\Sh(M\times N\times \RR) \otimes_{\Sh(\RR)} \Sh(M\times \RR) \rightarrow \Sh(N\times \RR),\quad K\otimes_{\Sh(\RR)} F\mapsto K\circ_{\Sh(\RR)} F\]
descends to the standard convolution action  defined by $\sT$-linear 6-functors. 
\[\Sh(M\times N;\sT) \otimes_\sT \Sh(M;\sT) \rightarrow \Sh(N;\sT), \quad K\otimes_\sT F\mapsto K\circ_\sT F,\]
\end{proposition}
\begin{proof}
    This follows from \cite[Proposition 3.13]{Guillermou-Schapira} as recalled above as Lemma \ref{convolution support estimate}.
\end{proof}

\begin{proposition}\label{proposition:fully-faithful of forgetful functor between functors}
The regular representation $\rho_\sT$ is fully-faithful. In particular, the $\sT$-valued $\star$-integral kernel/usual integral kernel representation comparison functor
\[\Fun^L_{\sT}(\Sh(M;\sT),\Sh(N;\sT))\rightarrow \Fun^L(\Sh(M;\sT),\Sh(N;\sT))\]
is fully-faithful.
\end{proposition}
\begin{proof}
We identify $\sT$ with a full-subcategory of $\Sh(\RR)$ using left-adjoint, then $\rho_\sT$ fits into the diagram
\[\begin{tikzcd}
\sT \arrow[r,"\rho_\sT"] \arrow[d, hook] & {\Fun^L(\sT,\sT)=\sT^\vee\otimes \sT} \arrow[d]      \\
\Sh(\RR) \arrow[r,"\rho_{\Sh(\RR)}"]      & {\Fun^L(\Sh(\RR),\Sh(\RR))=\Sh(\RR)\otimes \Sh(\RR)}.
\end{tikzcd}\]

The left vertical is fully-faithful by construction and $\rho_{\Sh(\RR)}$ is fully-faithful by Corollary \ref{corollary: m^r_T fully-faithful}. 

Then the right vertical is fully-faithful since it factor through fully-faithful functors
\[\sT^\vee\otimes \sT \rightarrow \Sh(\RR)\otimes \sT \rightarrow \Sh(\RR)\otimes \Sh(\RR).\]

Then for the second statement, we use Lemma \ref{lemma: restriction fully-faithful-PrL}.
\end{proof}

\begin{lemma}\label{lemma: product reduced microsupport}
For $F \in \Sh(M;\sT)$ and $G \in \Sh(N;\sT)$, we have  $\underline{ss}(F \boxtimes_{\sT} G) \subseteq  \underline{ss}(F) \times \underline{ss}(G)$. 
\end{lemma}
\begin{proof}
Viewing $\sT(U)$ as a quotient of $\Sh(M \times \RR)$, the statement is a combination of \cite[Proposition 5.4.1]{Kashiwara-Schapira-SoM} and \cite[Proposition 4.13]{Guillermou-Schapira}.
\end{proof}

\begin{example} \label{identity-functor-as-Tamarkin-kernel}
The constant sheaf on the diagonal $1_{\Delta_M}^\sT \in \Sh(M \times M; \sT)$ is the integral kernel for the identity functor $\sT(T^* M) = \Sh(M;\sT)$.  Note however that $1^\sT$ is the symmetric monoidal unit of $\sT$, and under the inclusion $\sT \hookrightarrow \Sh(\RR; \mathbf{k})$, we have $1^\sT \mapsto 1_{[0,\infty)}$, where the RHS $1$ denotes the monoidal unit of $\mathbf{k}$.  Correspondingly,
under $\Sh(M \times M; \sT) \hookrightarrow \Sh(M \times M \times \RR, \mathbf{k})$, we have $1_{\Delta_M}^{\sT} \mapsto 1_{\Delta_M} \boxtimes 1_{[0,\infty)}$.  The identity kernel often appeared under the latter guise in the previous literature. 
\end{example}

\subsection{Sheaf quantization of Hamiltonian isotopies}
We recall the results of \cite{Guillermou-Kashiwara-Schapira}. 
Let $Y$ be a manifold.  On the cotangent bundle, we choose the
exact symplectic form $\omega=d\lambda$ with $\lambda=pdq$ (this determines some signs).

Let $\dot{T}^*Y$ be the complement of the zero section in $T^* Y$.  Let $(I, 0)$ be a pointed interval. Consider an $\RR_{>0}$-equivariant $C^\infty$ symplectic isotopy
\[\phi: I \times \dot{T}^*Y \rightarrow \dot{T}^*Y,\]
which is the identity at $0 \in  I$. 
Such an isotopy is always the Hamiltonian flow for a unique $\RR_{>0}$-equivariant function $H:  I \times \dot{T}^*Y  \to \RR$. We identify $T^*(Y\times Y) = \overline{T^*Y}\times T^*Y$.

At fixed $z \in I$, we have the graph of $\phi_{z}$: 
\begin{equation}\label{equation: graph}
  \Lambda_{{\phi}_{z}}\coloneqq \left\lbrace ((q,-p),{\phi}_{{z}}(q,p) ) : (q,p) \in \dot{T}^*Y\right \rbrace\subseteq \dot{T}^*(Y\times Y) .
\end{equation}

As for any of Hamiltonian isotopy, we may consider the Lagrangian graph, which by definition is a Lagrangian subset $\Lambda_\phi \subseteq  T^*I \times  \dot{T}^*(Y\times Y) $
with the property that 
$ \Lambda_{{\phi}_{z_0}}$ is 
the symplectic reduction 
of $\Lambda_{{\phi}}$ along  $\{z=z_0\}$.  It is given by the formula: 
\begin{equation}\label{equation: totalgraph}
  \Lambda_{{\phi}}\coloneqq \left\lbrace (z, - H(z,{\phi}_{z}(q,p))  ,(q,-p),{\phi}_{z}(q,p) ) : z\in I, (q,p) \in \dot{T}^*Y \right \rbrace .
\end{equation}

For $F\in \Sh(Y)$, we set $\dot{ss}(F)\coloneqq ss(F)\cap \dot{T}^*Y$.

\begin{theorem}[{\cite[Theorem 3.7, Prop. 4.8]{Guillermou-Kashiwara-Schapira}}]\label{theorem: GKS} For $\phi$ as above, there is a sheaf $K=K({\phi}) \in \Sh( I\times  Y\times Y)$ such that $\dot{ss}(K)= \Lambda_{{\phi}}$ and $K|_{ \{0\}\times Y^2}\cong 1_{\Delta_{Y}}$.  
The pair $(K, K|_{\{0\}\times Y^2}\cong 1_{\Delta_{Y}})$ is unique up to unique isomorphism.

Moreover, 
for isotopies $\phi, \phi'$  generated by Hamiltonians
$H' \leq H$, there's a map $K(\phi')|_{\{1\} \times Y^2}  \to K(\phi)|_{\{1\} \times Y^2} $.  
In particular, when $H \geq 0$, then there is a map $1_{I \times \Delta_Y}|_{\{1\} \times Y^2} \to K(\phi)|_{\{1\}\geq \times Y^2}$. 
\end{theorem}
From general properties of microsupports, one has 
\begin{equation}\phi_t(\dot{ss}(F)) = \dot{ss}(  K_t\circ F).\end{equation}

\begin{remark}\label{remark: GKS theorem}
    The basic idea of the proof of Theorem \ref{theorem: GKS} is that
    (1) for sufficiently small positive $H$, 
    the locus $\Lambda_\phi$ is the conormal to the boundary of a neighborhood of the diagonal, and the corresponding $K$ is just the constant sheaf on the closed  neighborhood and (2) any $\phi$ can be obtained by composing $\phi$ as in (1) and their inverses.

To see the existence of the morphism $K(\phi')|_{\{1\} \times Y^2}  \to K(\phi)|_{\{1\} \times Y^2} $, one also reduces to the case when $\phi'$ is the identity so $H \geq 0$. We set $T^*I= I\times \RR_{\zeta}$ and $T_{\zeta \leq 0}^*I =I\times  \RR_{\zeta \leq 0}$. By Equation \eqref{equation: totalgraph}, we have that $K(\phi)\in \Sh_{T^*_{\zeta \leq 0}I}(I  \times Y \times Y)$ since $H\geq 0$. Then the canonical map comes from the property of the latter category (c.f. \cite[Proposition 4.8]{Guillermou-Kashiwara-Schapira}).

We will give a detailed account of the $\sT$-linear version of this statement in the proof of Corollary \ref{GKS tamarkin version} below.  
\end{remark}

\vspace{2mm}
To apply to non-conic situations, consider some manifold $M$.  We write coordinates $q$ on $M$, and $(q, p)$ on $T^*M$.  We write coordinates $(q,t)$ on $M \times \RR$ and $(q, p, t, \tau)$ on $T^*(M \times \RR)$. We identify $T^*(M\times \RR)$ with $\overline{T^*M}\times T^*\RR$. We consider the map 
\begin{eqnarray*}
    \rho: T^*M \times \dot T^* \RR & \to & T^*M \\
    (q, p, t, \tau) & \mapsto & (q, p/\tau)  .  
\end{eqnarray*}

For a smooth function $H$ with compactly supported derivative, denote $X_H$ the Hamiltonian vector field defined by $\iota_{X_{H}}\omega =-dH$, and let $\varphi: I\times T^*M \rightarrow T^*M$ be the isotopy generated by $X_H$. One can lift $\varphi$ to $\widehat{\varphi}: I\times\dot{T}^*(M\times \RR)  \rightarrow \dot{T}^*(M\times \RR) $:
\begin{proposition}[{\cite[Proposition A.6]{Guillermou-Kashiwara-Schapira}}]
Let $H: I \times T^*M \to \RR$ be a  function with compactly supported derivative, and 
$\varphi: I\times T^*M \rightarrow T^*M$ the corresponding Hamiltonian isotopy. 
Then $\varphi$ lifts along $\rho$ to some conic 
\[\widehat{\varphi}: I \times \dot T^*(M \times \RR) \to \dot T^*(M \times \RR).\] 

On the locus $I \times T^*M \times \dot T^* \RR$, i.e. where $\tau \ne 0$, the corresponding Hamiltonian function is
\[\widehat{H}:=\tau H(-,\rho(-)) : I \times T^*M \times \dot T^* \RR  \to \RR.\]  

The extension of $\varphi$ over $\tau = 0$ has the following property.  
Let
\begin{equation} \label{S H formula} S_{H}^z(q,p)=\int_0^z [\lambda(X_{H_s})-H_s]\circ \varphi^s_{H}(q,p)ds
\end{equation}
be the symplectic action function.  Then there exists
$v\in C^\infty(I)$
such that: 
\begin{equation}\label{equation: formula of hat-varphi}
    \begin{aligned}
     &\widehat{\varphi}(z,q,t,p,\tau)=(\tau\cdot\varphi(z,q,p/\tau), t-S_{H}^z(q,p/\tau),\tau), &\tau \neq 0,\\
&\widehat{\varphi}(z,q,t,p,0)=(q,p, t+v(z),0), &\tau =0.
   \end{aligned}
\end{equation}
\end{proposition}
We call this $\widehat{\varphi}$ the conification of $\varphi$. 

\begin{remark}
    The condition that $H$ has compact support serves to ensure that a certain differential equation characterizing $v(z)$ has a solution.
\end{remark}

\begin{corollary}\cite[Corollary 2.3.2.]{Guillermou-book} \label{GKS tamarkin version}
    Given a compactly supported Hamiltonian isotopy $\varphi: I \times T^*M \to T^*M$, 
    there is a unique $\cK(\varphi) \in \Sh\left(I, \sT\left(T^*(M\times M) \right) \right)$ such that 
    $\cK(\varphi)|_0 \cong 1_{\Delta_{M}}^\sT$, and  
    \begin{equation}\label{equatio: GKS T-linear total graph}
        \dot{ss}(\cK(\varphi))/\RR_+\subseteq \{(z, - H(z,{\varphi}_{z}(q,p))  ,(q,-p),{\varphi}_{z}(q,p),-S_{H}^z(q,p) ) : z\in I, (q,p)\in T^*M\}.
    \end{equation}

    Moreover, if $\varphi, \varphi'$ are generated by compactly supported Hamiltonians with $\varphi' \leq  \varphi$, then
    there is a map $\cK(\varphi')|_{1}\rightarrow \cK(\varphi)|_{1}$.
\end{corollary}
\begin{proof} 
We first apply the GKS theorem to $\phi=\widehat{\varphi}$ to obtain a sheaf $K(\widehat{\varphi})\in \Sh(I\times (M\times  \RR)^2)$. By the formula \eqref{equation: formula of hat-varphi}, we have $\tau'=-\tau$, then
\[\dot{ss}(K(\widehat{\varphi}))\subseteq \Lambda_{\widehat{\varphi}} \subseteq \{\tau+\tau'=0\}.  \]

By \cite[Proposition 5.4.5]{Kashiwara-Schapira-SoM}, for $m(z,q_1,t_1,q_2,t_2)=(z,q_1,q_2,t_2-t_1)$, we have $K(\widehat{\varphi})\cong m^{*}m_*K(\widehat{\varphi})$. Then we can take $\cK(\varphi)$ as the image of $m_*K(\widehat{\varphi}) \in \Sh(I\times M^2 \times \RR)$ under the natural functor $\Sh(I\times M^2 \times \RR)\simeq \Sh(I;\Sh( M^2 \times \RR))\rightarrow \Sh(I;\Tam( M^2 ))$.

If $\varphi' \leq \varphi$, we cannot use Theorem \ref{theorem: GKS} directly since $\widehat{\varphi'} \leq \widehat{\varphi}$ is not true on whole $\dot{T}^*(M\times \RR)$. Therefore, we embedding $\Sh\left(I, \sT\left(T^*(M^2) \right) \right)\hookrightarrow \Sh(I\times M^2\times \RR)$. Under this identification, we have $\cK(\varphi)\simeq \cK( {\varphi})\star 1_{\RR_{\geq 0}}$. Then we have that $ss(\cK(\varphi)) \subseteq ss(\cK( {\varphi}))\cap \{\tau\geq 0\} $ by Lemma \ref{convolution support estimate}. In the rest part of the proof, we always think of $\cK(\varphi)$ as a sheaf in $\Sh(I\times M^2\times \RR)$ under the fully embedding.

We first reduce to the case that $\varphi'=\id$. For an open interval $J$ containing $[0,1]$, let us consider the $2$-parameter family of Hamiltonian $\Phi: J\times I \times T^*M \to T^*M$ defined by $\Phi_{s,t}(q,p)=\varphi(st,q,p)$. Then we apply the $2$-parameter family of GKS theorem (c.f. \cite[Remark 3.9]{Guillermou-Kashiwara-Schapira}) to $\widehat{\Phi}$, and the previous discussion provides us an sheaf $\cK(\Phi)\in \Sh(J\times I;\Tam( M^2 ))  \hookrightarrow \Sh(J\times I\times M^2 \times \RR)  $ . It satisfies the $\cK(\Phi)_{s = 0} = 1_{I\times \Delta_{M}\times \RR_{\geq 0}}$, $\cK(\Phi)_{s = 1} = \cK(\varphi)$ and $\cK(\Phi)\simeq \cK({ {\Phi}} )\star 1_{\RR_{\geq 0}}$. Importantly, $\cK(\Phi)$ admits the following microsupport estimation similar to $ {\varphi}$
 \[\dot{ss}(\cK(\Phi))/\RR_+\subseteq \{(s,-zH,z, -sH  ,(q,-p),{\varphi}_{sz}(q,p),-S_{H}^{sz}(q,p) ) : (s,z)\in J\times I, (q,p)\in T^*M\}.\]

 We refer to \cite[Proposition 6.1.C]{Polterovich_Geometry_SympDiff} for the computation of of the $T^*(J\times I)$-component. 

Consequently, for an open interval $I_0\subset I\cap (0,\infty)$, we have the microsupport constraint $ss(\cK(\Phi)|_{I_0})\subseteq {T^*_{\leq 0}J} \times T^*(I_0 \times M^2\times \RR)$ follows from the formula above. Then the morphism $\cK(\varphi')|_{I_0}\rightarrow \cK(\varphi)|_{I_0}$ follows from applying \cite[Proposition 4.8]{Guillermou-Kashiwara-Schapira} or \cite[Proposition 3.2]{Kuo-wrapped-sheaves} to $\cK(\Phi)|_{I_0}$, which descends to a morphism in $\Sh\left(I_0, \sT\left(T^*(M^2) \right)\right)$. In particular, we can take $I_0=(1-\epsilon,1+\epsilon)$, and the restriction to $1$ gives the required morphism.
\end{proof}

\begin{corollary}For each $z\in I$, 
    there is an equivalence $\cK({\varphi})|_z \circ_\sT -: \sT(U)\xrightarrow{\sim} \sT({\varphi}_z(U))$.     
\end{corollary}
\begin{proof}
Note that
$\cK({\varphi})|_z \circ_\sT -: \sT(T^*M)\rightarrow \sT(T^*M)$
defines an equivalence, as its inverse is given by 
$\cK({\varphi}^{-1})|_z$.  It remains to check how this acts
on microsupports.  We set
\begin{equation}\label{equation: slicewise graph, non homogeneous case}
   \hatLam_{\varphi_z}\coloneqq \{((q,-p),{\varphi}_{z}(q,p),-S_{H}^z(q,p) ) :  (q,p)\in T^*M\}. 
\end{equation}

Then $\dot{ss}(\cK(\varphi)|_{z})/\RR_+\subseteq \hatLam_{\varphi_z}$ and 
and so
$\underline{ss}(\cK({\varphi})|_z \circ_\sT F)= \varphi_z(\underline{ss}(F))$.
\end{proof}

In fact, as explained in \cite[Proposition 4.18]{JunZhang}, a similar argument in families shows that if ${\varphi}_z$ is fixed, then $\cK({\varphi})|_z$ depends on the relative homotopy class of the path $[s\in [0,z] \mapsto \varphi_s]$.

Finally, let us recall the following lemma:
\begin{lemma}[{\cite[Lemma 4]{gammasupport-as-micro-support}}]\label{lemma: sheaves fixed by GKS}For $F\in \Sh(M,\sT)$ and a compactly supported Hamiltonian function $H: I\times T^*M\rightarrow \RR$. If $\supp(H_s)\cap \underline{ss}(F) =\varnothing$, then $F {\simeq} \cK({\varphi^H})|_1 \circ_\sT F$. When $H$ is non-negative, the equivalence is induced by the continuation map $1_{\Delta_{M}}^\sT \to  \cK(\varphi^{H})|_1$.    
\end{lemma}
\begin{proof}Let us explain its proof here for our later applications. We define $G=\cK(\varphi^{H_\alpha})\circ_\sT F$ in $ \Sh(M\times I,\sT)$. Then $(s,\sigma,q,p)\in \underline{ss}(G)$ if there exists $(q',p')\in T^*M$ such that $(s,\sigma,q,p)=(s,-H(s,\varphi^H(s,q',p')),\varphi(s,q',p'))$, and then $\sigma=0$ under the condition. So, we have $F=G|_{s=0}\xrightarrow{\simeq} G|_{s=1}=\cK({\varphi^H})|_1 \circ_\sT F$ by \cite[Prop. 5.4.5]{Kashiwara-Schapira-SoM}. When $H\geq 0$, the statement is not included in \cite{gammasupport-as-micro-support}, however it is clear that the morphism is induced from the continuation map. 
\end{proof}

\section{Trace of the Tamarkin category} \label{section trace of tamarkin}

For $U \subseteq T^*M$, we have by now shown that $\sT(U)$ is $\sT$-linear (Prop. \ref{t linear}) and dualizable over $\sT$ (Prop. \ref{t dualizable}).  Thus we have a well defined element
\[\Tr(1_{\sT(U)}) \in \sT\]

We introduce the notation 
\[P_U : \sT(T^*M) \to \sT(T^*M)\]
for the projector with image $\sT(U)$.  
Recalling that 
$\sT(T^*M) = \Sh(M, \sT)$, the projector
$P_U$ can be expressed via an 
integral kernel $P_U \in \Sh(M \times M, \sT)$.  Given the projector in such a form, we have, by Lemma \ref{duality-of-retraction} and Corollary \ref{trace of sheaves}: 
\begin{equation} \label{trace of projector diagonal}
    \Tr(1_{\sT(U)},{\sT(U)}) = 
    \Tr(P_U,\Sh(M, \sT)) = 
    \Gamma_c^\sT(M, \Delta_\sT^* P_U) \in \sT,
\end{equation}
where $\Delta: M \to M \times M$ is the inclusion of the diagonal. For example, we have $\Tr(1_{\sT(T^*M)})\simeq \Gamma_c^\sT(M, 1^{\sT}_M)\simeq \Gamma_c(M, 1_{M})\otimes 1^{\sT}$.  (Recall from 
Section \ref{Tamarkin-coefficient-shaves} that we use the $\sT$ sub- and super-scripts to remind the reader we are using the $\sT$-linear six operations.) 

In short, the problem of computing 
the trace of $\sT(U)$ is reduced to that of explicitly expressing the projector $P_U$ as an integral kernel. 

We will also be interested in the projector $Q_U$ to the full subcategory of $\Sh_{U^c}(M, \sT)$.  For formal reasons, there is a fiber sequence 
\[P_U \to 1_{\Delta_{M}}^\sT \to Q_U .\]

Then we have an fiber sequence of traces
\begin{equation}\label{equation:fiber sequence for trace}
    \Tr(1_{\sT(U)}) \rightarrow  \Tr(1_{\sT(T^*M)}) \rightarrow\Tr(1_{\Sh_{U^c}(M, \sT)}).
\end{equation}

In fact, much work has already been done
on expressing the projectors 
and computing the RHS of 
Equation \eqref{trace of projector diagonal} \cite{Chiu-nonsqueezing, Zhang-capacities-Chiu-Tamarkin, Zhang-S1-Chiu-Tamarkin, Kuo-wrapped-sheaves}.  In this section we recall these results and further develop related ideas.

\subsection{Projector via Fourier transform and cutoff} \label{fourier projector}

Given a sheaf $F$ on a manifold $M$, 
one can `cut off' the support of the sheaf to some closed $c: C \subseteq M$
by e.g. $F \mapsto c_* c^* F$.  
The classical `devissage' arguments
in sheaf theory amount to the 
fact that $c_* c^*$ is the projector 
associated to $\Sh(C) \to \Sh(M)$.  

In favorable situations, one can perform a `microlocal cutoff', for instance by composing cutoffs and Fourier transform.  
Such a cutoff for the Tamarkin category was constructed in \cite{Chiu-nonsqueezing} when $U$ is a ball; in fact the method works more generally:

\begin{theorem}[{\cite[Proposition 2.8]{Zhang-capacities-Chiu-Tamarkin}}]Let $\varphi^H$ be a complete Hamiltonian flow on $T^*X$ with a Hamiltonian function $H$. Assume that there exists a sheaf quantization, i.e. some \[\cK(\varphi^H) \in \Sh(\RR_z,\sT(T^*(M\times M)))\] such that $\cK(\varphi^H)|_0=1_{\Delta_{M}}^\sT$ and $\underline{ss}(\cK(\varphi^H))$ is contained in the Lagrangian graph of $\varphi^H$. If we further assume that, for all $\zeta>0$, the level set $\{H=\zeta\}$ is compact, then for the open set $U=\{H>0\}$, we have that the fiber sequence $
P_U  \rightarrow 1_{\Delta_{M}}^\sT \rightarrow  Q_U   $
is isomorphic to
\[{\cK(\varphi^H)}\circ_\sT 1_{\{t+z\zeta\geq 0\}}[1] \circ [1_{\RR_{>0}}\rightarrow 1_\RR \rightarrow 1_{\RR_{\leq 0}} ].\]
\end{theorem}

\begin{remark}
Let us explain the idea of the theorem. 
Recall that $\underline{ss}({\cK(\varphi^H)})$ is bounded by the Lagrangian graph of $\varphi^H$, i.e.
\begin{equation}
\underline{ss}({\cK(\varphi^H)})\subseteq \{z, -H(q,p), q, -p,\varphi_z(q, p)):(z,q,p)\in  \RR\times T^*X \}.    \end{equation}

We want to cut off the microsupport of $\cK(\varphi^H)$ on $\zeta$-variable. 
If we use the Fourier-Sato-Tamarkin 
transform to the $z$-variable, i.e. $\widehat{\cK(\varphi^H)}=\cK(\varphi^H)\circ_\sT 1_{\{t+z\zeta\geq 0\}}[1]\in \Sh(\RR_{\zeta},\Tam(T^*(M\times M))$, then, by \cite[Theorem 1.14.]{Zhang-capacities-Chiu-Tamarkin}, we have 
\begin{equation}\label{fourier transform of sheaf quantization}
\underline{ss}(\widehat{\cK(\varphi^H)})\subseteq \{ (H(q,p),z, q, -p,\varphi_z(q, p)):(z,q,p)\in  \RR\times T^*X \}.    
\end{equation}

Now composition with $1_{\RR_{\zeta >0}} \in \Sh(\RR_\zeta)$ produces the desired element of $\sT(T^*(M \times M))$. 
\end{remark}

For Hamiltonian functions $H$ with compactly supported derivative, 
existence of the sheaf quantization follows from 
Cor. \ref{GKS tamarkin version} above.  Chiu constructed such a sheaf quantization of $H(q,p)=(q^2+p^2)/2$ on $\RR^{2n}$ \cite{Chiu-nonsqueezing}. 
We do not know a general result on existence of sheaf quantization 
for not compactly supported Hamiltonian.

\subsection{Projector via wrapping} \label{wrapping projector}

Motivated by ideas of \cite{Nadler-pants, Ganatra-Pardon-Shende3}, it was shown in \cite{Kuo-wrapped-sheaves} that for any closed set $ X \subseteq S^*M$, the left adjoint $\iota^*$ to the inclusion $\iota_*: \Sh_X(M) \hookrightarrow \Sh(M)$ can be computed `by wrapping'. More precisely, 
 
  \begin{theorem}[{\cite[Thm. 1.2]{Kuo-wrapped-sheaves}}] \label{chris thesis}
      If $H_\alpha$ is any increasing sequence of positive compactly supported Hamiltonians supported on $S^*M \setminus X$ such that $H_\alpha \to \infty$ pointwise in this locus.  Then 
 \[\iota^* F = \varinjlim (K(\phi_{H_\alpha})|_1 \circ F) = (\varinjlim K(\phi_{H_\alpha})|_1) \circ F.\]
 That is, $\varinjlim K(\phi_{H_\alpha}) \circ $
 is left adjoint to  $\iota_*: \Sh_X(M) \hookrightarrow \Sh(M)$. 

 Moreover, the unit of the adjoint is given by the map $1_{\Delta_M} \to \varinjlim K(\phi_{H_\alpha})|_1$, which is induced by the continuation map $1_{\Delta_M} \to  K(\phi_{H_\alpha})|_1$ defined by positivity of $H_\alpha$.
  \end{theorem}

  \begin{remark}\label{remark: wrapping formula for kernel}
  Let us give the idea of the proof. 
  By compact support of $M$, any two such sequences of such $H_\alpha$ can be interleaved, hence 
have the same colimit. 
In particular, 
it follows from this that 
$K_X=\varinjlim K(\phi_{H_\alpha})|_1$
is idempotent and the equivalence $K_X\xrightarrow{\simeq} K_X\circ K_X$ is induced by the continuation map. Next, the kernels $K(\phi_{H_\alpha})$ 
preserve $\Sh_X(M)$ by the same discussion of Lemma \ref{lemma: sheaves fixed by GKS}, and we need to check that 
$ ss( K_X \circ F) ) \subseteq X$, and moreover it 
suffices to do some for one family $H_\alpha \to \infty$.  
In fact, it moreover suffices to construct
one such family for each point of $T^*M \setminus X$
to witness that $K_X \circ F$ has no microsupport at this point.  
Such a sequence of Hamiltonians is described in \cite{Kuo-wrapped-sheaves}. Then we have that the functor $K_X \circ $ is the projector with image $\Sh_X(M)$, and being left adjoint by \cite[Proposition 5.2.7.4]{Lurie-HTT}.

(The above argument differs from \cite{Kuo-wrapped-sheaves} in that in said reference, the colimit is taken over 
an $\infty$-categorical `wrapping category'.  For the purposes here, one can compute the colimit by mapping telescope. The constructions agree: It is explained in \cite[Lemma 3.31]{Kuo-wrapped-sheaves} that such a family $H_\alpha $ is cofinal in the ``wrapping category", and one can compute the sequential colimit by mapping telescope because the inclusion of simplicial sets $\mathbb{N}\rightarrow N(\mathbb{N})$ is cofinal.)
\end{remark}

We have
the following corresponding $\sT$-linear result for Tamarkin categories.
\begin{proposition}\label{wrapping formula for projector}
    Let $U \subseteq T^*M$ be an open set, and let $H_\alpha$ be any increasing sequence of compactly supported Hamiltonians which are supported on $U$ and such that $H_\alpha(u) \to \infty$ for all $u \in U$.  Let $\cK({\varphi^{H_\alpha}})|_1$ be their sheaf quantizations from Corollary \ref{GKS tamarkin version}, which form a directed system
    along continuation maps.  We define 
    \begin{equation} \label{qu} Q'_U := \varinjlim\nolimits_{\alpha}\cK(\varphi^{H_\alpha})|_1,
\end{equation}
    then $Q_U \cong Q'_U$,
and the morphism $1_{\Delta_{M}}^\sT \rightarrow Q_U$ is intertwined with the limit of continuation maps $1_{\Delta_{M}}^\sT \to  \cK(\varphi^{H_\alpha})|_1$.
\end{proposition} 
\begin{proof}
We cannot directly apply Theorem \ref{chris thesis}, because $\widehat{\varphi^{H_\alpha}}$ is not compactly supported.   However, 
we can adapt the ideas of the proof, as recalled in Remark \ref{remark: wrapping formula for kernel}, to the present setting.  

Recall the following general fact.  Suppose given a category $C$, endofunctor $F: C \to C$, and natural transformation $\eta: 1_C \to F$ such that $F \circ \eta: F \to F^2$ is an equivalence.  Then $F$ is the projector onto the full subcategory  $\{x | \eta: x \xrightarrow{\sim} F(x)\} \subset C$ (see e.g. \cite[Proposition 5.2.7.4]{Lurie-HTT}).   
We will apply this to the endofunctor $Q'_U \circ_\sT$, and the natural transformation induced from continuation maps.  To establish the result of the proposition, it then suffices to show that (i) $Q'_U \circ_\sT (1 \to Q'_U)$
gives an isomorphism $Q'_U \xrightarrow{\sim} Q'_U \circ_\sT Q'_U$, and (ii) 
$F \to Q'_U \circ_\sT F$ is an isomorphism iff $F \in \Sh_{U^c}(M;\sT)$. 

To check (i), note that any two sequences of such (compactly supported!) $H_\alpha$ can be interleaved, hence have the same colimit. Therefore,
\[ Q_U'\circ Q_U'=\varinjlim\nolimits_{\alpha,\beta}\cK( {\varphi^{H_\alpha}})|_1 \circ_\sT \cK( {\varphi^{H_\beta}})|_1 = \varinjlim\nolimits_{\alpha,\beta}\cK( {\varphi^{H_\alpha+H_\beta}})|_1 \xleftarrow{\simeq} Q_U',\]
because $\{H_\alpha\}_\alpha$ and $\{H_\alpha+H_\beta\}_{\alpha,\beta}$ can be interleaved, and in addition, the isomorphism is induced by the continuation maps $1_{\Delta_{M}}^\sT  \to  \cK(\varphi^{H_\alpha})|_1$.

It remains to prove (ii). For $F\in \Sh_{U^c}(M;\sT)$, Lemma \ref{lemma: sheaves fixed by GKS} shows the continuation map induces $F\xrightarrow{\simeq} \cK({\varphi^{H_\alpha}})|_1 \circ_\sT F$. Passing to the colimit, we have $F\xrightarrow{\simeq} Q_U'\circ_\sT F$.  Finally we must show that, conversely, if $F \xrightarrow{\sim} Q'_U \circ_\sT F$, then $F \in \Sh_{U^c}(M;\sT)$.  
The basic point
is that we may show that any given point in $U$ is excluded from the microsupport using an appropriately adapted choice of sequence of $H_\alpha$.  We now explain in detail. 

Take $(q,p,t)\in U\times \RR$. We will show that $(q,p,t,1)\notin ss( Q'_U \circ_{\sT} F) \cap \{\tau=1\}$. As the situation is local, we assume $(q,p,t,\tau)=(0,0,0,1)$, then we take the test function $f(q,t)=t$. By openness of $U\times \RR$, we can take a system of box neighborhood $B_n\times C_n\times I_n\subset U$ of $(0,0,0)$ shrinks to $\{(0,0,0)\}$. 
Take a sequence of functions $G_n \geq 0$ that is supported in $B_{n+1}\times C_{n+1}\times I_{n+1}$, equals to $1/n$ on $B_n\times C_n\times I_n$.

\begin{figure}[htbp]
    \centering
\begin{tikzpicture}[xscale=1,yscale=0.8]
\draw[->] (-5,0)--(5,0)node[anchor=north] {$J^1M$};
\draw[->] (0,0)--(0,2);
\draw[fill=black] (0,0) circle (.05) node[anchor=south east] {$(0,0,0)$};

\draw[->] (-3,-0.6)--(-3,-0);
\draw[->] (3,-0.6)--(3,-0);
\draw (-3,-0.5)-- node[fill=white,inner sep=1mm,midway] {$B_{n}\times C_{n}\times I_{n}$} (3,-0.5);

\draw[->] (-4,-1.1)--(-4,-0);
\draw[->] (4,-1.1)--(4,-0);
\draw (-4,-1)-- node[fill=white,inner sep=1mm,midway] {$B_{n+1}\times C_{n+1}\times I_{n+1}$} (4,-1);
 
\draw (-4,0){[rounded corners]--(-3.6,0)--(-3.3,1)--(3.3,1)--(3.6,0)--(4,0)};
\draw[fill=black] (0,1) circle (.05) node[anchor=south west]{$G_n(0,0,0)=1/n$};

\end{tikzpicture}

\end{figure}

The effect of $\cK(G_n)|_1\circ_\sT$ near $(0,0)$ is a small vertical translation, so \[\cK(G_n)|_1\circ_\sT 1_{B_n \times I_n \cap \{t<0\}}=1_{B_n \times I'_n },\] where $I'_n $ is small translation of $I_n\cap \{t<0\}$ along the $\RR$-direction.  

Notice that, by the interleaving argument for idempotence, we have that, for all $n$, \[\cK(G_n)|_1\circ_\sT Q'_U = \varinjlim\nolimits_{\alpha}\cK(G_n)|_1\circ_\sT \cK({\varphi^{H_\alpha}})|_1 =\varinjlim\nolimits_{\alpha}  \cK({\varphi^{H_\alpha+G_n}})|_1 \simeq Q_U'.\]

Therefore, we have
\begin{eqnarray*}
   \Gamma(B_n \times I'_n,Q'_U \circ_\sT F) &=& \Hom( 1_{B_n \times I'_n}, Q'_U \circ_\sT F)\\
   &=& \Hom( 1_{B_n\times I_n\cap \{t<0\}}, \cK(G_n)|_1\circ_\sT Q'_U \circ_\sT F) \\
   & \simeq &\Gamma(B_n\times I_n\cap \{t<0\},Q'_U \circ_\sT F).
\end{eqnarray*}

Taking colimit over $n$, we have $\Gamma_{\{t\geq 0\}}(Q'_U \circ_\sT F)_{(q,p,t)}=0$.  By definition of $ss$, we see $(q,p,t) \notin ss( Q'_U \circ_\sT F )\cap \{\tau=1\}$. 
\end{proof}

 \subsection{Tensor product and integral functors}
For this section, we take the point of view that, for an open set $U \subseteq T^* M$, the Tamarkin category $\sT(U)$ is the quotient of the inclusion $\Sh_{U^c}(M;\sT) \hookrightarrow \Sh(M;\sT)$, where the subscript indicates the not-necessarily-conic microsupport condition on $\underline{ss}(F)$. Thus, all the functors, categorical operations, etc. will automatically be $\sT$-linear as discussed in Remark \ref{tamarkin-six-functors}. In this subsection, we take open sets $U\subseteq T^*M$ and $V\subseteq T^*N$ and we set $Z=T^*M\setminus U$ and $X=T^*N\setminus V$.

\begin{proposition}\label{proposition: tensor product closed of kernels}We set $W=T^*M\times T^*N \setminus Z\times X $. Then we have $Q_W \simeq Q_U \boxtimes_{\sT} Q_V=s_!(Q_U \boxtimes Q_V)$.    
\end{proposition}
\begin{proof}We have $\Sh_{Z}(M;\sT) \otimes_{\sT} \Sh_{X}(M;\sT) \hookrightarrow \Sh_{Z\times X}(N\times M;\sT)$ by Lemma \ref{lemma: product reduced microsupport}. We set $Q=Q_U \boxtimes_{\sT} Q_V$, and then we have $Q\circ_{\sT} Q \simeq Q$. We conclude by proving that $F\in \Sh_{Z\times X}(N\times M;\sT)$ if and only if of $Q \circ _\sT F \simeq F$. Then the definition of $Q_W$ implies that $Q_W\simeq Q$.

If $Q \circ _\sT F \simeq F$, then we write $F=\varinjlim A\boxtimes_\sT B$ and we see that $F \simeq\varinjlim (Q_U\circ_\sT A)\boxtimes_\sT (Q_V\circ_\sT B) $, which is clear in $\Sh_{Z\times X}(N\times M;\sT)$ by Lemma \ref{lemma: product reduced microsupport}. 

Conversely, take $F\in \Sh_{Z\times X}(N\times M;\sT)$. Using the wrapping formula, we have that $Q \simeq \varinjlim_{(\alpha,\beta)}\cK(\varphi^{H_\alpha})|_1 \boxtimes_\sT \cK(\varphi^{G_\beta})|_1$ for cofinal sequences of non-negative Hamiltonians with $\supp(H_\alpha)\subset U$ and $\supp(G_\beta)\subset V$. Then Lemma \ref{lemma: sheaves fixed by GKS} shows that $\cK(\varphi^{H_\alpha})|_1 \boxtimes_\sT \cK(\varphi^{G_\beta})|_1 \circ_\sT  F\simeq F$, and the isomorphism is compatible with the continuation map. In the colimit, we have $Q \circ _\sT F \simeq F$.
\end{proof}
With the product property of kernel $Q$, we have
\begin{proposition}\label{proposition: tensor product closed}We set
\[\Sh_{U\times X }(M\times N;\sT)\coloneqq \Sh_{T^*M\times  X}(M\times N;\sT)/\Sh_{Z\times X}(M\times N;\sT).\]

Then we have an isomorphism of fiber sequences
\begin{equation*}
\begin{split}
       &\Sh_{Z\times X}(M\times N;\sT) \rightarrow \Sh_{ T^*M \times X  }(M\times N;\sT) \rightarrow  \Sh_{U\times X }(M\times N;\sT) \\
       \simeq & [\Sh_{Z}( M;\sT) \rightarrow \Sh( M;\sT) \rightarrow  \sT(U)]  \otimes_\sT \Sh_X(N;\sT).
\end{split}
\end{equation*}
\end{proposition}
\begin{proof} 
By \cite[Corollary 2.29]{Haine}, we have that the 
\[[\Sh_{Z}( M;\sT) \rightarrow \Sh( M;\sT) \rightarrow  \sT(U)]  \otimes_\sT \Sh_X(N;\sT)\]
is a Verdier sequence. Therefore, $\Sh_{Z}( M;\sT) \otimes_\sT  \Sh_X(N;\sT) \rightarrow \Sh( M;\sT) \otimes_\sT  \Sh(N;\sT)$ is fully-faithful, and we can identify the target of the functor by Corollary \ref{tamarkin-category-product} with $\Sh(M\times N;\sT)$. To conclude the statement, it remains to prove that $\Sh_{Z}( M;\sT) \otimes_\sT  \Sh_X(N;\sT)\simeq \Sh_{Z\times X}(M\times N;\sT)$.

By Corollary \ref{lemma: product reduced microsupport}, the functor $\Sh_{Z}( M;\sT) \otimes_\sT  \Sh_X(N;\sT) \rightarrow \Sh( M;\sT) \otimes_\sT  \Sh(N;\sT)$ factor through the fully-faithful subcategory $\Sh_{Z\times X}(M\times N;\sT)$. Therefore, we only need to show that $\Sh_{Z\times X}(M\times N;\sT)$ is also the essential image. 

To do so, we run the same argument as as Proposition \ref{tamarkinsheaf} to show that $\Sh_{Z\times X}(M\times N;\sT) \subset \Sh_{Z}( M;\sT) \otimes_\sT  \Sh_X(N;\sT)$. Any object $F\in \Sh_{Z\times X}(M\times N;\sT)$ can be written as a colimit $F=\varinjlim A\boxtimes_\sT B $ for $A\in \Sh_{Z}( M;\sT)$ and $B\in \Sh_{X}( N;\sT)$. On the other hand, by definition of $Q_W $, we know that $F\simeq Q_W\circ_\sT F$. Then $Q_W \simeq Q_U \boxtimes_{\sT} Q_V$ implies that 
\[F\simeq Q_W\circ_\sT F \simeq \varinjlim Q_U \circ_\sT A \boxtimes_{\sT} Q_V \circ_\sT B \in \Sh_{Z}( M;\sT) \otimes_\sT  \Sh_X(N;\sT).\qedhere\]
\end{proof}
\begin{proposition}\label{proposition: tensor product open}We have the equivalence
\[\sT(U)\otimes_\sT \sT(V) = \sT( U \times V).\]
\end{proposition}
\begin{proof}By the third isomorphism theorem for the Verdier quotient, we have an equivalence
\[\sT(U\times V)\simeq \Sh_{T^*M \times V}(M\times N;\sT)/\Sh_{Z\times V}(M\times N;\sT).\]

The Proposition \ref{proposition: tensor product closed} shows that
\[\Sh_{Z\times V}(M\times N;\sT) \simeq \Sh_{Z}(M;\sT)\otimes_\sT \sT(V),\quad \Sh_{T^*M\times V}(M\times N;\sT) \simeq \Sh(M;\sT)\otimes_\sT \sT(V),\]
and we have the following equality
\[[\Sh_{Z\times V}(M\times N;\sT) \hookrightarrow \Sh_{T^*M\times V}(M\times N;\sT)] = [\Sh_{Z}(M;\sT) \hookrightarrow \Sh(M;\sT) ]\otimes_\sT \sT(V).\]

Then the result follows.
\end{proof}

\begin{proposition} \label{classification-of-kernels}We have the equivalence 
\begin{align*}
\sT(-U \times V) &\xrightarrow{\sim} \Fun^L_{\sT}\left(\sT(U),\sT(V)\right) \\
K &\mapsto (F \mapsto K \circ_\sT F).
\end{align*} 
\end{proposition}
\begin{proof}
For any symmetric monoidal category $\cM$, there is a canonical equivalence $\inHom(X,Y) = X^\vee \otimes Y$, for any dualizable object $X$. Thus, the right hand side is given by $\Fun^L_{\sT}\left(\sT(U),\sT(V)\right) = \sT(U)^\vee \otimes_\sT \sT(V)$ by Lemma \ref{duality-of-retraction} and Proposition \ref{T linear retract}. A small modification of Remark \ref{dual-microsupport-condition} implies that $\sT(U)^\vee = \sT(-U)$.
\end{proof}

\subsection{Microsupport estimation of trace for open sets with restricted contact boundary}\label{section: RCT estimation} 
Recall that a compact smooth hypersurface $S\subseteq T^*M$ is of restricted contact type (RCT) if the Liouville field on $T^*M$ is outward pointingly transverse to $S$. In particular, $S$ is a contact manifold.

\begin{lemma}[{\cite[(74)]{JunZhang}, \cite[Lemma 3.28]{zhang-thesis}}]\label{lemma-actionspectrumestimate}Let $U$ be a bounded open set such that $\partial U$ is a RCT hypersurface.  
Consider $\Tr(1_{\sT(U)}) \in \sT \hookrightarrow \Sh(\RR)$.  Its
microsupport is contained in the action spectrum: 
\begin{equation}
    \label{action spectrumestimate}\dot{ss}\big(\Tr(1_{\sT(U)})\big) /\RR_+\subseteq \mathcal{A}(\partial U) \subseteq [0,\infty).
\end{equation}
\end{lemma}
\begin{proof}
    What was actually calculated in the mentioned references is right hand side of \eqref{trace of projector diagonal}.  Specifically,  \cite[(74)]{JunZhang} treated the case of $U$ a ball, and \cite[Lemma 3.28]{zhang-thesis} verified that a similar argument works in the general case.
\end{proof}

When $\partial U$ is of RCT, we can take a Hamiltonian function $H$ such that $U=\{H<1\}$ and $\partial U=\{H=1\}$ is a regular hypersurface. In particular, for an $e>0$ small enough, we can identify a collar neighborhood $N$ of $\partial U=\{H=1\}$ with $(1-e,1+e)\times \partial U$, and under this identification, we have $H(r,y)=r$ for $(r,y)\in (1-e,1+e)\times \partial U=N$. Now, we can take a particular cofinal sequence $H_\alpha$ to define the kernel $Q_U$.

For $\alpha>0$, we can take $C^\infty$-functions $\rho_\alpha: I\times [0,1+e)\rightarrow [0,\infty)$ such that, for each $z\in I$,
\begin{enumerate}
    \item $\rho_\alpha(z,-)=\alpha$ on $[0,d_\alpha]$ for $d_\alpha>0$ small enough;
    \item $\rho_\alpha(z,-)=0$ on $(1-e_\alpha,1+e)$, $0<e_\alpha<e$;
    \item $\rho_\alpha'(z,-)<0$ and $\rho_\alpha''(z,-)>0$ on $(2d_\alpha,1-2e_\alpha)$;
    \item $\rho_\alpha'(z,-)<0$ elsewhere.
\end{enumerate}

We set $H_\alpha=\rho_\alpha\circ H$. Then there exists $\rho_\alpha$ such that $H_\alpha$ form a cofinal sequence, and moreover we assume all non-constant 1-periodic orbits of $X_{H_\alpha}$ are non-degenerate.

Then we use this sequence $H_\alpha$ to exhibit the wrapping formula $Q_U \cong \varinjlim_{\alpha}\cK(\varphi^{H_\alpha})|_1.$
\begin{lemma}\label{lemma: negative part vanishing}   If $U$ has a RCT boundary and $L<0$, we have \[i_{L}^*\Tr(1_{\sT(U)})=0 \qquad \qquad i_{L}^*\Tr(1_{\Sh_{U^c}(M;\sT)})=0.\]
\end{lemma}
\begin{proof}By the fiber sequence \eqref{equation:fiber sequence for trace} and $\Tr(1_{\sT(T^*M)}) \simeq \Gamma_c(M,1_M)\otimes 1^{\sT}$, we only need to prove $i_{L}^*\Tr(\Sh_{U^c}(M;\sT))=0$ for $L<0$.

Actually, by Lemma \ref{lemma-actionspectrumestimate}, it is sufficient to prove $i_{L}^*\Tr(\Sh_{U^c}(M;\sT))=0$ for $L\ll 0$. By virtue of the wrapping formula, it remains to show that for the cofinal sequence $H_\alpha$ taken as above, we have $i_{L}^*\underline{a}_{!}\underline{\Delta}^*\cK(\varphi^{H_\alpha})|_{1}=0$ for big enough $\alpha$ and $L\leq -1$. 

Consider \[\mathcal{S}(H)=\{t=-S_{H}^1(q,p):(q,p) \text{ is a fixed point of }\varphi^{H}_1\}.\]

Using the given cofinal sequence $H_\alpha$, we have that $\mathcal{S}(H_\alpha)$ is a discrete subset of $\RR$, and if $\alpha \gg 0$, we have $\mathcal{S}(H_\alpha)\subseteq \RR_{t\geq -1}$ for $\alpha\gg 0$. Then the same argument of Lemma \ref{lemma-actionspectrumestimate} implies that $\dot{ss}(\underline{a}_{!}\underline{\Delta}^*\cK(\varphi^{H_\alpha})|_{1})/\RR_+\subseteq  \mathcal{S}(H_\alpha) \subseteq  \RR_{t\geq -1} $. Therefore, by the microlocal Morse lemma, we only need to show that $i_{L}^*\underline{a}_{!}\underline{\Delta}^*\cK(\varphi^{H_\alpha})|_{1}=0$ for $L\ll 0$.

On the other hand, it is explained in \cite[Section 4]{Guillermou-Viterbo} that, for all compactly supported functions $H$, $\cK(\varphi^{H})|_{1}$ is isomorphic to $1^\sT_{\Delta_ {M}  }$ outside a compact set of $M^2\times \RR$ (Argument therein is microlocal, so the coefficient category does not matter.) In particular, we have $i_{L}^* {\underline{a}_{!}\underline{\Delta}^*}\cK(\varphi^{H})|_{1}=0$  for $L\ll 0$. Then the result follows.  
\end{proof}

\section{Hochschild cohomology}
Let $\cA$ be a $\mathbf{k}$-linear symmetric monoidal category and $\sC$ be an $\cA$-linear category $\sC$. For a functor $f:\sC\rightarrow \sC$ in $\PrLst(\cA)$, we consider the Hochschild cochains and cohomology (in case that $\cA$ has a t-structure): 
\[\HH_\cA^\bullet(\sC,f)\coloneqq \Hom_{\End_{\cA}{(\sC)}/\cA}(1_\sC,f), \qquad \HH_\cA^*(\sC,f)\coloneqq H^*\HH^\bullet(\sC,f).\]

We can also discuss the $\mathbf{k}$-linear Hochschild cochains
\[\HH_\mathbf{k}^\bullet(\sC,f)\coloneqq \Hom_{\End_{\mathbf{k}}{(\sC)}/\mathbf{k}}(1_\sC,f).\]

If we consider $\End_{\cA}{(\sC)}$ as a $\mathbf{k}$-linear category, we have a couple of comparison morphisms, as we discussed in Section \ref{subsection: Extending coefficient},
\begin{equation}\label{equation: comparision morphism of HH}\Hom_{\cA/\mathbf{k}}(1_\cA,\HH_\cA^\bullet(\sC,f))=\Hom_{\End_{\cA}{(\sC)}/\mathbf{k}}(1_\sC,f)\rightarrow \Hom_{\End_{\mathbf{k}}{(\sC)}/\mathbf{k}}(1_\sC,f)=\HH_\mathbf{k}^\bullet(\sC,f) \,\in \mathbf{k}.
\end{equation}
\begin{lemma}\label{lemma: Hochschild comparison}Let $X$ be a $\mathbf{k}$-linear dualizable category and $\sC$ be a $\cA$-linear quotient of $X\otimes \cA$. If the regular representation $\rho_\cA$ (see \eqref{regular representation}) is fully-faithful, then the comparison map \eqref{equation: comparision morphism of HH} is an isomorphism.    
\end{lemma}
\begin{proof}We use the left adjoint of the quotient functor $X\otimes \cA \rightarrow \sC$ to treat $\sC$ as a full subcategory of $X\otimes \cA$. Then this is a consequence case of Corollary \ref{lemma: restriction fully-faithful for quotient}.    
\end{proof}

In this section, we will mainly discuss the case $\cA=\sT$ and $\sC=\sT(U)$.

\subsection{Calabi-Yau structure} \label{CY and cohomology}
In this section we give a Calabi-Yau structure on $\sT(U)$, and use this to 
relate the Hochschild homology and cohomology of this category. We will fix and will frequently use $a:M \rightarrow \pt$ as the constant map on $M$ and $\pi:M^2\rightarrow \pt$ as the constant map on $M^2$.

Recall we have fixed our coefficients  $\mathbf{k}$ to be a  compactly generated symmetric monoidal category.  As always, the  dualizing sheaf on $M$ is by definition $\omega_{M,\mathbf{k}}\coloneqq a^{!}1_{\mathbf{k}}$, and a $\mathbf{k}-$orientation of $M$ is defined as an isomorphism of sheaves $\omega_{M,\mathbf{k}}= a^{!}1_{\mathbf{k}}\simeq a^{*}1_{\mathbf{k}}[n] $. When $\mathbf{k}=\Mod_R$ for a commutative ring spectrum $R$, this agrees with the usual notion of $R$-orientation \cite[Proposition 6.18]{volpe_6_functor}. 

The same proposition implies that if $M$ is $\mathbf{k}-$orientable,  the orientation induces an equivalence of functors $a^{!}_{\sT}\simeq a^{*}_{\sT}[n]$ and we have $\omega_{M}^{\sT}\coloneqq a^{!}_{\sT}1^{\sT}\simeq a^{*}_{\sT}1^{\sT}[n]$. 

Now, we explain the right Calabi-Yau property of the Tamarkin category. We require a version of Calabi-Yau structure suitable for categories with presentable coefficients, for which we follow \cite[Section 4.3.5]{Brav-Rozenblyum}. 

We start from properness. For a presentable symmetric monoidal category $\cA$ and an $\cA$-linear dualizable category $\sC$, we say $\sC$ is proper if the $\cA$-linear counit map $\epsilon: \sC^\vee\otimes_\cA \sC\rightarrow \cA$ has a continuous $\cA$-linear right adjoint $\epsilon^r: \cA \rightarrow\sC^\vee\otimes_\cA \sC$ (In particular, it requires that $\epsilon^r$ has a right adjoint). If $\sC$ is proper, we call $\id^{\vee}_{\sC}=\epsilon^r(1_\cA)\in \sC^\vee\otimes_\cA \sC=\End_{\cA}(\sC)$ the Serre functor of $\sC$. In this case, the $\cA$-linear dual of $\epsilon$ is represented by $\epsilon^r$, and we have a natural identification
\begin{equation}\label{equation: proper category HH=HH}
 \inHom_{\cA}( \Tr(\id_\sC), 1_\cA)= \HH_\cA^\bullet(\id_\sC, \id^{\vee}_{\sC}). \end{equation}

\begin{proposition} \label{tamarkin proper serre} For open sets $U\subseteq T^*M$,  $\sT(U)$ is proper and has  Serre functor $$\id^{\vee}_{\sT(U)}(F)=F\otimes_{\sT} \omega_{M}^{\sT}.$$
\end{proposition}
\begin{proof}By Lemma \ref{duality-of-retraction}, the counit of $\sT(U)$ as a dualizable category is computed by
\[\sT(U)^\vee\otimes \sT(U) \hookrightarrow \sT(T^*M)^\vee\otimes \sT(T^*M)=\sT(T^*(M\times M))\xrightarrow{a^{\sT}_!\Delta^*_{\sT}} \sT.\]

The right adjoint of the first functor is the natural quotient functor $\sT(T^*M)^\vee\otimes \sT(T^*M) \rightarrow\sT(U)^\vee\otimes \sT(U) $, which is continuous. The last functor has right adjoint $\Delta_*^{\sT}a_{\sT}^!$, which is naturally equivalent to $\Delta_!^{\sT}(a_{\sT}^*(-) \otimes_{\sT} \omega_{M}^{\sT})$. Therefore, the right adjoint $\Delta_*^{\sT}a_{\sT}^!$ is continuous. 

Then the counit has a continuous right adjoint, i.e. $\sT(U)$ is $\sT$-linearly proper. In particular, the Serre functor $\id^{\vee}_{\sT(U)}$ of $\sT(U)$ is given by $\id^{\vee}_{\sT(U)}(F)=F\otimes_{\sT} \omega_{M}^{\sT}$ for $F\in \sT(U)$. 
\end{proof}

\begin{corollary}\label{proposition: tw HH co using sheaf}We have 
\[\HH_{\sT}^\bullet(\sT(U),\id^{\vee}_{\sT(U)}) \simeq\sHom_\sT(\Tr(1_{\sT(U)}), 1^\sT)\simeq  \pi_*^\sT\sHom_\sT(P_U, \Delta_{*}^{\sT}\omega_{M}^\sT)\in \sT.\]

In particular, when $M$ is $\mathbf{k}$-orientable, we have
\[\HH_{\sT}^\bullet(\sT(U))\simeq\sHom_\sT(\Tr(1_{\sT(U)}), 1^\sT [-n]) . \]
\end{corollary}
\begin{proof}The equivalence between Hochschild cohomology valued in the Serre functor and the linear dual of Hochschild homology is a general fact of proper categories as explained in \eqref{equation: proper category HH=HH}. We apply the general fact to $\sT$-linear Hochschild (co)homology of the proper $\sT$-linear category $\sT(U)$ to obtain $\HH_{\sT}^\bullet(\sT(U);\id^{\vee}_{\sT(U)}) \simeq\sHom_\sT(\Tr(1_{\sT(U)}), 1^\sT)$. The equivalence $\sHom_\sT(\Tr(1_{\sT(U)}), 1^\sT)\simeq \pi_*^\sT\sHom_\sT(P_U, \Delta_{*}^{\sT}\omega_{M}^\sT)$ follows adjunctions of 6-operations of $\sT$-valued sheaves.

For the last statement, notice that when $M$ is $\mathbf{k}$-orientable, we have the equivalence of functor: $\id^{\vee}_{\sT(U)}\simeq \id_{\sT(U)}[n]$ by Proposition \ref{tamarkin proper serre}.
\end{proof}

It is known that there exists a $S^1$-action on $\HH_\bullet(\sC)=\Tr(\id_\sC)$. 
We say $\sC$ is (absolutely) right $d$-Calabi-Yau, if there exists an $S^1$-equivariant equivalence $\Tr(\id_\sC)\rightarrow 1_\cA[-d]$ (equipping $1_\cA$ with the trivial action). Under the natural identification \eqref{equation: proper category HH=HH}, the Calabi-Yau structure gives an $S^1$-equivariant equivalence $\id_\sC\simeq \id^{\vee}_{\sC}[-d]$. 

Corollary \ref{proposition: tw HH co using sheaf} has shown that there exists an equivalence $\id_{\sT(U)}\simeq \id^{\vee}_{\sT(U)}[-n]$. Now we show that the $\sT(U)$ is right Calabi-Yau, and equivalence $\id_{\sT(U)}\simeq \id^{\vee}_{\sT(U)}[-n]$ is $S^1$-equivariant.

\begin{theorem} \label{right calabi yau}
     If $M$ is $\mathbf{k}$-orientable, then for open $U \subset T^*M$, the category $\sT(U)$ is (absolutely) right Calabi-Yau.
 In particular, we have $\HH_{\sT}^\bullet(\sT(U);\id^{\vee}_{\sT(U)}) \simeq\HH_{\sT}^\bullet(\sT(U))[n] .$
\end{theorem}
\begin{proof}
From Proposition \ref{tamarkin proper serre}, we see that
when $\omega_M$ is trivial, the Serre functor $\id^{\vee}_{\sT(U)}$ is equivalent to $\id_{\sT(U)}[n]$.  This isomorphism is moreover
induced by the following map, obtained by tracing the inclusion $\sT(U) \to \sT(T^*M)$ and composing with the orientation of $M$: 
\[\Tr_{\sT}(\id_{\sT(U)}) \rightarrow \Tr_{\sT}(\id_{\sT(T^*M)})=a^{\sT}_!\Delta^*_{\sT}(1_{\Delta_{M}}^\sT)=\Gamma_c(M,\mathbf{k)}\otimes 1^\sT \rightarrow 1^\sT[-n].\]

Finally we must check that the above morphism is $S^1$-equivariant. By the  result of \cite[Theorem 2.14]{Hoyois-Scherotzke-Sibilla} (or \cite[Proposition 4.3]{Brav-Dyckerhoff2}), the first arrow is $S^1$-equivariant. For the arrow $\Tr_{\sT}(\id_{\sT(T^*M)})\rightarrow 1^\sT[-n]$, all $S^1$-actions are trivial.
\end{proof}

\begin{remark}\label{Remark: cyclic Deligne}
By the cyclic Deligne conjecture for Hochschild cochains, now proven \cite{Kaufmann_2008, Tradler_2006, Brav-Rozenblyum}, the $\mathbb{E}_2$ structure on Hochschild cochains $\HH^\bullet(\sC)$ of a dualizable category $\sC$, which is predicted by the Deligne conjecture, can be lift to a framed $\mathbb{E}_2$-structure if $\sC$ is Calabi-Yau. Then the Hochschild cochains $\HH^\bullet_{\sT}(\sT(U))$ carries a $\mathbb{E}_2$ structure, which could be lifted to a framed $\mathbb{E}_2$-structure if $M$ is orientable. 
\end{remark}

\subsection{Hochschild cohomology from projectors}\label{section: HH computed by kernels}

Recall we write $\nT_L$ for translation by $L$ in the $\RR$ direction. We now explain how to compute Hochschild cohomologies using the projector.

\begin{proposition}\label{proposition: HH co using sheaf}

We have
\[
 \HH_{\sT}^\bullet(\sT(U), {T_L}_*) \simeq \pi_*^\sT\sHom_\sT(P_U, \nT_{L*} P_U) \simeq \pi_*^\sT\sHom_\sT(P_U, \nT_{L*} 1_{\Delta_{M}}^\sT).\]
\end{proposition}
\begin{proof}By $\sT$-linear dualizability of $\sT(U)$ (Proposition \ref{t dualizable}), we have the fully-faithful embedding $\sT(U)^\vee\otimes_\sT \sT(U) \hookrightarrow \sT(T^*M)^\vee\otimes_\sT \sT(T^*M)=\Sh(M^2;\sT)$. Therefore, by identifying $\sT(U)^\vee\otimes_\sT \sT(U)\simeq \Fun^L_{\sT}(\sT(U),\sT(U))$, we have the fully faithful embedding
\begin{align*}
    \Fun^L_{\sT}(\sT(U),\sT(U)) \hookrightarrow \Fun^L_{\sT}(\Sh(M;\sT),\Sh(M;\sT)),\\
    \cF \mapsto [ \Sh(M;\sT) \twoheadrightarrow \sT(U) \xrightarrow{\cF} \sT(U) \xrightarrow{\Phi_{P_U}}  \Sh(M;\sT)].
\end{align*}
Under this fully-faithful embedding, we have that $\nT_{L*}:\sT(U)\rightarrow \sT(U)$ is mapped to $\nT_{L*} \Phi_{P_U}:\Sh(M;\sT)\rightarrow \Sh(M;\sT)$. Here $\Phi_{P_U}$ is the integral functor defined by $P_U$.  
As we explained in Remark \ref{remark: different hom}, the $\sT$-valued hom of $\Sh(M^2;\sT) \simeq \Fun^L_{\sT}(\Sh(M;\sT),\Sh(M;\sT))$ is computed by $\pi_*^\sT\sHom_\sT$. Then we have the first equivalence
\[\Hom_{\End_\sT{\sT(U)}/\sT}(\id_{\sT(U)}, \nT_{L*})=\Hom_{\End_\sT{\Sh(M;\sT)}/\sT}(\Phi_{P_U}, \nT_{L*} \Phi_{P_U})=\pi_*^\sT\sHom_\sT(P_U, \nT_{L*} P_U) .\]

For the second equivalence, by the fiber sequence $P_U \to 1_{\Delta_{M}}^\sT \to Q_U$, we only need to prove that 
\[\Hom_{\End_\sT{\Sh(M;\sT)}/\sT}(\Phi_{P_U}, \nT_{L*} \Phi_{Q_U})=\pi_*^\sT\sHom_\sT(P_U, \nT_{L*} Q_U)=0.\]

This is true because $P_U$ defines the projector onto the left orthogonal complement of the essential image of the functor $Q_U$. Precisely, for natural transformation $\eta: \Phi_{P_U}\rightarrow \nT_{L*} \Phi_{Q_U}$ and all $F\in \Sh(M;\sT)$, we have $\eta(F)\in \pi_*^\sT\sHom_\sT(\Phi_{P_U}(F), \nT_{L*} \Phi_{Q_U}(F))=0$, and then $\eta=0$. \end{proof}

The above Proposition \ref{proposition: HH co using sheaf} thus implies its $k$-linear version:

\begin{proposition}\label{proposition: k-HH co using sheaf}For $L\in \RR$, we have
\[\HH_{\mathbf{k}}^\bullet(\sT(U),\nT_{L*}) \simeq \Hom(P_U,\nT_{L*} P_U)  \simeq \Hom(P_U,\nT_{L*} 1_{\Delta_{M}}^\sT).
\]
\end{proposition}
\begin{proof}By Proposition \ref{proposition:fully-faithful of forgetful functor between functors} and Lemma \ref{lemma: Hochschild comparison}. We have
\[\HH_{\mathbf{k}}^\bullet(\sT(U),\nT_{L*})= \Hom \left(1_\sT,  \HH_{\sT}^\bullet(\sT(U), {T_L}_*) \right).\]

Then the rest statements follow from the equation and Proposition \ref{proposition: HH co using sheaf}.
\end{proof}

\begin{remark}\label{remark: CY true for the "out" category} All results in Sections \ref{CY and cohomology}, \ref{section: HH computed by kernels} are true for $\Sh_{U^c}(M;\sT)$ by replacing $\sT(U)$ by $\Sh_{U^c}(M;\sT)$ and replacing $P_U$ by $Q_U$.
\end{remark}

It follows Proposition \ref{proposition: HH co using sheaf}, Remark \ref{remark: CY true for the "out" category} and the fiber sequence
$P_U \to 1^\sT_{\Delta_{M}} \to Q_U$ that we have fiber sequences as follows.
\begin{equation}\label{equation: HH fiber sequences}
\HH_{\sT}^\bullet(\Sh_{U^c}(M;\sT)) \to \Gamma(M,1_M^\sT)
   \to    \HH_{\sT}^\bullet(\sT(U)). 
\end{equation}

The analogous sequences for Hochschild cohomology with coefficients in the Serre functor also hold.

\subsection{Action window}
For $-\infty \leq a <b < \infty$, we introduce 
\begin{align*}
    \HH_{\sT}^\bullet(\sT(U),{(a,b]})
    &= \Hom(1_{ [-b,-a)},\HH_{\sT}^\bullet(\sT(U)))  ,\\ \HH_{\sT}^\bullet(\Sh_{U^c}(M;\sT),{(a,b]})
    &= \Hom(1_{ [-b,-a)}, \HH_{\sT}^\bullet(\Sh_{U^c}(M;\sT)))  .  
\end{align*}
By Proposition \ref{proposition: k-HH co using sheaf}, for all $L\in \RR$, we have \[\HH_{\sT}^\bullet(\sT(U),{(-\infty,L]})=\HH^\bullet_{\mathbf{k}}(\sT(U), \nT_{L*}).\]
For $-\infty \leq a <b  <c < \infty$ We have fiber sequences
\begin{equation}\label{equation: action window sequence for HH}
    \begin{split}
  &\HH^\bullet_{\sT}(\sT(U), (a,b])\rightarrow \HH^\bullet_{\sT}(\sT(U),(a,c])\rightarrow \HH^\bullet_{\sT}(\sT(U),{(b,c]})      ,\\ &\HH^\bullet_{\sT}(\Sh_{U^c}(M;\sT), (a,b])\rightarrow \HH^\bullet_{\sT}(\Sh_{U^c}(M;\sT), (a,c])\rightarrow \HH^\bullet_{\sT}(\Sh_{U^c}(M;\sT),{(b,c]}) .       
    \end{split}
\end{equation}

 \begin{lemma}\label{lemma: 9-diagram}For $\infty>L>\epsilon > 0 >- \delta\geq-\infty$, we have an isomorphism of fiber sequences:
\[\begin{tikzcd}
{\HH_{\sT}^\bullet(\Sh_{U^c}(M;\sT),(-\delta,\epsilon])   }  \arrow[d] \arrow[r]   & {\HH_{\sT}^\bullet(\Sh_{U^c}(M;\sT),{(-\delta,L]})} \arrow[d] \arrow[r] & {\HH_{\sT}^\bullet(\Sh_{U^c}(M;\sT),{(\epsilon,L]})} \arrow[d]   \\
{{\Gamma(M,\sT)}} \arrow[d] \arrow[r, "\simeq"] & {{\Gamma(M,\sT)}} \arrow[d] \arrow[r]                    & 0 \arrow[d]                                 \\
{\HH_{\sT}^\bullet(\sT(U),(-\delta,\epsilon])   }   \arrow[r]   & {\HH_{\sT}^\bullet(\sT(U),{(-\delta,L]})}  \arrow[r]                   & {\HH_{\sT}^\bullet(\sT(U),{(\epsilon,L]})}                                                                                                                                                                    
\end{tikzcd}\]
\end{lemma}
\begin{proof}We apply the fiber sequence of functors
\[\Hom(1_{[-\epsilon,\delta)},-)\rightarrow \Hom(1_{[-L,\delta)} ,-)\rightarrow \Hom(1_{[-L,-\epsilon)},-)  \]
to
\[ \HH_{\sT}^\bullet(\Sh_{U^c}(M;\sT)) \to \Gamma(M,1_M^\sT)
   \to    \HH_{\sT}^\bullet(\sT(U)).\]
   
For the first and last rows, the results are tautology by definition of notations. The second row follows from
\[\Hom(1_{[a,b)},\Gamma(M,1_M^\sT))=\Hom(1_{[a,b)}, a_{*}^\sT(1_M^\sT))=\Hom(1_{M\times [a,b)}, 1_{M\times [0,\infty)}). \]
So, for $[a,b)=[-L,-\epsilon)$, we have $0$; for $[a,b)=[a,\delta)$ with $a<0$, we have $\Hom(1_{M}, 1_{M})=\Gamma(M,1_M)$.
\end{proof}

\subsection{\texorpdfstring{Recollections from \cite{Zhang-capacities-Chiu-Tamarkin, Zhang-S1-Chiu-Tamarkin} and re-interpretation}{}} \label{bingyu results}

In \cite{Zhang-capacities-Chiu-Tamarkin, Zhang-S1-Chiu-Tamarkin}, motivated by and building on \cite{Chiu-nonsqueezing}, the expression $\Gamma_c^\sT(M, \Delta_\sT^* P_U)$ was studied both in terms of its general properties, 
and values for certain specific $U$.  
Here we recall the results. 

In \cite{Zhang-capacities-Chiu-Tamarkin, Zhang-S1-Chiu-Tamarkin}, all formulas were expressed in terms of the fully faithful images of $\sT \hookrightarrow \Sh(\RR)$ and, for $U \subseteq T^*M$, of 
$\sT(U) \hookrightarrow \Sh(M \times \RR)$. 
With this identification, $\Gamma_c^\sT(M, \Delta_\sT^* P_U)$ is computed in terms of 
$\underline{a}: M \times \RR \rightarrow \RR$ and $\underline{\Delta}: M \times \RR\rightarrow M\times M \times \RR$
as 
$\underline{a}_{!}\underline{\Delta}^*P_U \in \Sh(\RR)$ (cf. Remark \ref{tamarkin-six-functors}).

We write $i_L: \{L\} \to \RR$ for the inclusion and $\nT_L$ the translation by $L$ along $\RR$-direction. The following notation was used 
\cite[(2.1), Definition 2.1]{Zhang-S1-Chiu-Tamarkin}\footnote{We call them Chiu-Tamarkin invariants in {\it loc.cit.}}:
\begin{align*}
    &F_1(U,\mathbf{k})\coloneqq \underline{a}_{!}\underline{\Delta}^*P_U \in \Sh(\RR),&\quad &C_L(U,\mathbf{k})\coloneqq \Hom(i_L^*F_1(U,\mathbf{k}),1[-n]) \in \mathbf{k};\\
    &F_1^{out}(U,\mathbf{k})\coloneqq \underline{a}_{!}\underline{\Delta}^*Q_U \in \Sh(\RR),&\quad &C_L^{out}(U,\mathbf{k})\coloneqq \Hom(i_L^*F_1^{out}(U,\mathbf{k}),1[-n]) \in \mathbf{k}.
\end{align*}

We now discuss the relation of these previous constructions to those of the
present article. Lemma \ref{duality-of-retraction} implies that
\begin{equation}\label{equation: F=Trace}
    F_1(U,\mathbf{k}) =\Tr(1_{\sT(U)}),\quad F^{out}_1(U,\mathbf{k}) =\Tr(1_{\Sh_{U^c}(M;\sT)}).
\end{equation}
Then we have 
\begin{eqnarray*}
C_L(U, \mathbf{k})  = &  \Hom(i_L^*\Tr(1_{\sT(U)}) ,1[-n]) 
& =  \Hom(\Tr(1_{\sT(U)}),1_{\{L\}}[-n]) \\
 = & \Hom(\Tr(1_{\sT(U)}),1_{\RR_{\ge L} }[-n]) 
& =  \Hom(\Tr(1_{\sT(U)}), \nT_{L*}1^\sT [-n]).
\end{eqnarray*}
Therefore, Corollary \ref{proposition: tw HH co using sheaf} shows that \[C_L(U,\mathbf{k})=\Gamma_{[-L,\infty)}(\RR,\HH_{\sT}^\bullet(\sT(U),\id^{\vee}_{\sT(U)})[-n]),\] and Remark \ref{remark: CY true for the "out" category} shows that $C_L^{out}(U,\mathbf{k})=\Gamma_{[-L,\infty)}(\RR,\HH_{\sT}^\bullet(\Sh_{U^c}(M;\sT),\id^{\vee}_{\Sh_{U^c}(M;\sT)})[-n]))$.

In Remark \ref{Remark: cyclic Deligne}, we explained that there exists a framed $\mathbb{E}_2$-structure on $\HH_{\sT}^\bullet(\sT(U))$. Corresponding chain-level Gerstenhaber product is given by the endomorphism ring structure of $\HH_{\sT}^\bullet(\sT(U))$. In \cite{Zhang-S1-Chiu-Tamarkin}, we defined by hand a `cup product' on $C_L$, when $M$ is orientable, via the identification $C_{L}(U,\mathbf{k}) \simeq \Hom(P_U, \nT_{L*}P_U) $ using the isomorphism of Proposition \ref{proposition: HH co using sheaf}. Therefore, by the identification $C_{L}(U,\mathbf{k}) \simeq \Hom(P_U, \nT_{L*}P_U)  \simeq \HH_{\sT}^\bullet(\sT(U),(-\infty,L]) $, the manually defined `cup product' on $C_L$ in \cite{Zhang-S1-Chiu-Tamarkin} is exactly the corresponding Gerstenhaber product on cohomology $\HH_{\sT}^*(\sT(U))$. 

Therefore, we reinterpret computations of $C_{L}(U,\mathbf{k})$ in \cite{Zhang-capacities-Chiu-Tamarkin,Zhang-S1-Chiu-Tamarkin} as computations of $    \HH_{\sT}^\bullet(\sT(U),(-\infty,L])$. We emphasize that those computations are done via the Fourier transform formula of projectors in Section \ref{fourier projector}. Precisely, we have: 1) In \cite{Zhang-capacities-Chiu-Tamarkin}, we showed how to compute $\HH_{\sT}^\bullet(\sT(U),(-\infty,L])$ for convex toric domains $U\subseteq T^*\RR^n$ by approximating the convex dual of the moment map image by rectangles, and computing explicitly for the rectangular approximation. 2) In \cite{Zhang-S1-Chiu-Tamarkin}, we showed the following Viterbo-type isomorphism:
\begin{theorem}[{\cite[Subsection 4.3]{Zhang-S1-Chiu-Tamarkin}}]
Let $M$ be a compact oriented Riemannian manifold, $D^*M$ its open unit disk bundle, and $L\geq 0$.  Then there is an isomorphism
\begin{equation*}
   \HH_{\sT}^\bullet(\sT(D^*M),(-\infty,L]) \cong H_{n-*}(\mathcal{L}_{\leq L}M,\ZZ) .
\end{equation*}
intertwining theg Gerstenhaber product on $\HH_{\sT}^\bullet(\sT(D^*M),(-\infty,L]) $ with the loop product. 
\end{theorem}

\subsection{\texorpdfstring{Wrapping formula for $\HH^\bullet$}{}}

We will now give some formulas expressing Hochschild cohomologies in terms of the `wrapping' formula for the projector.  We will need the formula \eqref{wrapping formula for C out} for the comparison to symplectic cohomology in the next section.

\begin{lemma} For $-\infty \leq a <b < \infty$, we have an isomorphism 
    \begin{equation} \label{wrapping sheaf hom formula of CLplus}
    \HH_{\sT}^\bullet(\Sh_{U^c}(M;\sT),(a,b])=  \varprojlim\nolimits_{\alpha}\Hom(1_{M^2 \times [-b,-a)},\sHom_\sT(\cK(\varphi^{H_\alpha})|_1,\cK(\id))),
 \end{equation}
 where the inverse system is induced by 
the continuation maps.
 \end{lemma}
\begin{proof}
First, by Proposition \ref{proposition: HH co using sheaf} and Remark \ref{remark: CY true for the "out" category}, we have
\[\HH_{\sT}^\bullet(\Sh_{U^c}(M;\sT),(a,b])=\Hom(1_{M^2 \times [-b,-a)},\sHom_\sT(Q_U,1^\sT_{\Delta_{M}})).\]
Then we plug $\cK(\id)=1^\sT_{\Delta_{M}}$ and the wrapping formula for $Q_U$ (Proposition \ref{wrapping formula for projector}):
\begin{align*}
   &\Hom(1_{M^2\times [-b,-a)},\sHom_\sT(Q_U,1^\sT_{\Delta_{M}}))\\
   =&\Hom(1_{M^2\times [-b,-a)},\sHom_\sT(\varinjlim\nolimits_{\alpha}\cK(\varphi^{H_\alpha})|_1,\cK(\id)))\\
  =&\varprojlim\nolimits_{\alpha}\Hom(1_{M^2\times [-b,-a)},\sHom_\sT(\cK(\varphi^{H_\alpha})|_1,\cK(\id))).\qedhere
\end{align*}

\end{proof}

\begin{lemma}\label{lemma: cohomology RCT vanishing}When $a<0$ and $U$ has a RCT boundary, we have
\[
\HH^\bullet_{\sT}(\sT(U), (-\infty,a])=0=\HH^\bullet_{\sT}(\Sh_{U^c}(M;\sT),  (-\infty,a]).\]
\end{lemma}
\begin{proof}We explain the proof for $\Sh_{U^c}(M;\sT)$. The statement for $\sT(U)$ then follows from the fiber sequences \eqref{equation: action window sequence for HH}. 

Similar to Lemma \ref{lemma: negative part vanishing}, we take a particular sequence $H_\alpha$ as therein, and then use the wrapping formula \eqref{wrapping sheaf hom formula of CLplus}. Therefore, it is sufficient to prove that, for $\alpha\gg 0$,
\[\Hom(1_{M^2 \times [-a,\infty)},\sHom_\sT(\cK(\varphi^{H_\alpha})|_1,\cK(\id)))=\Hom(\cK(\varphi^{H_\alpha})|_1,\nT_{a*}1^\sT_{\Delta_M})=0.\]

We first see that
\[\begin{split}\Hom(\cK(\varphi^{H_\alpha})|_1,\nT_{a*}1^\sT_{\Delta_M})
&=\Hom(\Delta^{*}_\sT\cK(\varphi^{H_\alpha})|_1,\nT_{a*}1^\sT_{M})\\
&=\Hom( (\Delta^{*}_\sT\cK(\varphi^{H_\alpha})|_1)\otimes_\sT \omega_M^\sT ,\nT_{a*}\omega_M^\sT)\\
&=\Hom( a_{!}^\sT[(\Delta^{*}_\sT\cK(\varphi^{H_\alpha})|_1)\otimes_\sT \omega_M^\sT] ,\nT_{a*}1^\sT).
\end{split}
\]

The first and third isomorphism are adjunctions, and the second follows from that $\omega_M^\sT$ is an invertible sheaf. Then we can apply the argument of Lemma \ref{lemma: negative part vanishing}. 

We first have the following microsupport estimation
\[\dot{ss}(a_{!}^\sT[(\Delta^{*}_\sT\cK(\varphi^{H_\alpha})|_1)\otimes_\sT \omega_M^\sT])/\RR_+ \subset \mathcal{S}(H_\alpha)\footnote{In fact, the object $a_{!}^\sT[(\Delta^{*}_\sT\cK(\varphi^{H_\alpha})|_1)\otimes_\sT \omega_M^\sT]\in \sT$ is the trace of Serre twisting of the functor $\cK(\varphi^{H_\alpha})|_1\circ_\sT$. We will not use this fact, but only the formula here.   }.\]

And then $a_{!}^\sT[(\Delta^{*}_\sT\cK(\varphi^{H_\alpha})|_1)\otimes_\sT \omega_M^\sT]$ restricted to $0$ at $a\ll 0$ ($\sT$ fully faithfully embedded into $\Sh(\RR)$). So, the required vanishing is first proven for $a\ll 0$, and then proven by the microlocal Morse lemma for $a<0$.\end{proof}

We will use the following finiteness result to pass to cohomology when working over a field: 
\begin{proposition}Set $\mathbf{k}=\FF\dMod$. For a compactly supported Hamiltonian function $H$ such that the graph of the Hamiltonian diffeomorphism $\varphi^{H}_1$ is transverse to the diagonal over the support of $H$, then there exists a finite family $\mathcal{I}$ such that 
\[\pi_*^\sT \sHom_\sT(\cK(\varphi^{H})|_1,\cK(\id))=\oplus_{i\in \mathcal{I}}1_{[a_i,b_i)}[-n_i] \in \sT,\]    
\end{proposition}
where for each $i\in \mathcal{I}$, $n_i\in \ZZ$, $a_i\in \RR$ and $b_i\in \RR\cup\{\infty\}$.
\begin{proof}  
By \eqref{equation: slicewise graph, non homogeneous case}, we have the following microsupport estimation
\[\dot{ss}(\pi_*^\sT \sHom_\sT(\cK(\varphi^{H})|_1,\cK(\id))) \subset -\mathcal{S}(H) .\]

On the other hand, $\pi_*^\sT \sHom_\sT(\cK(\varphi^{H})|_1,\cK(\id))$ is constructible since $\cK(\varphi^{H})|_1$ is constructible by Theorem \ref{GKS tamarkin version}. Then the corresponding stratification on $\RR$ of $\pi_*^\sT \sHom_\sT(\cK(\varphi^{H})|_1,\cK(\id))$  is finite as jumps correspond to finitely many (under the transversality assumption) action values of fixed points of $\varphi^{H}_1$. Then the result follows from \cite[Proposition B.12]{Guillermou-Viterbo}. 
\end{proof}
\begin{remark}
    In \cite{Guillermou-Viterbo}, the corresponding result is formulated somewhat differently, by giving the decomposition of $\pi_*^\sT \sHom_\sT(\cK(\varphi^{H})|_1,\cK(\id))$ in $\Sh(\RR)$.  Thus, factors of the form $1_{\RR_{< a}}[1]$ appear. Here, we work in $\sT$, where $1_{\RR_{< a}}[1]\simeq 1_{\RR_{\geq a}}$. 
\end{remark}

\begin{corollary}\label{corollary: wrapping of Hom(F+)}Set $\mathbf{k}=\FF\dMod$. We assume that, in the wrapping formula $Q_U \simeq \varinjlim_{\alpha} \cK(\varphi^{H_\alpha})|_1$, the graph of the Hamiltonian diffeomorphism $\varphi^{H_\alpha}_1$ is transverse to the diagonal over the support of $H_\alpha$. Then, for $-\infty \leq a <b < \infty$, we have \[\Ext^*(1_{M^2\times [-b,-a)},\sHom_\sT(\cK(\varphi^{H_\alpha})|_1,\cK(\id)))\]
is finite dimensional for all $\alpha$.
In particular, we have
\begin{equation}\label{wrapping formula for C out}
     \HH_{\sT}^*(\Sh_{U^c}(M;\sT),(a,b])
     =\varprojlim\nolimits_{\alpha}\Ext^*_{\sT}(1_{M^2\times [-b,-a )},\sHom_\sT(\cK(\varphi^{H_\alpha})|_1,\cK(\id))).
\end{equation} 
\end{corollary}
\begin{proof}
The first statement follows from the non-degenerate assumption of $H_\alpha$ and that
\begin{equation} \label{wrapping formula bounded action window sheaf side}
    \Ext^*(1_{M^2\times [-b,-a )},\sHom_\sT(\cK(\varphi^{H_\alpha})|_1,\cK(\id)))
    =\Ext^*(1_{[-b,-a )},\pi_*^\sT \sHom_\sT(\cK(\varphi^{H_\alpha})|_1,\cK(\id))).
\end{equation}

The second statement follows from the fact that inverse systems of finite dimensional vector spaces satisfy the Mittag-Leffler condition.
\end{proof}

\subsection{Low action Hochschild cohomology}

We work in the embedding $\sT(T^*(M^2)) \hookrightarrow \Sh(M^2\times\RR)$,
recall that $1^\sT_{\Delta_{M}}$ is then identified as $1_{\Delta_{M}\times \RR_{\geq 0}}$. 
We identify $J^1(M^2)$ with $T^*(M^2\times\RR) \cap\{\tau=1\}$. 

For $F_1,F_2 \in \Sh(M\times\RR)$, there exists a natural morphism of bifunctors:
\begin{equation}\label{equation: muhom restriction}\Hom(F_1,F_2) = \Gamma(T^*(M\times \RR), \mhom(F_1 ,F_2) )\rightarrow \Gamma(\dot{T}^*(M\times \RR), \mhom(F_1 ,F_2)).
    \end{equation}
    
We will show below that under certain geometric conditions introduce below, the fiber of this map is $\Hom(F_1,\nT_{-\epsilon*}F_2)$. 

Precisely, consider two Legendrians $\Lambda_1, \Lambda_2 \subset J^1M $ with $\Lambda \cap \nT_\epsilon(\Lambda')=\varnothing$ for small $\epsilon > 0$. For $F_1,F_2 \in \Sh(M\times \RR)$ satisfying  $ss(F_i)\subset \RR_{\geq 0}\Lambda_i$ ($i=1,2$), we have $\supp(\mhom(F_1 ,F_2)) \subset ss(F_1)\cap ss(F_2) \subset \RR_{\geq 0}(\Lambda_1\cap \Lambda_2) $ by \cite[Cor. 6.4.3]{Kashiwara-Schapira-SoM}. Then we have \[\Gamma( \dot{T}^*(M\times \RR), \mhom(F_1 ,F_2))=\Gamma( T^*_{>0}(M\times \RR), \mhom(F_1 ,F_2))=\Gamma( J^1M, \mhom(F_1 ,F_2)),\] where the last equal is true because $T^*_{>0}(M\times \RR) \rightarrow J^1M $ is a trivial $\RR_>$-bundle. Compose with the natural morphism \eqref{equation: muhom restriction}, we have a morphism
\begin{equation}\label{equation: sabloff origin}
   \Hom(F_1,F_2)\rightarrow \Gamma( J^1M, \mhom(F_1 ,F_2)). 
\end{equation}
Then \cite[Theorem 4.3]{Kuo-Li-duality-spherical} proves that, under the assumption $\Lambda_i$ are compact and $\supp(F_1) \cap \supp(F_2)$ is compact in $M\times \RR$, we have the fiber sequence for small enough $\epsilon >0$:
\[
    \Hom(F_1,\nT_{-\epsilon*}F_2)  \rightarrow \Hom(F_1,F_2)\xrightarrow{\eqref{equation: sabloff origin}} \Gamma( J^1M, \mhom(F_1 ,F_2)).
\]

\begin{lemma}\label{lemma: corollary of Sato-Sabloff}
For all compactly support Hamiltonian function $H$ on $T^*M$ and small enough $\epsilon \geq 0$, we have
\[\Hom(\cK(\varphi^{H})|_1,1_{\Delta_{M}\times \RR_{\geq \epsilon}})  \simeq \Gamma( J^1(M^2), \mhom(\cK(\varphi^{H})|_1 ,1_{\Delta_{M}\times \RR_{\geq 0}})).\]
The isomorphism is functorial with respect to $\cK(\varphi^{H})|_1$.
\end{lemma}

\begin{proof}[Proof of Lemma \ref{lemma: corollary of Sato-Sabloff}]
We set $F_1=\cK(\varphi^{H})|_1$ and $F_2=1_{\Delta_{M}\times \RR_{\geq 0}}$. 

First of all, by microsupport estimation in the proof of Lemma \ref{lemma: cohomology RCT vanishing}, we know that, for small enough $\epsilon >0$,
\[\Hom(F_1,\nT_{-\epsilon*}F_2)=0,\quad \Hom(F_1,\nT_{\epsilon*}F_2)\xrightarrow{\simeq} \Hom(F_1,F_2).\]

Equations \eqref{equation: slicewise graph, non homogeneous case} show that $ss(F_i)\subset \{\tau\geq 0\}$. Then it remains to check the natural morphism 
\[\Hom(F_1,F_2)  \rightarrow\
\Gamma( J^1(M^2), \mhom(F_1,F_2))\]
has the trivial fiber.

We are not able to use the result of \cite[Theorem 4.3]{Kuo-Li-duality-spherical} directly because of the absence of required compactness. Here, we make a reduction to apply \cite[Theorem 4.3]{Kuo-Li-duality-spherical} by employing a trick comes from \cite[Proposition 4.10]{Ike}.

As it is explained in \cite[Section 4]{Guillermou-Viterbo}, for compactly supported functions $H$, $\cK(\varphi^{H})|_{1}$ is isomorphic to $1_{\Delta_M\times \RR_{\geq 0}}$ outside a compact set of $M^2\times \RR$. Therefore, there exists a relatively compact open set $U \subset M^2$ and $ A \in (1,\infty)$ such that $(F_1)_{M^2 \times [A,\infty)} \simeq 1_{\Delta_M\times \RR_{\geq A}}$. We set $G_1=(F_1)_{M^2 \times (-\infty,A)}$. It has a compact support. And we have the fiber sequence by excision sequence
\[G_1 \rightarrow F_1 \rightarrow 1_{\Delta_M\times \RR_{\geq A}}.\]

We now compute fibers of \eqref{equation: sabloff origin} for $1_{\Delta_M\times \RR_{\geq A}}$ or $G_1$ in the place of $F_1$.

For $1_{\Delta_M\times \RR_{\geq A}}$, we do this by a direct computation, which shows that the fiber sequence
\[\Hom(1_{\Delta_M\times \RR_{\geq A}},\nT_{-\epsilon*}F_2)  \rightarrow \Hom(1_{\Delta_M\times \RR_{\geq A}},F_2) \rightarrow \Gamma( J^1(M^2), \mhom(1_{\Delta_M\times \RR_{\geq A}},F_2)) \]
is exactly given by
\[\Gamma(M^2,1) \simeq \Gamma(M^2,1) \rightarrow 0.\]

For $G_1$, we adopt the proof of \cite[Theorem 4.3, Propositions 4.6]{Kuo-Li-duality-spherical}. We first notice that now $G_1$ has a compact support, so $\supp(G_1)\cap \supp(F_2)$ is compact. Then we aware of that, in the proof of \textit{loc. cit.}, compactness of microsupports guarantee correctness of certain microsupport estimations, which could be achieved for $G_1$ and $F_2$ because we have good control of their microsupport at infinity (they are both conormal of diagonal). Then the argument in \cite[Theorem 4.3]{Kuo-Li-duality-spherical} works for $G_1,F_2$ to obtain the fiber sequence for small $\epsilon>0$:  
\[ \Hom(G_1,\nT_{-\epsilon*}F_2)  \rightarrow \Hom(G_1,F_2) \rightarrow \Gamma( J^1(M^2), \mhom(G_1 ,F_2)) .\]

Lastly, we apply the morphism of functor $\Hom(-,F_2) \rightarrow \Gamma( J^1(M^2), \mhom(- ,F_2))$ to the fiber sequence $G_1 \rightarrow F_1 \rightarrow 1_{\Delta_M\times \RR_{\geq A}}$ to see that the fiber of  
\[  \Hom(F_1,F_2) \rightarrow \Gamma( J^1(M^2), \mhom(F_1 ,F_2))\]
is given by the fiber of the morphism
\[\Hom(G_1,\nT_{-\epsilon*}F_2) \rightarrow \Hom(1_{\Delta_M\times \RR_{\geq A}}[-1],\nT_{-\epsilon*}F_2)\]
because limits are commute. We finish the proof by notice that the fiber of the last morphism is $\Hom(F_1,\nT_{-\epsilon*}F_2)=0$.\end{proof}

\begin{remark}
    A different proof of Lemma \ref{lemma: corollary of Sato-Sabloff} can be obtained by modifying the proof of \cite[Theorem 12.4.7]{Guillermou-book}. 
\end{remark}

\begin{corollary}\label{corollary: corollary of corollary of Sato-Sabloff}For $U \subset T^*M$ has a RCT boundary, and sufficiently small  $\epsilon \geq 0$, we have the commutative diagram:
\[\begin{tikzcd}
{\Hom(Q_U,1_{\Delta_{M}\times \RR_{\geq \epsilon}}) } \arrow[d] \arrow[r, "\simeq"] & { \varprojlim_{\alpha}\Gamma(J^1(M^2), \mhom(\cK(\varphi_{H_\alpha})|_1 ,1_{\Delta_{M}\times \RR_{\geq 0}}))} \arrow[d] \\
{\Hom(1_{\Delta_{M}\times \RR_{\geq 0}},1_{\Delta_{M}\times \RR_{\geq \epsilon}}) } \arrow[r, "\simeq "]                                     & {\Gamma(J^1(M^2), \mhom(1_{\Delta_{M}\times \RR_{\geq 0}} ,1_{\Delta_{M}\times \RR_{\geq 0}}))}  .                               
\end{tikzcd}\]
\end{corollary}
\begin{proof}We apply Lemma \ref{lemma: corollary of Sato-Sabloff} to the morphism $1_{\Delta_{M}\times \RR_{\geq 0}} \rightarrow Q_U \cong \varinjlim_{\alpha}\cK(\varphi^{H_\alpha})|_1$.
\end{proof}
Next, we compute $\mhom$ involved. For a  compactly supported Hamiltonian function $H$, recall the ``Legendrian graph" $\widehat{\Lambda}_{\varphi^H_1}$ defined in Equation \eqref{equation: slicewise graph, non homogeneous case}. We set $Z_H=\widehat{\Lambda}_{\id} \cap \widehat{\Lambda}_{\varphi^H_1} \subset J^1(M^2)$, and $i_H: Z_H \rightarrow J^1(M^2)$ be the closed inclusion. By definition $(q,-p,q,p,0)\in Z_H$ if and only if $(q,p)$ is a constant $1$-periodic closed orbit of the flow $\varphi_z^H$.
\begin{lemma}\label{lemma: muhom computation}For a compactly supported Hamiltonian function $H$ and assume $Z_H$ is a (codimension 0) submanifold with boundary in $\Delta_{T^*M}\times \{0\}$, then we have
\[\mhom(\cK(\varphi^{H})|_1 ,1_{\Delta_{M}\times \RR_{\geq 0}}) |_{J^1(M^2)}\simeq i_{H*}\omega_{{Z_H}}[-2n].\] 
\end{lemma}
\begin{proof}

By the condition on $Z_H$, we have that Legendrian submanifolds $\widehat{\Lambda}_{\varphi^H_1}$ and $\widehat{\Lambda}_{\id}$ (c.f. Equation \eqref{equation: slicewise graph, non homogeneous case} ) intersect at $Z_H$ cleanly. By \cite[Lemma 6.14]{Guillermou-quantization-cotangent}, we then see that $\mhom(\cK(\varphi^{H})|_1 ,1_{\Delta_{M}\times \RR_{\geq 0}}) |_{J^1(M^2)}$ is a locally constant sheaf supported on $Z_H$. 

We first assume $H$ to be nonnegative. Then we have a morphism $1_{\Delta_{M}\times \RR_{\geq 0}} \rightarrow \cK(\varphi^{H})|_1 $ by Corollary \ref{GKS tamarkin version} and the morphism
\begin{equation}\label{equation: mu-hom monotonic morphism}
    \mhom(\cK(\varphi^{H})|_1 ,1_{\Delta_{M}\times \RR_{\geq 0}}) |_{J^1(M^2)}  \rightarrow  \mhom(1_{\Delta_{M}\times \RR_{\geq 0}} ,1_{\Delta_{M}\times \RR_{\geq 0}}) |_{J^1(M^2)}.
\end{equation}

By \cite[Theorem 4.14, Proposition 4.15]{Ike}, we have the following direct computation
\[\mhom(1_{\Delta_{M}\times \RR_{\geq 0}} ,1_{\Delta_{M}\times \RR_{\geq 0}}) |_{J^1(M^2)}=\omega_{\Delta_{T^*M}\times \{0\}}[-2n].\]

Now, since $\mhom(\cK(\varphi^{H})|_1 ,1_{\Delta_{M}\times \RR_{\geq 0}}) |_{J^1(M^2)}$ is a locally constant sheaf supported on $Z_H$. We will conclude by showing that \eqref{equation: mu-hom monotonic morphism} restricts to an equivalence on $Z_H$. 

To do so, we only need to check the equivalence stalk-wise. Here, we notice that as $\cK(\varphi^{H})|_z$ has (conic) Lagrangian microsupport inside an open neighberhood of $\RR_{\geq 0}Z_H$. Then, by \cite[Equation I.4.6]{Guillermou-book} for example, we can compute stalks of $\mhom(\cK(\varphi^{H})|_z ,1_{\Delta_{M}\times \RR_{\geq 0}})$ (for all $z\in\RR$) using their microstalks. Moreover, stalk of \eqref{equation: mu-hom monotonic morphism} at $(q,-p,q,p,0)\in Z_H$ is induced by the microstalk of the morphism $1_{\Delta_{M}\times \RR_{\geq 0}} \rightarrow \cK(\varphi^{H})|_1 $ at $(q,-p,q,p,0)\in Z_H$, which we are going to prove to be the identity. 

It is easy to see that the GKS quantization $\cK(\varphi^{H})$ is simple, i.e. microlocally rank 1, for all Hamiltonians $H$ (or see \cite[Theorem II.1.1, Remark II.1.2]{Guillermou-book}).  It shows that all microstalks at $(z,\zeta, q,-p,q,p,t)$ in \eqref{equatio: GKS T-linear total graph}, i.e.,
 \[\dot{ss}(\cK(\varphi))/\RR_+\subseteq \{(z, - H(z,{\varphi}_{z}(q,p))  ,(q,-p),{\varphi}_{z}(q,p),-S_{H}^z(q,p) ) : z\in I, (q,p)\in T^*M\}\]
are $1$ for all $z\in \RR$ upto a suitable degree shifting. 

It is also shown in \textit{loc. cit. }that the microstalks of $\cK(\varphi)|_z$ at $( q,-p,q,p,t) \in Z_H$ are naturally isomorphic to microstalks at $(z,\zeta, q,-p,q,p,t)$ of $\cK(\varphi)$. In the $H$ positive case, the isomorphisms between microstalks are induced by the morphism $1_{\Delta_{M}\times \RR_{\geq 0}} \rightarrow \cK(\varphi^{H})|_z $. Therefore, it remains to see that the microstalks of $\cK(\varphi)$ at $\dot{ss}(\cK(\varphi))/\RR_+ \cap T^*\RR_z\times Z_H$ are naturally identical. To see this, we observe that $\dot{ss}(\cK(\varphi))/\RR_+\cap T^*\RR_z\times Z_H$ is in a connected component of $\dot{ss}(\cK(\varphi))/\RR_+$. Then we can conclude by \cite[Corollary 7.5.7]{Kashiwara-Schapira-SoM}.\end{proof}

Now we proof the main theorem of this subsection.
\begin{theorem}\label{proposition: computation of HH at level 0}If $U\subseteq T^*M$ has a RCT boundary, then for $\epsilon \geq 0$ smaller than the minimal period of $\partial U$, we have
\[\HH^\bullet_{\sT}(\sT(U);(-\infty,\epsilon])\simeq \Gamma^{BM}_{2n-\bullet}(U,\mathbf{k}), \quad \HH^\bullet_{\sT}(\Sh_{U^c}(M;\sT);(-\infty,\epsilon])\simeq \Gamma^{BM}_{2n-\bullet}(U^c,\mathbf{k}).\]
\end{theorem}
\begin{proof}
As $U$ has contact boundary, we take a cofinal sequence $H_\alpha$ as in Subsection \ref{section: RCT estimation}. 

By construction of $H_\alpha$, the condition of Lemma \ref{lemma: muhom computation} is satisfied. Then we have 
\[{\Gamma(J^1(M^2), \mhom(\cK(\varphi_{H_\alpha})|_1 ,1_{\Delta_{M}\times \RR_{\geq 0}}))} \simeq  \Gamma_{Z_{H^\alpha }} (J^1  (M^2),\mathbf{k})[-2n].\]

So, since $\cap_{\alpha} Z_{H^\alpha }= U^c\times \{0\}$ by our construction of $H_\alpha$, we have
\[
\begin{split}
   &\varprojlim\nolimits_{\alpha}{\Gamma(J^1(M^2), \mhom(\cK(\varphi_{H_\alpha})|_1 ,1_{\Delta_{M}\times \RR_{\geq 0}}))} 
    \\\simeq &\varprojlim\nolimits_{\alpha} \Gamma_{Z_{H^\alpha}}(J^1(M^2),\mathbf{k})[-2n]
    \simeq \Gamma_{U^c\times \{0\}}(J^1(M^2),\mathbf{k})[-2n] 
    \simeq \Gamma^{BM}_{2n-\bullet}(U^c,\mathbf{k}). 
\end{split}
\]

Therefore, the diagram appears in Corollary \ref{corollary: corollary of corollary of Sato-Sabloff} is given by
\[
\begin{tikzcd}
{\Hom(Q_U,1_{\Delta_{M}\times \RR_{\geq \epsilon}})  } \arrow[d] \arrow[r, "\simeq"]                    & \Gamma^{BM}_{2n-\bullet}(U^c,\mathbf{k}) \arrow[d] \\
{\Hom(1_{\Delta_{M}\times \RR_{\geq 0}},1_{\Delta_{M}\times \RR_{\geq \epsilon}}) } \arrow[r, "\simeq"] & \Gamma^{BM}_{2n-\bullet}(T^*M,\mathbf{k}).       
\end{tikzcd}
\]

The result now follows from Proposition \ref{proposition: HH co using sheaf} and the fiber sequence  \eqref{equation: HH fiber sequences}.\end{proof}

By Corollary \ref{proposition: tw HH co using sheaf}, we have
\begin{corollary}\label{corollary: HH cohomology with oriention}Under the same condition of Proposition \ref{proposition: computation of HH at level 0}, we assume further that $M$ is $\mathbf{k}$-orientable, then we have
\[C_\epsilon(U,\mathbf{k})\simeq \Gamma^{BM}_{2n-*}(U,\mathbf{k}),\quad C^{out}_\epsilon(U,\mathbf{k})\simeq \Gamma^{BM}_{2n-*}(U^c,\mathbf{k}).\]    
\end{corollary}

\begin{remark}If we assume $U$ is $\mathbf{k}$-orientable, which is automatically true if $\mathbf{k}=R\dMod$ for a complex oriented commutative ring spectrum $R$ (for example discrete rings or the complex cobordism $\operatorname{MU}$), we have $\Gamma^{BM}_{2n-\bullet}(U,R)\simeq \Gamma( \overline{U},\partial U,R)$ by Poincar\'e duality.    
\end{remark}

\section{Comparison to  symplectic cohomology} \label{sec: guillermou-viterbo}

In this section, we take $\mathbf{k}=\FF\dMod$ for a field $\FF$. Recall that $T(T^*M)\simeq p^*TM\oplus p^*T^*M$, where $p:T^*M\rightarrow M$ is the projection. Then we have $c_1(T(T^*M))=0$, $w_2(T(T^*M))=0$. For the open set $U \subseteq T^*M$, we assume that $\overline{U}$ is a closed manifold with boundary, and the restriction of the canonical 1-form to $\partial U$ is a contact form. In particular, $\overline{U}$ is a Liouville domain. In this case we say $\partial U$ is {\em restricted contact type (RCT)}. 

For any open interval $(a,b) \subseteq \RR$ with $a,b \notin \mathcal{A}(\partial U)$,  there is a graded vector space $SH_{(a,b)}^{*}(\overline{U})$.

We want to prove the following theorem. 

\begin{theorem}\label{theorem: CT=SH}For an open set $U\subseteq T^*M$ with RCT boundary. If $L\in \RR \setminus \mathcal{A}(\partial U)$, we have an isomorphism
\[\HH_\sT^*(\sT(U);(-\infty,L])\simeq SH_{(-\infty,L)}^{*}(\overline{U}).\]    
\end{theorem}
\begin{remark}
 When $\partial U$ has good dynamics in the sense that all Reeb orbits are non-degenerated, the action spectrum $\mathcal{A}(\partial U)$ is discrete. Then, the theorem is still true for $L\in \mathcal{A}(\partial U)$ by passing to limits. 
\end{remark}

The tool we ultimately use to connect Floer homology with sheaves is   a comparison result of Guillermou and Viterbo.  Recall that for a Hamiltonian $H$, we write $\widehat{\Lambda}_{\phi_1^H}$ for the Legendrian graph of the time-1 symplectomorphism, see \eqref{equation: slicewise graph, non homogeneous case}.  Now we view this instead as an exact Lagrangian with fixed primitive (above denoted $S_H^1$).  Recall also that the Floer homology between such exact Lagrangians carries an action filtration. 

\begin{theorem}[{\cite[Theorem E.1]{Guillermou-Viterbo}}]\label{theorem: GV-comparision with HF}For a compactly supported Hamiltonian function $H$ and $a,b\in \RR$, denote $\varphi^H$ the Hamiltonian flow of $H$, we have
\begin{equation}\label{guillermou-viterbo formula}
\Ext^*(1_{M^2\times [a,b)},\sHom_\sT(\cK(\varphi^H)|_1,\cK(\id)))\simeq HF^*(\hatLam_{\varphi_1^H},\hatLam_{\id};[a,b)). 
\end{equation}
\end{theorem}

\subsection{Floer theory review (and some lemmas)}
\subsubsection{Symplectic cohomology}
We follow  \cite{Cieliebak-Frauenfelder-Oancea2010} for filtered symplectic (co)homology, except we use a cohomological grading as in \cite{Ganatra-thesis}.  We recall the definitions, 
in particular to fix conventions for gradings and filtrations. 

On $T^*M$ we take the standard symplectic form $\omega=dp\wedge dq=d \lambda$, where $\lambda=pdq$. For a Hamiltonian function $H: I\times T^*M \rightarrow \RR$, its Hamiltonian vector field is given by $\iota_{X_H}\omega=-dH$. 

The action functional for loops $x:S^1\rightarrow T^*M$ is given by
\[\mathcal{A}_H(x)=\int_{0}^1[x^* \lambda-H(z,x(z))]dz.\]
Its critical points are $1$-periodic orbits of $X_H$.  Note that if $x$ is a $1$-periodic orbit of $X_H$ with $x(0)=(q,p)$, we have
$\mathcal{A}_H(x)=S_H^1(q,p)$, where the right hand side was defined in Equation \eqref{S H formula}. 

For a non-degenerate 1-periodic orbit $x$ of $X_H$, the degree $|x|$ is define by $|x|=n-\CZ(x)$, where $\CZ(x)$ is the Conley-Zehnder index of $x$ (\cite{Robbin-Salamon-index,Salamon1999lectures}). As $c_1(T(T^*M))=0$, the Conley-Zehnder index takes values in $\ZZ$. 

For an almost complex structure $J$ on $T^*M$, we say $J$ is compatible with $\omega$ if $g(v,w)=\omega(v,Jw)$ is a Riemannian metric. Then we have $X_H=J\nabla_g H$, and the positive gradient flow of $\mathcal{A}_H$ gives the Floer equation 
\[u_s+J(u_t- X_H)=u_s+Ju_t+\nabla H=0.\]

For two non-degenerate 1-periodic orbits $x_{\pm}$ of $X_{H}$, we define $\widehat{\cM}(x_-,x_+)$ to be the space of solutions of the Floer equation with $\lim _{s\rightarrow \pm\infty} u(s,t)=x_{\pm}(t)$. The space admits a $\RR$-action by $a\cdot u(s,t)=u(s+a,t)$, and its quotient is denoted by $\cM(x_-,x_+)=\widehat{\cM}(x_-,x_+)/\RR.$ For generic convex at infinity $J_t$, the moduli space $\cM(x_-,x_+)$ is a manifold of dimension $|x_-|-|x_+|-1$.

For $u\in \cM(x_-,x_+)$, the energy $ E(u) = \int
|\partial_su|^2 ds\wedge dt$ satisfies:
\[E(u) = \mathcal{A}_H(x_+)-\mathcal{A}_H(x_-)\geq 0.\]

We write $\cS(H) \subseteq \RR$ for the action spectrum -- values of $\cA_H$ on 1-periodic orbits.  
For $-\infty \leq a<b \leq \infty$, and we denote $\cH^{(a,b)}$ the set of all Hamiltonian functions $H$ on $T^*M$ such that $a,b \notin \cS(H)$, and all 1-periodic orbits of $X_H$ with action $a<\mathcal{A}(x)<b$ are non-degenerate. 

For $H\in \cH^{(a,b)}$, the Floer complex $CF^*_{<b}(H)$ is the graded $\FF$-vector space generated by all 1-periodic orbits of $X_H$ with action $\mathcal{A}(x)<b$.  One writes:\footnote{It might be more reasonable to denote this as $CF^*_{[a,b)}(H)$.  We have nevertheless followed the notation in \cite{Cieliebak-Frauenfelder-Oancea2010}; the difference being immaterial as $a$ and $b$ are anyway forbidden from being in the action spectrum.}  
\[CF^*_{(a,b)}(H):=CF^*_{<b}(H)/CF^*_{<a}(H).\]

As usual, the Floer differential $\delta: C^q_{(a,b)}(H) \rightarrow C^{q+1}_{(a,b)}(H)$ counts holomorphic strips: 
\[\delta x_+= \# \cM(x_-,x_+) x_-.\]

The standard argument shows that $\delta^2=0$; we write
\[HF^*_{(a,b)}(H):=H^*(CF^*_{(a,b)}(H),\delta).\]
\begin{remark} So long as $0 \notin (a,b)$, we may and will make the same definition for compactly supported $H$.  (Any degenerate orbits of action $0$ do not enter into the definition of the complex.) 
\end{remark}

It is immediate from the definition that for $a<b<c$ such that corresponding Floer cohomology can be defined, we have a long exact sequence
\begin{equation}\label{eq: filtered long exact seq of HF}
    HF^*_{(a,b)}(H) \rightarrow HF^*_{(a,c)}(H) \rightarrow HF^*_{(b,c)}(H) \xrightarrow{+1}.
\end{equation}

If $H_-\leq H_+$, then a monotone increasing homotopy induces a chain map, 
\[(CF^*_{(a,b)}(H_-),\delta)\rightarrow (CF^*_{(a,b)}(H_+),\delta),\]
which is independent with the homotopy on cohomology.

We write $\cH_{\infty}^{a,b}(U) \subseteq \cH^{a,b}$ for the Hamiltonians which are linear at infinity and nonpositive on $U$.  One takes by definition:
\[SH^*_{(a,b)}(\overline{U})=\varinjlim_{H\in \cH_{\infty}^{a,b}(U)}HF^*_{(a,b)}(H).\]

\begin{proposition}[{\cite[Proposition 1.4]{Viterbo-functorsI}}]For $\delta, \epsilon>0$ and small enough, we have 
\[SH^*_{(-\delta,\epsilon)}(\overline{U})\simeq H^*(\overline{U},\partial {U},\FF ).\]   
\end{proposition}

Now let us take $L>\epsilon>0 >-\delta$.  Taking direct limits in  \eqref{eq: filtered long exact seq of HF} yields: 
\begin{equation}\label{eq: filtered long exact seq of SH}
    H^*(\overline{U},\partial {U},\FF ) \rightarrow SH^*_{(-\delta,L)}(\overline{U}) \rightarrow SH^*_{(\epsilon,L)}(\overline{U}) \xrightarrow{+1}.
\end{equation}

On the other hand, for $a,b$ with $0\notin (a,b)$, 
we write $\cH_c^{a,b}(U)\subseteq \cH^{a,b}$ for the  Hamiltonians compactly supported in $U$, and  
\[\widetilde{SH}^*_{(a,b)}(\overline{U})=\varprojlim_{H\in \cH_c^{a,b}(U)}HF^*_{(a,b)}(H).\]

When defining symplectic homology using different types of Hamiltonians, one should be careful on the direction of the (co)limit. It depends on certain conventions about grading, direct of continuation maps and type of Hamiltonians. In the introduction, we use one definition of $SH$ involving a colimit for Hamiltonian convergent to infinity outside of the domain; here we introduce another definition $\widetilde{SH}$ involving a limit for compactly supported Hamiltonian convergent to negative infinity inside the domain. This kind of discrepancy has shown up in \cite[Section 2.6]{Cieliebak-Frauenfelder-Oancea2010}.  We use a variant of \cite[Proposition 2.5]{Cieliebak-Frauenfelder-Oancea2010}, for which we thank Kai Cieliebak and Alexandru Oancea:

\begin{proposition}\label{proposition: Cieliebak-Oancea}For $L> \epsilon > 0$, we have
\[\widetilde{SH}_{(\epsilon,L]}^{*+1}(\overline{U})\simeq SH_{(\epsilon,L]}^{*}(\overline{U}).\] 
\end{proposition}
\begin{proof}
First, for a big-enough $\beta$, we have
\[\widetilde{SH}^*_{(a,b)}(\overline{U})=\varprojlim\nolimits_{\alpha}  HF^{*}_{(\epsilon,L]}(-H_\alpha)= HF^{*}_{(\epsilon,L]}(-H_\beta).\]

Consider the following function $K$.
\begin{figure}[htbp]
    \centering
\begin{tikzpicture}[xscale=1.2,yscale=0.8]
\draw[->] (-1,0)--(4,0)node[anchor=north] {$x$};
\draw[->] (0,-3.6)--(0,2.5);
\draw[fill=black] (0,0) circle (0.05);
\draw (0,0)node[anchor=north east] {$O$};

\draw[rounded corners] (-0.5,-3)--(1.5,-3)--(1.8,-2.5)--(2,-0.2)--(2.8,0.2)--(3.2,0.5)--(3.8,2);
\draw[red] (2.1,-0.15) circle (0.1);
\node[anchor=east] at (2,-0.3) {I};
\draw[blue] (2.7,0.15) circle (0.1);
\node[anchor=east] at (2.8,0.5) {II};
\draw[dashed] (1.8,-2.4)--(3.2,0.55);
 \end{tikzpicture}\label{fig:enter-label}
\end{figure}
The function below $0$, $K$ equals to $-H_\beta$; above $0$, $K$ is linear at infinity with a big enough slope. 

Closed orbits of $\varphi^K_1$ could be classified in the following types:

In bottom of the function, action of orbits will big enough, so none of them contributes to Floer homology. Inside the red circle, they contribute closed orbits generates $CF^*_{(\epsilon,L]}(K)$ that comes from $\varphi^{-H_\beta}_1$, we call them type I; inside the blue circle, they contribute closed orbits generates $CF^*_{(\epsilon,L]}(K)$ that give $SH_{(\epsilon,L]}^{*}(\overline{U})$, we call them type II.

Let us denote subcomplexes generated by corresponding orbits as $C_I^*, C_{II}^*$. Then we have
\[H^*(C_I^*) \simeq H^*_{(\epsilon,L]}(-H_\beta),\quad H^*(C^*_{II}) \simeq SH^{*}_{(\epsilon,L]}(\overline{U}).\]

On the other hand, they form a short exact sequence 
\[0\rightarrow C_{I}^* \rightarrow CF^*_{(\epsilon,L]}(K) \rightarrow C_{II}^* \rightarrow 0.\]
It remains to show that $CF^*_{(\epsilon,L]}(K)$ is acyclic. 
This can be done by comparing with the symplectic
homology as computed by the Hamiltonians indicated
by the dashed line, as in \cite[Prop. 2.5]{Cieliebak-Frauenfelder-Oancea2010}
\end{proof}

\subsubsection{Lagragian Floer cohomology}
We now fix conventions for  Lagrangian Floer theory,  following \cite{Viterbo}. Fix two transversely intersecting compact exact Lagrangians, $L_0$ and $L_1$.  We fix primitives:  $f_i: L_i \to \RR$ with
$df_i = \lambda|_{L_i}$.  We also fix grading and spin structures on the Lagrangians.

We write $f_{01}(p):=f_1(p)-f_0(p)$ for $p\in L_0\cap L_1$. The Maslov class enables us to define a $\ZZ$-grading $|p|\in \ZZ$ for $p\in L_0\cap L_1$.  We fix the sign of the grading as follows.  If $x$ is a non-degenerate 1-periodic orbit of $X_H$ for of a Hamiltonian function $H$, then there is a corresponding intersection point for $L_0$ the graph of the time-1 return map, and $L_1$ the diagonal.  We ask that the degree of the intersection point agrees with the degree in Hamiltonian Floer cohomology, as given above.

For a compatible almost complex structure $J$ and intersection points $p_{\pm}$, one considers the moduli  space
$\widehat{\mathscr{M}}(p_-,p_+)$  of solutions of the $J$-holomorphic curve equation: $u_s+Ju_t=0$ with $u(s,i)\in L_i$ and $\lim _{s\rightarrow \pm\infty} u(s,t)=p_{\pm}$; quotient by the translation action to get the moduli space ${\mathscr{M}}(p_-,p_+)=\widehat{\mathscr{M}}(p_-,p_+)/\RR$. For generic convex-at-infinity $J_t$, one has that $\mathscr{M}(p_-,p_+)$ is a manifold of dimension $|p_-|-|p_+|-1$. 

For $-\infty \leq a<b \leq \infty$, the Floer complex $CF^*_{J}(L_0,L_1;a)$ is the graded $\FF$-vector space generated by $p\in L_0\cap L_1$ with $f_{01}(x)\geq a$.  One writes 
\[CF^*_{J}(L_0,L_1;[a,b)):=CF^*_{J}(L_0,L_1;a)/CF^*_{J}(L_0,L_1;b).\]

We define the Floer differential $\delta: CF^*_{J}(L_0,L_1;[a,b)) \rightarrow CF^*_{J}(L_0,L_1;[a,b))$ such that
\[\delta p_+= \# \mathscr{M}(p_-,p_+) p_-.\]
The standard argument shows that $\delta^2=0$, and we denote the homology as: 
\[HF^*(L_0,L_1;[a,b)):=H^*(CF^*_{J}(L_0,L_1;[a,b)),\delta).\]

In $T^*(M\times M)$, the graph of a compactly supported Hamiltonian $\varphi^H_1$ defines an exact Lagrangian embedding whose primitive is $S_H^1$. We abuse notation $\hatLam_{\varphi_1^H}$ to represent both Legendrian we defined in \eqref{equation: slicewise graph, non homogeneous case}, and corresponding exact Lagrangian. Since $T^*M$ has trivial $c_1$ and $w_2$, the diagonal $\hatLam_{\id}$ define an exact Lagrangian brane, then also $\hatLam_{\varphi_1^H} $ as it is a Hamiltonian deformation of $\hatLam_{\id}$. We also use $\hatLam_{\varphi_1^H}$ to denote the corresponding exact Lagrangian brane. However, as $\hatLam_{\varphi_1^H}$ always have degenerate intersection at infinity, we need to define $HF^*(\hatLam_{\varphi_1^H},\hatLam_{\id};[a,b))$ ($HF^*(\hatLam_{\id},\hatLam_{\varphi_1^H};[a,b))$) by taking a $C^2$-small perturbation on Lagrangians and an almost complex structure convex-at-infinity as \cite[Appendix E]{Guillermou-Viterbo}, which also explains well-definedness when $a,b$ are finite. 

On the other hand, we will only use $HF^*(\hatLam_{\varphi_1^H},\hatLam_{\id};[a,b))$ in the case that $0$ is not in the action window $(a,b)$ (we also assume $a,b$ are not in the action spectrum of $X_H$), and the non-constant 1-periodic points of $X_H$ are non-degenerate.

Let us recall the relation between Lagrangian and Hamiltonian Floer cohomology:

\begin{proposition}[{\cite[lemma 3.2]{Viterbo-functorsII}, \cite[Section 5.2.]{Biran-Poltervich-Salamon}}]\label{proposition: LHF=HHF}
Fix a compactly supported function $H: T^*M \to \RR$.  Let $\phi_1^H$
be the time-1 map. Fix any $a,b$ such that $a,b$ are not in the action spectrum of $H$, and the interval $[a, b)$ does not containing $0$. 
Then
\[HF^*(\hatLam_{\id},\hatLam_{\varphi_1^H};[a,b))\simeq HF^{*}_{(-b,-a)}(H).\]
\end{proposition}
\begin{remark}
    The assumption that $0$ is not in the action window is only required for the trivial free homotopy class of loops (cf. \cite[Remark 4.4.2]{Biran-Poltervich-Salamon}). The minus sign in Hamiltonian Floer homology comes from a difference in our action filtration conventions between the two sides. 
\end{remark}

\begin{corollary}
    \label{lemma: two Floer comp.}For a compactly supported Hamiltonian functions $H$, and $L> \epsilon > 0$ such that $L,\epsilon$ are not in the action spectrum of $-H$, we have
\[HF^*(\hatLam_{\varphi_1^H},\hatLam_{\id};[-L,-\epsilon))\simeq HF^{*}_{(\epsilon,L)}(-H).\]
\end{corollary}
\begin{proof}The Hamiltonian map $\id\times \varphi_1^H$ on $T^*(M\times M)$ induces an isomorphism
\[HF^*(\hatLam_{\varphi_1^H},\hatLam_{\id};[-L,-\epsilon))\simeq HF^{*}( \hatLam_{\id},\hatLam_{\varphi_1^{-H}};[-L,-\epsilon)).\qedhere\]
\end{proof}

\subsection{Proof of Theorem \ref{theorem: CT=SH}}
Now, we first take $L>\epsilon > 0$ with $\epsilon>0$ small enough. Recall the 9-diagram in Lemma \ref{lemma: 9-diagram}. We set $H^*\HH^\bullet=\HH^*$, then corresponding cohomology diagram is
\[\begin{tikzcd}
{ \HH_{\sT}^*(\Sh_{U^c}(M;\sT),(-\infty,\epsilon])   }  \arrow[d] \arrow[r]   & {\HH_{\sT}^*(\Sh_{U^c}(M;\sT),{(-\infty,L]})} \arrow[d] \arrow[r] & {\HH_{\sT}^*(\Sh_{U^c}(M;\sT),{(\epsilon,L]})} \arrow[d] \arrow[r, "+1"] & {} \\
{H^{*}(T^*M,\FF)} \arrow[d] \arrow[r, "\simeq"] & {H^{*}(T^*M,\FF)} \arrow[d] \arrow[r]                    & 0 \arrow[d] \arrow[r, "+1"]                               & {} \\
{\HH_{\sT}^*(\sT(U),(-\infty,\epsilon])   }  \arrow[d, "+1"] \arrow[r]   & {\HH_{\sT}^*(\sT(U),{(-\infty,L]})} \arrow[d, "+1"] \arrow[r]                   & {\HH_{\sT}^*(\sT(U),{(\epsilon,L]})}      \arrow[d, "+1","\simeq"'] \arrow[r, "+1"] & {} \\
{}                                                          & {}                                                                   & {}                                                        & .  
\end{tikzcd}\]

We will now substitute various now-established results into the diagram.

Theorem \ref{proposition: computation of HH at level 0} determines the first column (note $U$ is $\FF$-orientable): we showed that 
\begin{equation}\label{eq: small action comparision}
   \HH_{\sT}^*(\sT(U),(-\infty,\epsilon])    \simeq H^{BM}_{2n-*}(U,\FF)\simeq H^*(\overline{U},\partial U,\FF). 
\end{equation}

Now, we use Theorem \ref{theorem: GV-comparision with HF} and Corollary \ref{corollary: wrapping of Hom(F+)}. For all $a<b$, 

\[\HH_{\sT}^*(\Sh_{U^c}(M;\sT),{(a,b]})
    =\varprojlim\nolimits_{\alpha}HF^*(\hatLam_{\varphi_1^{H_\alpha}},\hatLam_{\id};[-b,-a)).\]

In particular, for $L>\epsilon>0$, we use Corollary \ref{lemma: two Floer comp.} to see that 
\[\HH_{\sT}^*(\Sh_{U^c}(M;\sT),{(\epsilon,L]}) \simeq \widetilde{SH}_{(\epsilon,L]}^{*}(\overline{U}).\]

Then the third column of the 9-diagram together with Proposition \ref{proposition: Cieliebak-Oancea} gives: 
\begin{equation}\label{eq: postive action comparision}
    \HH_{\sT}^*(\sT(U),{(\epsilon,L]})\simeq \HH_{\sT}^{*+1}(\Sh_{U^c}(M;\sT),{(\epsilon,L]})
    =\widetilde{SH}_{(\epsilon,L]}^{*+1}(\overline{U})\simeq SH_{(\epsilon,L]}^{*}(\overline{U}).
\end{equation}

Then we plug results of \eqref{eq: small action comparision} and \eqref{eq: postive action comparision} into the third row, we have a long exact sequence
\begin{equation}\label{equation: almost final l.e.s}
 H^{*}(\overline{U},\partial U,\FF) \rightarrow \HH_{\sT}^*(\sT(U),{(-\infty,L]}) \rightarrow  SH_{(\epsilon,L]}^{*} (\overline{U})\xrightarrow{+1}.   
\end{equation}

On the other hand, the isomorphism of Theorem \ref{theorem: GV-comparision with HF} is compatible with filtration. Then, the long exact sequence \eqref{equation: almost final l.e.s} is compatible with filtration long exact sequence \eqref{eq: filtered long exact seq of SH} of symplectic cohomology:
 \[H^{*}(\overline{U},\partial U,\FF) \rightarrow  
 SH_{(-\delta,L]}^{*}(\overline{U})\rightarrow      SH_{(\epsilon,L]}^{*}(\overline{U})\xrightarrow{+1}.\]
 
 Then we conclude by comparing the long exact sequence \eqref{equation: almost final l.e.s} with the action-filtration long exact sequence of symplectic coholomogy. That is, for $L>0$, we have  
\[ \HH_{\sT}^*(\sT(U),{(-\infty,L]}) \simeq SH_{(-\delta,L)}^{*}(\overline{U}) \simeq  SH_{(-\infty,L)}^{*}(\overline{U}),\]
where the last isomorphism follows from the fact that no 1-periodic orbit for $X_{H_\alpha}$ with action $<-\delta$ with $\delta\gg 0$ for the choice cofinal sequence $H_\alpha$ (cf. \cite[p. 993]{Viterbo-functorsI}). For the same reason, 
both sides vanish if $L<0$. This completes the proof. $\square$

\begin{remark}We could have also argued for Equation \eqref{eq: small action comparision}  by applying Theorem \ref{theorem: GV-comparision with HF} to the action window $(-\delta,\epsilon]$. But now, since $0\in (-\delta,\epsilon]$, we should be careful that we need to take a $C^2$-small Morse perturbation of the diagonal $\hatLam_{\id}$ outside of $U$ to define Lagrangian Floer theory. In this case, one can conclude by adapting a result of Po\'zniak \cite[Theorem 5.2.2]{Biran-Poltervich-Salamon}, \cite[Theorem 3.4.11]{Pozniak-Thesis}. One advantage of this argument is that it is easy to see that \eqref{equation: almost final l.e.s} is compatible with the action window long exact sequence of symplectic cohomology.  
\end{remark}

\section{Comparison to generating function homology}

Using equation \eqref{wrapping formula for C out}, we can compare the generating function (GF) homology for $U$, defined by \cite{Traynor_SH-via-GF} based on the work of \cite{Viterbo-GF}, with $\HH_{\sT}^*(\Sh_{U^c}(M;\sT),(a,b]) $.  First we recall the definition; we follow the conventions of \cite{Fraser-Sandon-Zhang}.

For a symplectomorphism $\varphi: \RR^{2n}\to \RR^{2n}$, its graph $\operatorname{gr}(\varphi)=\{(q,p,\varphi(q,p):(q,p)\in \RR^{2n}\}$ is Lagrangian submanifold of $\overline{\RR^{2n}}\times \RR^{2n}$. We identify $\overline{\RR^{2n}}\times \RR^{2n}$ with $T^*\RR^{2n}$ by the symplectomorphism
\begin{equation}\label{equation: twisting symplectomorphism tau}
\tau: \overline{\RR^{2n}}\times \RR^{2n} \to T^*\RR^{2n},\quad \tau(q,p,Q,P)=\left(\frac{q+Q}{2}, \frac{p+P}{2},p-P, Q-q\right).    
\end{equation}
Then $\Gamma_{\varphi}\coloneqq \tau (\operatorname{gr}(\varphi))$ is a Lagrangian submanifold of $T^*\RR^{2n}$.

For any compactly supported Hamiltonian function $H:[0,1]\times \RR^{2n}\rightarrow\RR$, and its Hamiltonian flow $\varphi^H_z$, 
one can associate a 1-parameter family of functions $F: [0,1]\times S^{2n}\times \RR^{N} \rightarrow \RR$ such that for each $z\in [0,1]$, $F_z=F(z,-)$ is a generating function quadratic at infinity (GFQI) of
$\Gamma_{\varphi^H_z}$. Here, GFQI means that outside of a compact set of $S^{2n}\times \RR^{N}$, $F_z$ equals a quadratic form on the fiber direction of index $\iota$. The generating function homology of $\varphi^H$ is defined as
\begin{equation}\label{definition: GF homology}
    G_{*}^{(a,b]} (\varphi^H;\FF)\coloneqq H_{*-\iota}(\{F_1\leq b\}, \{F_1\leq a\};\FF).
\end{equation}

In fact, $G_{*}^{(a,b]} (\varphi^H;\FF)$ is independent with choices of $F_z$.  Indeed, using Viterbo-Th\'eret uniqueness (c.f. \cite{Theret_Viterbo_uniqueness}), it is proven in \cite[Lemma 2.2-(iii)]{Fraser-Sandon-Zhang} that for two family of $F_z$, $G_z$, there are exist two quadratic forms $Q_0,Q_1$ and a 2-parameter family of diffeomorphisms $\Phi_{y,z}$ of $S^{2n}\times \mathbb{R}^M$ such that $F_z\oplus Q_0= (G_z\oplus Q_0)\circ \Phi_{1,z}$. Then the Thom isomorphisms induced by $Q_0,Q_1$ and the induced map of $\Phi_{1,z}$ on homology define an isomorphism $H(\{F_1\leq b\}, \{F_1\leq a\};\FF) \simeq H(\{G_1\leq b\}, \{G_1\leq a\};\FF)$ (with suitable grading). It is proven in \cite[Proposition 4.1]{Fraser-Sandon-Zhang}, the isomorphism only depends on $F_z$, $G_z$ independently with $(Q_0,Q_1,\Phi_{y,z})$.

Presently, we have a well-defined homology group $G_{*}^{(a,b]} $ for a compactly supported Hamiltonian flow $\varphi^H$. To define homology groups to domains, we shall construct continuation maps for two Hamiltonians. For two compactly supported Hamiltonian $H \geq K$, there exist two 1-parameter families of functions $F,G: [0,1]\times S^{2n}\times \RR^{N} \rightarrow \RR$ such that $F_z$ (resp. $G_z$) are GFQI for $\Gamma_{\varphi^H_z}$ (resp. $\Gamma_{\varphi^G_z}$), and $F_z\leq G_z$ for each $z\in [0,1]$. Therefore, the natural inclusions of sublevel sets $\{G_1\leq b\} \subset \{F_1\leq b\}$ induces a continuation morphism $G_{*}^{(a,b]} (\varphi^H;\FF) \rightarrow G_{*}^{(a,b]} (\varphi^K;\FF)$. Hence, for a cofinal sequence $H_\alpha$ supported in $U$, we define
\[
G_{*}^{(a,b]}( U ;\FF)
\coloneqq \varprojlim\nolimits_{\alpha} G_{*}^{(a,b]} (\varphi^{H_\alpha};\FF).
\]

Actually, similar to Remark \ref{remark: wrapping formula for kernel}, the limit is taken over a directed set of Hamiltonian functions supported in $U$, and can be computed by cofinal sequences. 

Using results from \cite[Section 5.2]{Guillermou-Viterbo}, we can prove:
\begin{theorem} \label{generating function comparison} Let $\FF$ be a field. For an open set $U\subseteq \RR^{2n}$ and $a<b$, we have
\[G_{*}^{(a,b]}( U ;\FF) \simeq \HH_{\sT}^*(\Sh_{U^c}(M;\sT),(a,b]).\]
\end{theorem}
\begin{proof}[Sketch of proof]
For a 1-parameter family of generating function $F: [0,1]\times  S^{2n}\times \RR^{N}$ of $\varphi^H$ and $p: S^{2n}\times \RR^{N} \rightarrow  S^{2n}$, we set $\cK_{F}\coloneqq {p}_!^\sT1_{\{F_1\leq t\}}$. Here we remind that $\cK_{F}$ only depends on $\varphi^H$ by the 1-parameter version of Viterbo-Th\'eret uniqueness theorem of GFQIs (cf. \cite[
Lemma 2.2]{Fraser-Sandon-Zhang}).

In \cite[Section 5.2]{Guillermou-Viterbo}, the authors constructed a $\sT$-linear quantization of the symplectomorphism $\tau$ (see \eqref{equation: twisting symplectomorphism tau}), which induces a $\sT$-linear equivalence between fully $\sT$-linear subcategories $C_1$ of $\sT(T^*(\RR^{n}\times \RR^{n}))$ and $C_2$ of $\sT(T^* S^{2n})$, such that $\cK(\varphi^{H})|_1\in C_1$ is mapped to $\cK_{F}\in C_2$. In particular, we have that $\cK(\id)$ is mapped to $1_{S^{2n}\times \RR_{\geq 0}}$. Consequently, we have that
\begin{align*}
    \begin{aligned}
        &\Ext^*(1_{\RR^{2n}\times [-b,-a )},\sHom_\sT(\cK(\varphi^{H})|_1,\cK(\id)))\\
  \simeq &\Ext^*(1_{S^{2n}\times [-b,-a )},\sHom_\sT(\cK_{F},1_{S^{2n}\times \RR_{\geq 0}}))  \\
  \simeq & H_{*-\iota}(\{F_1\leq b\}, \{F_1\leq a\};\FF)=G_{*}^{(a,b]} (\varphi^H;\FF).
    \end{aligned}
\end{align*}

We remind here that we compare a relative Borel-Moore homology and a relative singular homology in the second isomorphism. This is true because that $F_1$ is quadratic at infinity (in particular, $F_1$ satisfies the Palais-Smale condition). 

Therefore, in virtue of Equation \eqref{wrapping formula for C out}, we only need to compare the continuation maps on GF homology and sheaves cohomology. In both case, we reduce to $C^1$-small Hamiltonians (see Remark \ref{remark: GKS theorem} for sheaves and \cite[Proposition 5.3]{Traynor_SH-via-GF} for GF.) Then we are able to see that continuation maps on both sides are induced by closed inclusions of sublevel sets of certain functions defined over the diagonal the base $\RR^{2n}$ or $S^{2n}$ (that is vanish at $\infty$ in the $S^{2n}$ case). Moreover, upto the equivalence constructed \cite[Section 5.2]{Guillermou-Viterbo}(which is also given by a sublevel sets of a certain function), the closed inclusions on both sides are the same. Therefore, we match the continuation maps, and the conclusion follows.
\end{proof}
\begin{remark}We only use that $\mathbb{F}$ is a field in applying \eqref{wrapping formula for C out}, where we use the Mittag-Leffler condition to commute limit with cohomology.  Instead, we could work with generating function chains  $CG_{\bullet}^{(a,b]}(U;\mathbf{k})$; then the argument here works for arbitrary coefficients.
\end{remark}

\bibliographystyle{plain}
\bibliography{Notes_REF}

\end{document}